\documentclass[12pt,a4paper,pdftex]{amsart}
\usepackage[margin=2.6cm]{geometry}
\usepackage{amsmath,amssymb,amsthm}
\usepackage{tikz-cd}
\usepackage{hyperref}

\newtheorem{thm}{Theorem}[section]
\newtheorem{prop}[thm]{Proposition}
\newtheorem{lem}[thm]{Lemma}
\newtheorem{cor}[thm]{Corollary}
\theoremstyle{definition}
\newtheorem{defn}[thm]{Definition}

\theoremstyle{remark}
\newtheorem{rem}[thm]{Remark}

\newcommand*{\Z}{\mathbb{Z}}
\newcommand*{\R}{\mathbb{R}}
\newcommand*{\Coef}{\mathbb{K}}
\newcommand*{\finite}{\mathbf{f}}
\newcommand*{\aff}{\mathrm{aff}}
\newcommand*{\ext}{\mathrm{ext}}
\newcommand*{\triv}{\mathrm{triv}}
\newcommand*{\sgn}{\mathrm{sgn}}
\newcommand*{\bbopen}{\mathcal{O}_{\mathrm{bb}}}

\newcommand*{\BM}{\mathcal{BM}}
\newcommand*{\EB}{\mathcal{EB}}
\newcommand*{\ch}{\mathrm{ch}}
\newcommand*{\id}{\mathrm{id}}

\DeclareMathOperator{\grk}{grk}

\DeclareMathOperator{\Aut}{Aut}
\DeclareMathOperator{\Hom}{Hom}

\DeclareMathOperator{\End}{End}
\DeclareMathOperator{\Ker}{Ker}
\DeclareMathOperator{\Coker}{Coker}

\DeclareMathOperator{\supp}{supp}
\DeclareMathOperator{\Ima}{Im}
\DeclareMathOperator{\Stab}{Stab}
\DeclareMathOperator{\Sh}{Sh}

\numberwithin{equation}{section}

\title{Braden-MacPherson sheaves on alcoves}
\author{Noriyuki Abe}
\address{Graduate School of Mathematical Sciences, the University of Tokyo, 3-8-1 Komaba, Meguro-ku, Tokyo 153-8914, Japan.}
\email{abenori@ms.u-tokyo.ac.jp}
\subjclass[2020]{20G05,05E10}

\begin{document}
\begin{abstract}
We study Braden-MacPherson sheaves on the moment graph associated to the set of of alcoves.
We define an action of Soergel bimodules on the category of Braden-MacPherson sheaves.
We also prove a certain stability of morphisms between Braden-MacPherson sheaves.
\end{abstract}

\maketitle
\section{Introduction}\label{sec:Introduction}
Braden-MacPherson \cite{MR1871967} attached a moment graph to an algebraic variety with a nice torus action.
They also describe the torus equivariant intersection complexes on the variety via sheaves on this moment graph.
This is called the Braden-MacPherson algorithm and the sheaf given by this algorithm is called the Braden-MacPherson sheaf.
We note that the theory can apply to characteristic zero coefficients.
In the case of positive characteristic, this algorithm calculates an indecomposable parity sheaf \cite{MR3330913}.

Fiebig conducted a systematic study of Braden–MacPherson sheaves on moment graphs. Motivated by applications to representation theory, he further investigated sheaves on moment graphs arising from affine Weyl groups.
These moment graphs are associated with the affine flag varieties and are related to the representation theory of algebraic groups in positive characteristic.
As an application of his theory, he obtained lower bounds for the primes for which Lusztig conjecture holds~\cite{MR2726602,MR2999126}.

In this paper, we study moment graphs defined on the set of alcoves and sheaves on them.
This moment graph also appears in Lanini’s work~\cite{MR3324922}.
It corresponds geometrically to the semi-infinite flag variety and representation-theoretically to the representation theory of Lie algebras in positive characteristic.

This paper has two main objectives.

First, we define an action of Braden–MacPherson sheaves on the affine Weyl group.
The category of Braden–MacPherson sheaves on the affine Weyl group is equivalent to the category of Soergel bimodules, and hence carries a monoidal structure.
We define an action of this monoidal category.
Note that the action of the Soergel bimodule corresponding to a simple reflection is the same as wall-crossing functors in~\cite{MR2395170, MR3324922}.
This action gives a categorification of the periodic Hecke module.

Second, we consider homomorphisms between Braden–MacPherson sheaves and investigate a certain stability property.
Fixing a vertex $x$, we consider the subset $O_x = \{ y \mid y \ge x \}$.
For two Braden–MacPherson sheaves $\mathcal{F}_1$ and $\mathcal{F}_2$, we show that the space $\Hom(\mathcal{F}_1|_{O_x}, \mathcal{F}_2|_{O_x})$ becomes stable when $x$ is sufficiently small.
Technically, a major difficulty to study sheaves on alcoves arises from the fact that the underlying poset is not ideal finite; that is, for a vertex $x$, the set $\{ y \mid y \le x \}$ is not finite.
Since much of Fiebig’s theory is established for ideal finite moment graphs, it cannot be directly applied in this setting.
Although the set of all alcoves is not ideal finite, it is interval finite; that is, for any two vertices $x_1, x_2$, the set $\{ y \mid x_1 \le y \le x_2 \}$ is finite.
Therefore, $O_x$ is ideal finite, and Fiebig's theory can be applied. In this way, we regard the entire set of alcoves as a ``limit'' of the subsets $O_x$, and our main theorem may become effective in this perspective.
We note that, in the above, we consider only degree-zero homomorphisms.
We do not claim stability when degrees are allowed to vary.

We now explain the details of the contents of this paper.
We fix a coefficient ring $\widehat{\mathbb{K}}$ which is a complete Noetherian ring of mixed characteristic $(0,p)$ where $p$ is a prime number and let $\mathbb{K}$ be the residue field.
Since we need the GKM condition for the theory of moment graphs, that is satisfied only with characteristic $0$ ring, we work with $\widehat{\mathbb{K}}$ instead of $\mathbb{K}$.
Let $(\mathbb{X},\Phi,\mathbb{X}^{\vee},\Phi^{\vee})$ be a finite root datum.
Then we have affine hyperplanes in $\mathbb{X}^{\vee}_{\R} = \mathbb{X}^{\vee}\otimes_{\Z}\R$.
A connected component of the complement of the union on these hyperplanes is called an alcove.
Let $\mathcal{A}$ be the set of all alcoves.
We always assume that $p\ne 2$ and $p$ is not a torsion prime for this datum.

A moment graph is a graph such that the set of vertices is partially ordered, and each edge has a label which belongs to a certain abelian group.
We can define a structure of a moment graph on $\mathcal{A}$ as follows:
The set of vertices is $\mathcal{A}$.
The alcoves $A,A'$ are connected by an edge when $A'$ is the image of a reflection of $A$ across an affine hyperplane.
The label of the edge is the affine root which defines the hyperplane.
In our setting, the label is in $\widehat{\mathbb{X}}_{\aff}^{\vee} = \mathbb{X}\otimes_{\Z}\widehat{\mathbb{K}}\oplus \widehat{\mathbb{K}}\oplus \widehat{\mathbb{K}}$.
The affine Weyl group $W_{\aff} = W \rtimes \Z \Phi^{\vee}$ acts on this moment graph.
We upgrade this action on the level of sheaves.

Let $\BM(\mathcal{A})$ be the category of Braden-MacPherson sheaves on $\mathcal{A}$.
We note that $W_{\aff}$ is a Coxeter system, here the simple reflections are reflections with respect to the walls of the fundamental positive alcove $A_{0}^{+}$.
Let $\mathcal{S}$ be the category of Soergel bimodules.
We define an action of $\mathcal{S}$ on $\BM(\mathcal{A})$.
Namely for $\mathcal{F}\in \BM(\mathcal{A})$ and $B\in \mathcal{S}$, we construct $\mathcal{F}\star B$.
The construction is not so direct.
We first define an action on the level of global sections, namely we first define the space of global sections of $\mathcal{F}\star B$.
Using the localization functor \cite{MR2370278}, we get the sheaf $\mathcal{F}\star B$.

Let $\widehat{S}^{\vee}_{\aff}$ be a symmetric algebra of $\widehat{\mathbb{X}}_{\aff}$.
To say the construction more precisely, we recall that a sheaf $\mathcal{F}$ on the moment graph $\mathcal{A}$ is a collection of graded $\widehat{S}^{\vee}_{\aff}$-modules $((\mathcal{F}^{A})_{A},(\mathcal{F}^{E})_{E})$ where $A$ (resp.\ $E$) run through vertices (resp.\ edges) in $\mathcal{A}$ and maps $(\rho^{\mathcal{F}}_{E,A} \colon \mathcal{F}^{A}\to \mathcal{F}^{E})$ for all edges $E$ with an endpoint $A$.
The space of global sections $\Gamma(\mathcal{F})$ is defined by the set of $(a_{A})\in \prod_{A}\mathcal{F}^{A}$ such that $\rho_{E,A}^{\mathcal{F}}(a_{A}) = \rho_{E,A'}^{\mathcal{F}}(a_{A'})$ for any edge $E$ connecting $A$ with $A'$.
This is an $\widehat{S}^{\vee}_{\aff}$-module.
In fact we can also define a $\widehat{S}^{\vee}_{\aff}$-bimodule structure as follows.
The group $W_{\aff}$ acts on $\mathcal{A}$, hence we have a map $W_{\aff}\to \mathcal{A}$ defined by $w\mapsto w(A_{0}^{+})$.
This is a bijection and let $A\mapsto w_{A}$ be the inverse of this map.
Then the right action of $f\in \widehat{S}_{\aff}^{\vee}$ on $(a_{A})\in \Gamma(\mathcal{F})$ is defined by $(a_{A})f = (w_{A}(f)a_{A})$.
Moreover, $\Gamma(\mathcal{F})$ is an object in the following category $\EB(\mathcal{A})$. (To simplify the explanation, we give a slightly different definition from the one used in the main text.)
Let $Q^{\vee}_{\aff}$ be the field of fractions of $\widehat{S}^{\vee}_{\aff}$.
Then an object of $\mathcal{EB}(\mathcal{A})$ is $M = (M,(M_{Q}^{A})_{A\in \mathcal{A}})$ where $M$ is a graded $\widehat{S}_{\aff}^{\vee}$-bimodule and $M_{Q}^{A}$ is a $Q^{\vee}_{\aff}$-bimodule such that $m f = w_{A}(f)m$ for $m \in M_{Q}^{A}$ and $f\in Q_{\aff}^{\vee}$ and $M\subset \prod_{A}M_{Q}^{A}$.
(We will also impose some more conditions. We ignore them in the introduction.)
A morphism $(M,(M_{Q}^{A}))\to (N,(N_{Q}^{A}))$ is an $\widehat{S}_{\aff}^{\vee}$-bimodule homomorphism $M\to N$ of degree zero and a collection of $Q_{A}^{\vee}$-bimodule homomorphisms $(M_{Q}^{A}\to N_{Q}^{A})_{A}$ which are compatible with the embeddings $M\subset \prod_{A}M_{Q}^{A},N\subset \prod_{A}N_{Q}^{A}$.
It is easy to see that if $\mathcal{F}\in \BM(\mathcal{A})$, then $\Gamma(\mathcal{F})\in \EB(\mathcal{A})$.
On the other hand, we have a localization functor $\mathcal{L}$ from $\EB(\mathcal{A})$ to the category of sheaves on $\mathcal{A}$~\cite{MR2370278}.
Note that Fiebig regarded $\Gamma(\mathcal{F})$ as a module over a structure sheaf and define $\mathcal{L}$ on such modules.
His construction works also with our setting.

Let $M\in \EB(\mathcal{A})$ and $B \in \mathcal{S}$.
Here, for a realization of $\mathcal{S}$, we use the one in \cite{MR4321542}.
In particular, $B\in \EB(W_{\aff})$ where $\EB(W_{\aff})$ is defined as a similar way.
We define $M\otimes B\in \mathcal{E}(\mathcal{A})$ by $M\otimes B = M\otimes_{\widehat{S}_{\aff}^{\vee}}B$ as an $\widehat{S}_{\aff}^{\vee}$-bimodule and $(M\otimes B)_{Q}^{A} = \bigoplus_{x\in W_{\aff}}M_{Q}^{Ax^{-1}}\otimes B_{Q}^{x}$.
Now we define $\mathcal{F}\star B = \mathcal{L}(\Gamma(\mathcal{F})\otimes B)$.

\begin{thm}[Theorem~\ref{thm:action of Soergel bimodule preserves BM sheaves}]
If $\mathcal{F}$ is a Braden-MacPherson sheaf on $\mathcal{A}$, then $\mathcal{F}\star B$ is also a Braden-MacPherson sheaf.
The operation $(\mathcal{F},B)\mapsto \mathcal{F}\star B$ defines an action of $\mathcal{S}$ on $\BM(\mathcal{A})$.
\end{thm}

Recall that the split Grothendieck group of $\mathcal{S}$ is isomorphic to the Hecke algebra $\mathcal{H}$ attached to $W_{\aff}$.
We prove that there exists an $\mathcal{H}$-module homomorphism from the split Grothendieck group of $\BM(\mathcal{A})$ to the completed periodic module (Theorem~\ref{thm:BM sheaves and periodic module}).

Next result is a stability of homomorphisms.
Let $\mathcal{F},\mathcal{G}\in \BM(\mathcal{A})$ be indecomposable sheaves.
A morphism $\varphi\colon \mathcal{F} \to \mathcal{G}$ is a collection of degree zero $\widehat{S}_{\aff}^{\vee}$-module homomorphisms $\varphi_{A}\colon \mathcal{F}^{A}\to \mathcal{G}^{A}$ for each $A\in \mathcal{A}$ and $\varphi_{E}\colon \mathcal{F}^{E}\to \mathcal{G}^{E}$ for each edge $E$ which are compatible with the homomorphisms $\rho^{\mathcal{F}}_{E,A},\rho^{\mathcal{G}}_{E,A}$ for all edges $E$ with an endpoint $A$.
Let $\Hom(\mathcal{F},\mathcal{G})$ be the space of morphisms from $\mathcal{F}$ to $\mathcal{G}$.
For $A\in \mathcal{A}$, set $O_{\ge A} = \{A'\in \mathcal{A}\mid A'\ge A\}$.
Let $\mathcal{F}|_{O_{\ge A}}$ be the restriction of $\mathcal{F}$ to $O_{\ge A}$,
This is defined by ignoring all vertices and edges outside $O_{\ge A}$.

\begin{thm}[Theorem~\ref{thm:finiteness}]
The restriction map $\Hom(\mathcal{F},\mathcal{G})\to \Hom(\mathcal{F}|_{O_{\ge A}},\mathcal{G}|_{O_{\ge A}})$ is an isomorphism if $A$ is sufficiently small.
\end{thm}

\subsection*{Acknowledgment}
The author was supported by JSPS KAKENHI Grant Number 23H01065.

\section{Moment graphs of alcoves}

\subsection{Notation}
Let $M = \bigoplus_{i\in\Z}M^i$ be a graded object, e.g., a graded left module over a graded algebra, a graded vector space, etc.
Then the grading shift $M(n)$ is defined as $M(n)^i = M^{i + n}$.
Let $R$ be a commutative graded non-zero ring and $M$ a left or right graded $R$-module.
We say $M$ is graded free if it is isomorphic to $\bigoplus_i R(n_i)$ for some $n_1,\ldots,n_r\in \Z$.
(Hence, ``graded free'' means graded free of finite rank.)
In this case we put $\grk(M) = \sum v^{n_i}\in \Z[v,v^{-1}]$ and we call $\grk(M)$ the graded rank of $M$, here $v$ is an indeterminate.

Let $\widehat{\Coef}$ be a complete local noetherian ring of mixed characteristic $(0,p)$ with $p > 0$ and $\Coef$ the residue field of $\widehat{\Coef}$.
We always assume that $p\ne 2$.

Let $A$ be a commutative ring, $B$ a commutative $A$-algebra and $\mathcal{C}$ a category.
If $\mathcal{C}$ is an $A$-category, then we define a new category $B\otimes_{A}\mathcal{C}$ as follows.
Objects in $B\otimes_{A}\mathcal{C}$ are the same as those of $\mathcal{C}$ and we put $\Hom_{B\otimes_{A}\mathcal{C}}(X,Y) = B\otimes_{A}\Hom_{\mathcal{C}}(X,Y)$.
If it is natural to think $\Hom$ as a right $A$-module, we write $\mathcal{C}\otimes_{A}B$.
Assume that $A$ and $B$ are graded, $\mathcal{C}$ is graded and $\bigoplus_{n\in\Z}\Hom_{\mathcal{C}}(X,Y(n))$ is a graded $A$-module for any objects $X,Y$ of $\mathcal{C}$.
Then we define $\Hom_{B\otimes_{A}\mathcal{C}}(X,Y)$ be the $0$-degree part of $B\otimes_{A}\bigoplus_{n\in\Z}\Hom(X,Y(n))$.
We also use a notation $\mathcal{C}\otimes_{A}B$.

\subsection{Moment graphs and sheaves}
We introduce notation about moment graphs.
See \cite{MR2370278} for the details.
Let $\mathbb{Y}$ be a free $\widehat{\Coef}$-module of finite rank, $S(\mathbb{Y})$ the symmetric algebra.
The algebra $S(\mathbb{Y})$ is a graded $\widehat{\Coef}$-algebra via $\deg(\mathbb{Y}) = 2$.
Let $\mathcal{V}$ be a $\mathbb{Y}$-labeled moment graph~\cite[2.1]{MR2370278}.
We often write the set of vertices of $\mathcal{V}$ by the same letter $\mathcal{V}$.
For each edge $E$ in $\mathcal{V}$, let $\alpha_{E}\in \mathbb{Y}$ be the label on $E$.
A sheaf $\mathcal{F}$ on $\mathcal{V}$ is $\mathcal{F} = ((\mathcal{F}^{x}),(\mathcal{F}^{E}),(\rho_{E,x}^{\mathcal{F}}))$ where $\mathcal{F}^{x}$ is a graded $S(\mathbb{Y})$-module for each vertex $x$ of $\mathcal{V}$, $\mathcal{F}^{E}$ is a graded $S(\mathbb{Y})/\alpha_{E}S(\mathbb{Y})$-module for each edge $E$ of $\mathcal{V}$ and $\rho_{E,x}^{\mathcal{F}}\colon \mathcal{F}^{x}\to \mathcal{F}^{E}$ is an $S(\mathbb{Y})$-module homomorphism of degree zero where an edge $E$ connects $x$ with another vertex.
For sheaves $\mathcal{F},\mathcal{G}$ on $\mathcal{V}$, a morphism $\varphi\colon \mathcal{F}\to \mathcal{G}$ is a collection of $S(\mathbb{Y})$-module homomorphisms $\varphi^{x}\colon \mathcal{M}^{x}\to \mathcal{N}^{x}$ for each vertex $x$ and $\varphi^{E}\colon \mathcal{F}^{E}\to \mathcal{G}^{E}$ for each edge $E$ such that $\rho_{E,x}^{\mathcal{G}}\circ f^{x} = f^{E}\circ\rho_{E,x}^{\mathcal{F}}$ if $x$ is an endpoint of $E$.
We put $\supp \mathcal{F} = \{x\in \mathcal{V}\mid \mathcal{F}^{x}\ne 0\}$ and call it the support of $\mathcal{F}$.
A subset $X\subset \mathcal{V}$ is called bounded above (resp.\ bounded below) if there exists $x$ such that $X\subset \{y\in \mathcal{V}\mid y\le x\}$ (resp.\ $X\subset \{y\in \mathcal{V}\mid y\ge x\}$).
In this paper, for a sheaf on a moment graph, we always assume that $\supp \mathcal{F}$ is bounded above.
Let $\Sh(\mathcal{V})$ be the category of sheaves on $\mathcal{V}$.
For a vertex $x$ of $\mathcal{V}$, set $\mathcal{F}^{[x]} = \Ker(\mathcal{F}^{x}\to \bigoplus_{E}\mathcal{F}^{E})$ where $E$ runs through all edges of $\mathcal{V}$ connecting $x$ and a certain $y > x$.
The space of global sections $\Gamma(\mathcal{F})$ is the set of $(a_{x})\in \prod_{x\in \mathcal{V}}\mathcal{F}^{x}$ such that $\rho_{E,x}^{\mathcal{F}}(a_{x}) = \rho_{E,y}^{\mathcal{F}}(a_{y})$ for all $E,x,y$ such that $E$ connects $x$ with $y$.

Let $X$ be a subset of the set of vertices of $\mathcal{V}$. (We simply write $X\subset \mathcal{V}$ and say that $X$ is a subset of $\mathcal{V}$.)
Then $X$ has a induced structure of a moment graph: an edge $E$ of $X$ is an edge of $\mathcal{V}$ such that both endpoints are in $X$.
For a sheaf $\mathcal{F}$ on $\mathcal{V}$, by ignoring vertices and edges outside $X$, we have a sheaf $\mathcal{F}|_{X}$ on $\mathcal{V}$.
The sheaf $\mathcal{F}|_{X}$ is called the restriction of $\mathcal{F}$ to $X$.
We put $\Gamma(X,\mathcal{F}) = \Gamma(\mathcal{F}|_{X})$.
For each $a = (a_{x})_{x\in \mathcal{V}}\in \Gamma(\mathcal{F})$, we put $a|_{X} = (a_{x})_{x\in X}\in \Gamma(X,\mathcal{F})$.

We define a topology on $\mathcal{V}$ by declaring that $X\subset \mathcal{V}$ is open if $x\in X$, $y\in \mathcal{V}$, $y\ge x$ implies $y\in X$.
Let $\bbopen(\mathcal{V})$ be the set of bounded below open subsets.
We have $\bigcup_{O\in\bbopen(\mathcal{V})}O = \mathcal{V}$.

Assume that for any $x,y\in \mathcal{V}$ the set $\{z\in \mathcal{V}\mid x\le z\le y\}$ is finite.
We also assume that for any $x,y\in \mathcal{V}$ there exist $z,z'\in \mathcal{V}$ such that $z\le x,y\le z'$.
In particular, for any sheaf $\mathcal{F}$ on $\mathcal{V}$ and $O\in \bbopen$, $O\cap \supp \mathcal{F}$ is finite.
\begin{defn}
A sheaf $\mathcal{F}$ on $\mathcal{V}$ is called a Braden-MacPherson sheaf if it satisfies the following.
\begin{enumerate}
\item[(BM1)] For any $x\in \mathcal{V}$, $\mathcal{F}^{x}$ is a graded free $S(\mathbb{Y})$-module.
\item[(BM2)] Let $E$ be an edge connecting $x$ with $y < x$. Then $\rho_{E,x}^{\mathcal{F}}$ is surjective and the kernel is $\mathcal{F}^{x}\alpha_{E}$.
\item[(BM3)] $\mathcal{F}$ is flabby, namely $\Gamma(\mathcal{F})\to \Gamma(U,\mathcal{F})$ is surjective for any open subset $U\subset \mathcal{V}$.
\item[(BM4)] For any $x\in \mathcal{F}$, $\Gamma(\mathcal{F})\to \mathcal{F}^{x}$ is surjective.
\end{enumerate}
\end{defn}

\begin{lem}\label{lem:equivalence on flabby}
Let $\mathcal{F}$ be a sheaf on $\mathcal{V}$.
Then $\mathcal{F}$ is flabby if and only if $\mathcal{F}|_{O}$ is flabby for any $O\in \bbopen(\mathcal{V})$.
\end{lem}
\begin{proof}
For any $O\in \bbopen(\mathcal{V})$ and an open subset $U\subset O$, we have a sequence $\Gamma(\mathcal{F})\to \Gamma(O,\mathcal{F})\to \Gamma(U,\mathcal{F})$.
If $\mathcal{F}$ is flabby, then the composition is surjective.
Hence $\Gamma(O,\mathcal{F})\to \Gamma(U,\mathcal{F})$ is surjective.
Therefore $\mathcal{F}|_{O}$ is flabby.

Assume that $\mathcal{F}|_{O}$ is flabby for any $O\in\bbopen(\mathcal{V})$ and prove that $\mathcal{F}$ is flabby.
Let $U\subset \mathcal{V}$ be an open subset and $f\in \Gamma(U,\mathcal{F})$.
We extend $f$ to $\mathcal{V}$.
Consider a pair $(g,V)$ such that $V\subset \mathcal{V}$ is an open subset containing $U$ and $g\in \Gamma(V,\mathcal{F})$ such that $g|_{U} = f$.
We write $(g_{1},V_{1})\ge (g_{2},V_{2})$ if $V_{1}\supset V_{2}$ and $g_{1}|_{V_{2}} = g_{2}$.
By Zorn's lemma, there exists a maximal pair $(g,V)$.
We prove $V = \mathcal{V}$.
Assume that $V\ne \mathcal{V}$ and take $x\in \mathcal{V}\setminus V$.
Set $O = \{y\in \mathcal{V}\mid y\ge x\}$.
Since $\mathcal{F}|_{O}$ is flabby, there exists $g' = (g'_{x})_{x\in O}\in \Gamma(O,\mathcal{F})$ such that $g'|_{V\cap O} = g|_{V\cap O}$.
For each $y\in O\cup V$, we define $h_{y}\in \mathcal{F}^{y}$ as follows.
If $y\in V$, then $h_{y} = g_{y}$, otherwise $h_{y} = g'_{y}$.
Let $E$ be an edge connecting $y$ with $y' > y$.
If $y\in V$, then, since $V$ is open, $y'\in V$.
Hence $E$ is an edge in $V$.
Since $g$ is a section in $V$, we have $\rho^{\mathcal{F}}_{E,y}(h_{y}) = \rho^{\mathcal{F}}_{E,y'}(h_{y'})$ since $g_{y} = h_{y},g_{y'} = h_{y'}$.
Similarly, if $y\in O$, then $E$ is an edge in $O$, and hence $\rho^{\mathcal{F}}_{E,y}(h_{y}) = \rho^{\mathcal{F}}_{E,y'}(h_{y'})$.
Therefore $h\in \Gamma(V\cup O,\mathcal{F})$.
This contradicts the maximality of $(g,V)$.
Therefore $V = \mathcal{V}$.
Hence $f$ extends to $\mathcal{V}$.
\end{proof}

\begin{lem}\label{lem:BM is a limt of BM}
$\mathcal{F}\in \Sh(\mathcal{V})$ is a Braden-MacPherson sheaf if and only if $\mathcal{F}|_{O}$ is a Braden-MacPherson sheaf for any $O\in \bbopen(\mathcal{V})$.
\end{lem}
\begin{proof}
The sheaf $\mathcal{F}$ satisfies (BM1) (resp.\ (BM2), (BM3)) if and only if $\mathcal{F}|_{O}$ satisfies (BM1) (resp.\ (BM2), (BM3)) for any $O\in \bbopen(\mathcal{V})$.
For (BM1) and (BM2), this is easy and (BM3) follows from Lemma~\ref{lem:equivalence on flabby}.
Assume that $\mathcal{F}$ satisfies (BM4).
Let $O\in \bbopen(\mathcal{V})$ and $x\in O$.
Then we have $\Gamma(\mathcal{F})\to \Gamma(O,\mathcal{F})\to \mathcal{F}^{x}$.
The composition is surjective, hence $\Gamma(O,\mathcal{F})\to \mathcal{F}^{x}$ is also surjective.
Therefore $\mathcal{F}|_{O}$ satisfies (BM4).

On the other hand, assume that $\mathcal{F}|_{O}$ is a Braden-MacPherson sheaf for any $O\in \bbopen(\mathcal{V})$.
Let $x\in \mathcal{V}$ and take $O\in\bbopen(\mathcal{V})$ such that $x\in O$.
Then we have $\Gamma(\mathcal{F})\to \Gamma(\mathcal{F}|_{O}) \to \mathcal{F}^{x}$.
The first map is surjective since $\mathcal{F}$ is flabby and the second is surjective since $\mathcal{F}|_{O}$ satisfies (BM4).
We get (BM4).
\end{proof}

\begin{prop}\label{prop:classification of BM sheaves}
For each vertex $x$, there exists an indecomposable Braden-MacPherson sheaf $\mathcal{B}^{\mathcal{V}}(x)$ such that $\supp\mathcal{B}^{\mathcal{V}}(x)\subset \{y\in \mathcal{V}\mid y\le x\}$ and $\mathcal{B}^{\mathcal{V}}(x)^{x}\simeq S(\mathbb{Y})$.
Moreover, this sheaf is unique up to isomorphism.
These sheaves satisfy the following.
\begin{enumerate}
\item For a sheaf $\mathcal{F}\in \Sh(\mathcal{V})$ satisfying (BM3), (BM4) and an open subset $U\subset \mathcal{V}$, the map $\Hom_{\Sh(\mathcal{V})}(\mathcal{B}^{\mathcal{V}}(x),\mathcal{F}) \to \Hom_{\Sh(U)}(\mathcal{B}^{\mathcal{V}}(x)|_{U},\mathcal{F}|_{U})$ is surjective.
\item Let $\mathcal{F}\in \Sh(\mathcal{V})$ be a Braden-MacPherson sheaf. Assume that $x$ is a maximal element in $\supp \mathcal{F}$.
Then $\mathcal{B}^{\mathcal{V}}(x)(n)$ is a direct summand of $\mathcal{F}$ for some $n\in \Z$.
\item If $\{y\in \mathcal{V}\mid y\le x\}$ is finite for all $x\in \mathcal{V}$, then any Braden-MacPherson sheaf is a direct sum of $\mathcal{B}^{\mathcal{V}}(x)(n)$ where $x \in \mathcal{V}$ and $n\in\Z$.
\end{enumerate}
\end{prop}
\begin{proof}
Recall the construction in the proof of \cite[Theorem~6.4]{MR3330913} (Braden-MacPherson algorithm).
Let $y\in \mathcal{V}$ and let $\mathcal{E}_{\delta y}$ be the set of edges $E$ which connects $y$ with $z > y$ and $\mathcal{V}_{\delta y}$ the set of endpoints of $E\in \mathcal{E}_{\delta y}$ which is not $y$.
For a sheaf $\mathcal{F}$, we define $\mathcal{F}^{\delta y}$ as the image of
\[
\Gamma(\{z\in\mathcal{V} \mid z > y\},\mathcal{F})\subset \bigoplus_{z > y}\mathcal{F}^{z}\to \bigoplus_{z\in \mathcal{V}_{\delta y}}\mathcal{F}^{z}\xrightarrow{(\rho^{\mathcal{F}}_{E,z})}\bigoplus_{E\in \mathcal{E}_{\delta y}}\mathcal{F}^{E}.
\]
Then $\mathcal{B}^{\mathcal{V}}(x)$ is constructed inductively as follows:
\begin{enumerate}\renewcommand{\labelenumi}{(\alph{enumi})}
\item $\mathcal{B}^{\mathcal{V}}(x)^{x} = S(\mathbb{Y})$.
\item $\mathcal{B}^{\mathcal{V}}(x)^{E} = \mathcal{B}^{\mathcal{V}}(x)^{z} / \mathcal{B}^{\mathcal{V}}(x)^{z}\alpha_{E}$ where $E$ is an edge connecting $z$ with $y < z$.
\item $\mathcal{B}^{\mathcal{V}}(x)^{y}$ is a projective cover of $\mathcal{B}^{\mathcal{V}}(x)^{\delta y}$.
\end{enumerate}
By \cite[Theorem~6.4]{MR3330913}, $\mathcal{B}^{\mathcal{V}}(x)|_{O}$ is a Braden-MacPherson sheaf for any $O\in\bbopen(\mathcal{V})$.
By Lemma~\ref{lem:BM is a limt of BM}, $\mathcal{B}^{\mathcal{V}}(x)$ is a Braden-MacPherson sheaf.
It is easy to see that $\supp \mathcal{B}^{\mathcal{V}}(x)\subset\{y\in \mathcal{V}\mid y\le x\}$.

We remark that, for an edge $E$ connecting $y$ with $z > y$, if $\mathcal{B}^{\mathcal{V}}(x)^{E} \ne 0$, then $y,z\le x$.
Indeed, by (b), we have $\mathcal{B}^{\mathcal{V}}(x)^{z} \ne 0$, hence $z\le x$.
This implies $y\le x$.

Assume that $\mathcal{B}^{\mathcal{V}}(x)$ is decomposed as $\mathcal{F} \oplus \mathcal{G}$.
We may assume $\mathcal{F}^{x} \ne 0$.
Then for any $O\in\bbopen(\mathcal{V})$, we have $\mathcal{B}^{\mathcal{V}}(x)|_{O} = \mathcal{F}|_{O} \oplus \mathcal{G}|_{O}$.
The left hand side is indecomposable by \cite[Theorem~6.4]{MR3330913}.
Hence $\mathcal{G}|_{O} = 0$ for any $O \in \bbopen(\mathcal{V})$.
This implies $\mathcal{G} = 0$.
Hence $\mathcal{B}^{\mathcal{V}}(x)$ is indecomposable.

We prove (1).
Let $f\colon \mathcal{B}^{\mathcal{V}}(x)|_{U}\to \mathcal{F}|_{U}$.
First we prove that $f$ can be extended to $U\cup \{y\in \mathcal{V}\mid y\not\le x\}$ simply defining as $f^{y} = 0$ for $y\not\le x$ and $f^{E} = 0$ if one of endpoints of $E$ is not less than or equal to $x$.
(Note that $y\not\le x$ implies $y\notin \supp \mathcal{B}^{\mathcal{V}}(x)$.)
Let $y\in \mathcal{V}$ such that $y\not\le x$ and $E$ be an edge connecting $y$ with another vertex.
Then by the construction of $\mathcal{B}^{\mathcal{V}}(x)$, $\mathcal{B}^{\mathcal{V}}(x)^{E} = 0$.
Hence, the compatibility with $\rho_{E,x}^{\mathcal{B}^{\mathcal{V}}(x)^{E}}$ and $\rho_{E,x}^{\mathcal{F}}$ obviously holds.
Therefore, we may assume $(\mathcal{V}\setminus\supp\mathcal{B}^{\mathcal{V}}(x))\subset U$.

As in the proof of Lemma~\ref{lem:equivalence on flabby}, we can take a maximal pair $(V,g)$ where $g\colon \mathcal{B}^{\mathcal{V}}(x)|_{V}\to \mathcal{F}|_{V}$ extends $f$.
Assume $V\ne \mathcal{V}$.
Since $\mathcal{V}\setminus V\subset \mathcal{V}\setminus U\subset \supp\mathcal{B}^{\mathcal{V}}(x) \subset \{y\mid y\le x\}$, we can take a maximal element $y\in \mathcal{V}\setminus V$.
Set $V' = V\cup \{y\}$.
Then this is open.
We prove that $g$ extends to $V'$ to get a contradiction.
The map $g$ on $V$ induces $\mathcal{B}^{\mathcal{V}}(x)^{\delta y}\to \mathcal{F}^{\delta y}$.
Since $\mathcal{B}^{\mathcal{V}}(y)\to \mathcal{B}^{\mathcal{V}}(x)^{\delta y}$ is a projective cover, to extend $g$ to $V'$, it is sufficient to prove that $\mathcal{F}^{y}\to \mathcal{F}^{\delta y}$ is surjective.
This is \cite[Lemma~6.1, Lemma~6.2]{MR3330913}.

The proof of (2) is similar.
Since $\mathcal{F}^{x}$ is free, there exists $n\in \Z$ such that $\mathcal{B}^{\mathcal{V}}(x)^{x}(n) = S(\mathbb{Y})(n)$ is a direct summand of $\mathcal{F}^{x}$.
We may assume $n = 0$ and we prove that $\mathcal{B}^{\mathcal{V}}(x)$ is a direct summand of $\mathcal{F}$.
Let $(V,i,p)$ be the maximal triple where $V$ is an open subset, $i\colon \mathcal{B}^{\mathcal{V}}(x)|_{V}\to \mathcal{F}|_{V}$ and $p\colon\mathcal{F}|_{V}\to \mathcal{B}^{\mathcal{V}}(x)|_{V}$ are homomorphism such that $p\circ i = \id$.
Again we may assume $\{y\in \mathcal{V}\mid y\not\le x\}\subset V$ and take a maximal element $y\in \mathcal{V}\setminus V$.
Let $E$ be an edge connecting $y$ with $z > y$.
Then, $\mathcal{B}^{\mathcal{V}}(x)^{E} \simeq \mathcal{B}^{\mathcal{V}}(x)^{z}/ \mathcal{B}^{\mathcal{V}}(x)^{E}\alpha_{E} $ and $\mathcal{F}^{E} \simeq \mathcal{F}^{z}/ \mathcal{F}^{E}\alpha_{E} $ by (BM2).
Hence, $i^{z},p^{z}$ induces $i^{E}\colon \mathcal{B}^{\mathcal{V}}(x)^{E}\to \mathcal{F}^{E}$ and $p^{E}\colon \mathcal{F}^{E}\to \mathcal{B}^{\mathcal{V}}(x)^{E}$.
They induce $\mathcal{B}^{\mathcal{V}}(x)^{\delta y}\to \mathcal{F}^{\delta y}\to \mathcal{B}^{\mathcal{V}}(x)^{\delta y}$.
We take a lift $i^{y} \colon \mathcal{B}^{\mathcal{V}}(x)^{y}\to \mathcal{F}^{y}$ and $p^{y}\colon \mathcal{F}^{y}\to \mathcal{B}^{\mathcal{V}}(x)^{y}$.
Then $p^{y}\circ i^{y}$ induces identity $\mathcal{B}^{\mathcal{B}}(x)^{\delta x}$.
By the property of the projective cover, $p^{y}\circ i^{y}$ is an isomorphism.
By replacing $p^{y}$ with $(p^{y}\circ i^{y})^{-1}\circ p^{y}$, we can take $p^{y}\circ i^{y} = \id$.
This gives an extension of $p,i$.

The uniqueness of $\mathcal{B}^{\mathcal{V}}(x)$ and (3) follows from (2).
\end{proof}

Let $\BM(\mathcal{V})$ be the category of Braden-MacPherson sheaves on $\mathcal{V}$.
We slightly extend (1) of Proposition~\ref{prop:classification of BM sheaves}.
\begin{lem}\label{lem:surjective between hom}
Let $U\subset \mathcal{V}$ be an open subset, $\mathcal{F}\in \BM(\mathcal{V})$ and $\mathcal{G}\in\Sh(\mathcal{V})$ satisfying (BM3) and (BM4).
Then the map $\Hom(\mathcal{F},\mathcal{G})\to \Hom(\mathcal{F}|_{U},\mathcal{G}|_{U})$ is surjective.
\end{lem}
\begin{proof}
When $\mathcal{V}$ is finite, this follows from Proposition~\ref{prop:classification of BM sheaves}.
In general, one can apply the argument in the proof of Lemma~\ref{lem:equivalence on flabby}.
\end{proof}

We put 
\[
\Hom^{\bullet}(\mathcal{F},\mathcal{G}) = \bigoplus_{n\in\Z}\Hom(\mathcal{F},\mathcal{G}(n)).
\]
This is an $S(\mathbb{Y})$-module.
Recall that a sheaf $\mathcal{G}$ on $\mathcal{V}$ is said to have a Verma flag if $\mathcal{G}^{[x]}$ is a graded free $S(\mathbb{Y})$-module for any vertex $x$ of $\mathcal{V}$.
\begin{lem}\label{lem:graded rank of hom of BM sheaves}
If $\{y\in \mathcal{V}\mid y\le x\}$ is finite for all $x\in \mathcal{V}$, $\mathcal{F}\in \BM(\mathcal{V})$ and $\mathcal{G}$ has a Verma flag and satisfies (BM3) and (BM4), then $\Hom^{\bullet}(\mathcal{F},\mathcal{G})$ is graded free and we have
\[
\grk\Hom^{\bullet}(\mathcal{F},\mathcal{G}) = \sum_{x\in \mathcal{V}}\overline{\grk(\mathcal{F})^{x}}\grk\mathcal{G}^{[x]}.
\]
\end{lem}
Here, for $a = \sum_{i}a_{i}v^{i}\in \Z[v,v^{-1}]$, we write $\overline{a} = \sum_{i}a_{i}v^{-i}$.
\begin{proof}
First note that a morphism $f\colon \mathcal{F}\to \mathcal{G}$ is determined by $(f^{x})_{x\in \mathcal{V}}$ since $\mathcal{F}$ satisfies (BM2).
By replacing $\mathcal{V}$ with $\mathcal{V}\cap (\supp \mathcal{F}\cup \supp\mathcal{G})$, we may assume that $\mathcal{V}$ is finite.
Let $\mathcal{V} = \{x_{1},\ldots,x_{n}\}$ such that $O_{i} = \{x_{i},\ldots,x_{n}\}$ is open for $i = 1,\ldots,n$.
Let $\Phi_{i}\colon \Hom_{\Sh(O_{i})}^{\bullet}(\mathcal{F}|_{O_{i}},\mathcal{G}|_{O_{i}})\to \Hom_{\Sh(O_{i + 1})}^{\bullet}(\mathcal{F}|_{O_{i + 1}},\mathcal{G}|_{O_{i + 1}})$ be the restriction map.
By Lemma~\ref{lem:surjective between hom}, $\Phi_{i}$ is surjective.
By $f\mapsto f^{x_{i}}$, we have an embedding $\Ker\Phi_{i}\hookrightarrow \Hom^{\bullet}(\mathcal{F}^{x_{i}},\mathcal{G}^{x_{i}})$.
We prove that the image is $\Hom^{\bullet}(\mathcal{F}^{x_{i}},\mathcal{G}^{[x_{i}]})$.
Let $f\in \Ker\Phi_{i}$.
Then $f^{x_{j}} = 0$ for $j > i$.
Since $f$ is a morphism between sheaves, for $j > i$, if $x_{j}$ is connected to $x_{i}$ by an edge $E$, then the image of $f^{x_{i}}$ is contained in $\Ker\rho^{\mathcal{G}}_{E,x_{i}}$.
Since $O_{i}$ is open, any $y  > x_{i}$ which is connected to $x_{i}$ is in $\{x_{i + 1},\ldots,x_{n}\}$.
Hence $\Ima(f^{x_{i}})\subset \mathcal{G}^{[x_{i}]}$.
Conversely, if $\Ima(f^{x_{i}})\subset \mathcal{G}^{[x_{i}]}$, then $f\in \Ker\Phi_{i}$.
Therefore $\Ker\Phi_{i}\simeq \Hom^{\bullet}(\mathcal{F}^{x_{i}},\mathcal{G}^{[x_{i}]})$.
By the assumption, this is graded free.
Therefore, by induction on $n$, $\Hom^{\bullet}(\mathcal{F},\mathcal{G})$ is graded free and the graded rank is given as in the lemma.
\end{proof}

We say that $\mathcal{V}$ satisfies GKM condition if for any edges $E,E'$ connecting $x$ with $y\ne y'$, respectively, $\alpha_{E}$ and $\alpha_{E'}$ are linearly independent.
The following easy lemma will be used often.
\begin{lem}\label{lem:GKM condition}
Assume that $\mathcal{V}$ satisfies GKM condition and $\mathcal{F}$ satisfies (BM1) and (BM2).
Let $E$ be an edge connecting $x$ with $y$ and $E'$ an edge connecting $x$ with $y'\ne y$.
If $m\in \mathcal{F}^{x}$ and $n\in \mathcal{F}^{y}$ satisfies $\rho_{E,x}(m\alpha_{E'}) = \rho_{E,y}(n\alpha_{E'})$, then $\rho_{E,x}(m) = \rho_{E,y}(n)$
\end{lem}
\begin{proof}
By the assumption, $\mathcal{F}^{E}$ is a free $S(\mathbb{Y})/\alpha_{E}$-module.
By GKM condition, this is $\alpha_{E'}$-torsion free.
\end{proof}

\subsection{Root datum}
In this subsection, we introduce notation about root datum.
These notation will be used throughout this paper.
The notation introduced in this subsection is slightly different from the one in Section~\ref{sec:Introduction}.
The main difference is that, following \cite{MR591724}, we use two affine Weyl groups.
Then one is the semidirect product of the finite Weyl group and the root lattice which acts on the set of alcoves from the left.
The other one is defined as a subgroup of the permutations of alcoves and it acts on the set of alcoves from the right.
The latter has a canonical structure of the Coxeter system.

Let $(\mathbb{X},\Phi',\mathbb{X}^\vee,(\Phi')^\vee)$ be a root datum.
We assume that for a (equivalently any) positive system, the half sum of positive coroots is in $\mathbb{X}^{\vee}$.
We also assume that $p$ is not a torsion prime for this datum.

The pairing on $\mathbb{X}\times \mathbb{X}^{\vee}$ is denoted by $\langle\cdot,\cdot\rangle$.
Let $W_{\finite}$ be the (finite) Weyl group of $(\mathbb{X},\Phi',\mathbb{X}^{\vee},(\Phi')^{\vee})$.
Set $W'_{\aff} = W_{\finite}\ltimes \Z(\Phi')^{\vee}$.
We also put $W_{\aff}^{\ext\prime} = W_{\finite}\ltimes \mathbb{X}^{\vee}$.
An element corresponding to $\lambda\in \mathbb{X}^{\vee}$ is denoted by $t_{\lambda}\in W_{\aff}^{\ext\prime}$.

Set $\mathbb{X}^{\vee}_{\R} = \mathbb{X}^{\vee}\otimes_{\Z}\R$.
We define an action of $W^{\ext\prime}_{\aff}$ on $\mathbb{X}^{\vee}$ by $wt_{\lambda}(\nu) = w(\nu + \lambda)$ where $w\in W_{\finite}$ and $\lambda\in \mathbb{X}^{\vee}$.
Let $\mathcal{A}$ be the set of alcoves in $\mathbb{X}_{\R}^{\vee}$, namely the set of connected components of $\mathbb{X}^{\vee}_{\R}\setminus\bigcup_{(\alpha,n)\in\Phi'\times\Z}\{v\in \mathbb{X}^{\vee}_{\R}\mid \langle \alpha,v\rangle = n\}$.
As in \cite{MR591724}, let $S_\aff$ be the set of $W'_\aff$-orbits on the set of faces.
Then, for each $s\in S_\aff$ and $A\in \mathcal{A}$, we set $As$ as the alcove $\ne A$, which has a common face of type $s$ with $A$.
The subgroup of $\Aut(\mathcal{A})$ (permutations of elements in $\mathcal{A}$) generated by $S_\aff$ is denoted by $W_\aff$.
Then $(W_{\aff}, S_{\aff})$ is a Coxeter system.
The Bruhat order on $W_{\aff}$ is denoted by $\le$.
The group $W_{\aff}$ acts on $\mathcal{A}$ from the right.
It commutes with the action of $W_{\aff}'$.

If we fix $A\in \mathcal{A}$, then $W'_{\aff}\to \mathcal{A}$ and $W_{\aff}\to \mathcal{A}$ defined by $w\mapsto wA$, $x\mapsto Ax$ are both bijective.
Hence we get a bijective map $W'_{\aff}\simeq W_{\aff}$.
This bijection is a group isomorphism.
We write $w\mapsto w^{A}$ for the isomorphism $W'_{\aff}\xrightarrow{\sim}W_{\aff}$ and $x\mapsto x_{A}$ for the inverse.
In this subsection, we introduce some objects and state properties of these objects concerning $W_{\aff}$ without proofs.
Every properties will be reduced to well-known properties of $W'_{\aff}$ using this isomorphism.

Let $\Lambda^{\vee}$ be the set of maps $\mathcal{A}\to \mathbb{X}^{\vee}$ such that $\lambda(xA) = \overline{x}\lambda(A)$ for any $x\in W'_\aff$ and $A\in \mathcal{A}$ where $\overline{x}\in W_{\finite}$ is the image of $x$.
We write $\lambda_A = \lambda(A)$ for $\lambda\in \Lambda^{\vee}$ and $A\in \mathcal{A}$.
For each $A\in \mathcal{A}$, $\lambda\mapsto \lambda_A$ gives an isomorphism $\Lambda^{\vee} \xrightarrow{\sim}\mathbb{X}^{\vee}$ and the inverse of this isomorphism is denoted by $\nu\mapsto \nu^A$.
The group $W_\aff$ acts on $\Lambda^{\vee}$ by $(x(\lambda))(A) = \lambda(Ax)$.
This action is compatible with the isomorphism induced by $A\in\mathcal{A}$, namely, for $w\in W_{\aff}$ and $\lambda\in \Lambda^{\vee}$, we have $(w\lambda)_{A} = \overline{w_{A}}\lambda_{A}$.
By replacing $\mathbb{X}^{\vee}$ with $\mathbb{X}$, we can also define a $\Z$-module $\Lambda$.
Fix $A\in \mathcal{A}$ and put $\Phi^{\vee} = \{(\alpha^{\vee})^A\mid \alpha\in\Phi'\}$ (resp.\ $\Phi = \{\alpha^A\mid \alpha\in\Phi'\}$).
Then, this does not depend on $A$ and $(\Phi,\Lambda,\Phi^{\vee},\Lambda^{\vee})$ is a root datum isomorphic to $(\mathbb{X},\Phi',\mathbb{X}^{\vee},(\Phi')^{\vee})$.
For each $\lambda\in \Z\Phi^{\vee}$, we set $A\lambda = A + \lambda_A$.
If $\lambda,\mu\in\Z\Phi^{\vee}$, then $(A\lambda)\mu = A + \lambda_A + \mu_{A + \lambda_A}$.
From the definition of $\Lambda^{\vee}$, we have $\mu_{A + \lambda_A} = \mu_A$.
Hence this defines an action of $\Z\Phi^{\vee}$ on $\mathcal{A}$ and it induces an injective homomorphism $\Z\Phi^{\vee}\hookrightarrow W_{\aff}$.
We regard $\Z\Phi^{\vee}$ as a subgroup of $W_{\aff}$.
This is a normal subgroup of $W_{\aff}$ and we have $w\lambda w^{-1} = w(\lambda)$ in $W_{\aff}$.
The action of $W_{\aff}$ on $\Lambda^{\vee}$ is trivial on $\Z\Phi^{\vee}$.

In a semi-direct product $W_{\aff}\ltimes \Lambda^{\vee}$, the subgroup $\{(\lambda,-\lambda)\mid \lambda\in\Z\Phi^{\vee}\}$ is a normal subgroup.
We put $W_{\aff}^{\ext} = (W_{\aff}\ltimes \Lambda^{\vee})/\{(\lambda,-\lambda)\mid \lambda\in\Z\Phi^{\vee}\}$.
A natural map $W_{\aff}\to W_{\aff}^{\ext}$ is injective and we regard $W_{\aff}$ as a subgroup of $W_{\aff}^{\ext}$.
This is a normal subgroup.
For $(w,\lambda)\in W_{\aff}\ltimes \Lambda^{\vee}$ and $A\in \mathcal{A}$, set $A[(w,\lambda)] = Aw + \lambda_{Aw}$.
Then this is well-defined on $W_{\aff}^{\ext}$, namely it only depends on the coset in $W_{\aff}^{\ext}$.
Note that this is not an action of $W_{\aff}^{\ext}$.
\begin{lem}\label{lem:about root datum etc, length zero element}
\begin{enumerate}
\item The subset $\Omega = \{[(w,\lambda)]\in W_{\aff}^{\ext}\mid Aw + \lambda_{Aw} = A\}$ of $W_{\aff}^{\ext}$ does not depend on $A\in \mathcal{A}$ and it is a subgroup of $W_{\aff}^{\ext}$.
\item We have $W_{\aff}^{\ext}\simeq \Omega\ltimes W_{\aff}$.
\item Let $A\in \mathcal{A}$, $x\in W_{\aff}$ and $y\in W_{\aff}^{\ext}$.
Then $(Ax)y = A(xy)$.
In particular, if $x\in W_{\aff}$ and $y\in \Omega$, then $A(xy) = Ax$.
\end{enumerate}
\end{lem}
\begin{proof}
We prove (1).
Assume that $[(w,\lambda)]\in W_{\aff}^{\ext}$ satisfies $Aw + \lambda_{Aw} = A$.
Let $A'\in \mathcal{A}$ be another element and take $x\in W_{\aff}$ such that $A' = Ax$.
Then $A'w + \lambda_{A'w} = Axw + \lambda_{Axw} = x_{A}Aw + \lambda_{x_{A}Aw} = x_{A}Aw + \overline{x_{A}}\lambda_{Aw}$, here $\overline{x_{A}}\in W_{\finite}$ is the image $x_{A}$ under of the projection $W_{\aff}'\to W_{\finite}$.
By the definition of the action of $W_{\aff}'$, $x_{A}Aw + \overline{x_{A}}\lambda_{Aw} = x_{A}(Aw + \lambda_{Aw}) = x_{A}A = A'$.

Let $[(w',\lambda')]\in\Omega$.
Then we have $[(w,\lambda)][(w',\lambda')] = [(ww',(w')^{-1}(\lambda) + \lambda')]$.
We have $Aww' + ((w')^{-1}\lambda + \lambda')_{Aww'} = Aww' + \lambda_{Aw} + (\lambda')_{Aww'}$.
As $[(w',\lambda')]\in \Omega$, we have $(Aw)w' + (\lambda')_{(Aw)w'} = Aw$.
Hence $Aww' + ((w')^{-1}\lambda + \lambda')_{Aww'} = Aw + \lambda_{Aw} = A$.
Therefore $[(w,\lambda)][(w',\lambda')]\in \Omega$.
We have $[(w,\lambda)]^{-1} = [w^{-1},-w(\lambda)]$ and $Aw^{-1} - (w\lambda)_{Aw^{-1}} = Aw^{-1} - \lambda_{A}$.
Since $[(w,\lambda)]\in \Omega$, $(Aw^{-1})w + \lambda_{(Aw^{-1})w} = Aw^{-1}$.
Hence $Aw^{-1} - \lambda_{A} = A$.
Therefore $[(w,\lambda)]^{-1}\in \Omega$.

(2)
Let $[(w,\lambda)]\in W_{\aff}^{\ext}$ and take $x\in W_{\aff}$ such that $Aw + \lambda_{Aw} = Ax$.
Then $(Ax)(x^{-1}[(w,\lambda)]) = (Ax)([(x^{-1}w,\lambda)]) = Aw + \lambda_{Aw} = Ax$.
Hence $x^{-1}[(w,\lambda)]\in \Omega$.
Therefore $W_{\aff}^{\ext} = W_{\aff}\Omega$.
As $W_{\aff}$ acts on $\mathcal{A}$ simply (transitive), $W_{\aff}\cap \Omega$ is the trivial group.
Hence $W_{\aff}^{\ext}\simeq \Omega\ltimes W_{\aff}$.

(3)
Write $y = [(w,\lambda)]$, so $xy = [(xw,\lambda)]$.
Then $A(xy) = Axw + \lambda_{Axw} = (Ax)[(w,\lambda)]$.
\end{proof}

By Lemma~\ref{lem:about root datum etc, length zero element} (2), we have $\Omega\simeq W_{\aff}^{\ext}/W_{\aff}\simeq \Lambda/\Z\Phi^{\vee}$.
In particular, it is commutative.
By fixing $A\in \mathcal{A}$, we have an isomorphism $\Lambda^{\vee}\simeq \mathbb{X}^{\vee}$ defined by $\lambda\mapsto \lambda_A$ and it induces $\Lambda^{\vee}/\Z\Phi^{\vee}\simeq \mathbb{X}^{\vee}/\Z(\Phi')^{\vee}$.
The latter isomorphism does not depend on $A$.
Hence we have $\Omega\simeq \mathbb{X}^{\vee}/\Z(\Phi')^{\vee}$ canonically.
Since $\Omega$ acts on $W_{\aff}$ by the conjugation, we get an action of $\mathbb{X}^{\vee}/\Z(\Phi')^{\vee}$ on $W_{\aff}$.
This isomorphism is also given as follows.
Let $\lambda\in \mathbb{X}^{\vee}$.
Then for each $s\in S_{\aff}$ and $A\in \mathcal{A}$, an alcove $(A - \lambda)s + \lambda$ only depends on $\lambda + \Z(\Phi')^{\vee}$ and it shares a face with $A$.
Moreover, the $W'_{\aff}$-orbit of such face does not depend on $A$.
Hence it defines $s'\in S_{\aff}$.
The map $s\mapsto s'$ gives an isomorphism $W_{\aff}\to W_{\aff}$.
From this description, the conjugate action of $\Omega$ on $W_{\aff}$ preserves $S_{\aff}$.
Hence $\ell(\omega w\omega^{-1}) = \ell(w)$ for $w\in W_{\aff}$ and $\omega \in\Omega$.
The length function $\ell\colon W_{\aff}\to \Z_{\ge 0}$ can be extended to $W_{\aff}^{\ext}$ by $\ell(\omega w) = \ell(w)$ for $\omega\in \Omega$ and $w\in W_{\aff}$.
The Bruhat order on $W_{\aff}^{\ext}$ is defined as $\omega_1 w_1\le \omega_2w_2$ if and only if $\omega_1 = \omega_2$ and $w_1\le w_2$ for $\omega_1,\omega_2\in\Omega$ and $w_1,w_2\in W_{\aff}$.

We fix a positive system $(\Phi')^+\subset \Phi'$.
Let $s_{(\alpha,n)} = t_{-n\alpha^{\vee}}s_{\alpha}\in W'_{\aff}$ be the reflection with respect to the hyperplane $\{v\in \mathbb{X}^{\vee}_{\R}\mid \langle \alpha,v\rangle + n = 0\}$.
For $\alpha\in (\Phi')^+$ and $A\in \mathcal{A}$, take $n\in \Z$ such that $ - 1 < \langle \alpha,\lambda\rangle + n < 0$ for all $\lambda\in A$ and define $\alpha\uparrow A = s_{(\alpha,n)}(A)$ and $\alpha\downarrow A = s_{(\alpha,n + 1)}A$.
Let $\le$ be the order on $\mathcal{A}$ generated by $A < \alpha\uparrow A$.
Let $\lambda\in \mathbb{X}^{\vee}$.
Then the set $\{A\in\mathcal{A}\mid \lambda\in\overline{A}\}$ has the maximal / minimal element where $\overline{A}$ is the closure of $A$.
The maximal (resp.\ minimal) element is denoted by $A_{\lambda}^+$ (resp.\ $A_{\lambda}^-$).

\begin{lem}\label{lem:order of alcove and weights}
Assume that $A_{1} \le A_{2}$.
Take $x\in W'_{\aff}$ such that $A_{2} = xA_{1}$.
Then for $a\in A_{1}$, we have $x(a) - a\in\R_{\ge 0}(\Phi')^{+}$.
\end{lem}
\begin{proof}
We may assume $A_{2} = \alpha\uparrow A_{1}$ for some $\alpha\in(\Phi')^{+}$.
Take $n\in\Z$ such that $A_{2} = s_{(\alpha,n)}A_{1}$.
Then we have $s_{(\alpha,n)}(a) - a\in \R_{\ge 0}\alpha$ by the definition.
\end{proof}

\begin{lem}\label{lem:order of alcoves, reflection}
Let $\alpha\in (\Phi')^{+}$.
Then $A + \alpha^{\vee} > A$.
Assume that $\langle \alpha,\lambda\rangle + n < 0$ for all $\lambda\in A$.
Then $s_{(\alpha,n)}(A) > A$.
\end{lem}
\begin{proof}
We have $A + \alpha^{\vee} = \alpha\uparrow (\alpha\uparrow A) > A$.
Take $n$ as in the lemma and we prove $s_{(\alpha,n)}(A) > A$ by induction on $m_{A}(n) = \sup_{\lambda\in A}\lvert \langle \alpha,\lambda\rangle + n\rangle\rvert$.
If $m_{A}(n) = 0$, then $s_{(\alpha,n)}(A) = \alpha\uparrow A > A$.
Assume $m_{A}(n) > 0$.
We have $\langle \alpha,\lambda\rangle + n < \langle \alpha,\lambda + \alpha^\vee\rangle + (n - 1)$.
Hence $m_{A + \alpha^{\vee}}(n - 1) < m_{A}(n)$.
By inductive hypothesis and the first part of the lemma, we have $s_{(\alpha,n - 1)}(A + \alpha^{\vee}) > A + \alpha^{\vee} > A$.
By the definition, $s_{(\alpha,n - 1)}(A + \alpha^{\vee}) = s_{(\alpha,n)}(A)$.
We get the lemma.
\end{proof}

We put $\widehat{\mathbb{X}}_{\aff} = (\mathbb{X}\otimes_{\Z}\widehat{\Coef})\oplus \widehat{\Coef}\oplus \widehat{\Coef}$ and $\widehat{\mathbb{X}}_{\aff}^{\vee} = (\mathbb{X}^{\vee}\otimes_{\Z}\widehat{\Coef})\oplus\widehat{\Coef}\oplus\widehat{\Coef}$.
The natural pairing $\langle\cdot,\cdot\rangle$ extends to $\widehat{\mathbb{X}}_{\aff}\times \widehat{\mathbb{X}}_{\aff}^{\vee}\to \widehat{\Coef}$ and we have $\widehat{\mathbb{X}}_{\aff}^{\vee}\simeq\Hom_{\widehat{\Coef}}(\widehat{\mathbb{X}}_{\aff},\widehat{\Coef})$.
\begin{rem}
Perhaps it is better to use the notation $\mathbb{X}_{\aff,\widehat{\Coef}}$ for the group $\widehat{\mathbb{X}}_{\aff}$.
\end{rem}
Set $\Phi'_{\aff} = \{(\alpha,n,0)\mid \alpha\in\Phi',n\in\Z\}\subset \widehat{\mathbb{X}}_{\aff}$.
An element in $\Phi'_{\aff}$ (or $\Phi'\times\Z$) is called an affine root.
We fix a non-degenerate symmetric $W_{\finite}$-invariant bilinear form $(\cdot,\cdot)\colon (\mathbb{X}^{\vee}\otimes_{\Z}\widehat{\Coef})\times (\mathbb{X}^{\vee}\otimes_{\Z}\widehat{\Coef})\to \widehat{\Coef}$ such that $(\lambda,\lambda)\in 2\widehat{\Coef}$ for any $\lambda\in \mathbb{X}^{\vee}\otimes_{\Z}\widehat{\Coef}$.
Since the field of fractions of $\widehat{\Coef}$ has the characteristic zero, we have $\langle \alpha,\nu\rangle = 2(\alpha^\vee,\nu)/(\alpha^{\vee},\alpha^{\vee})$ for any $\alpha\in\Phi'$ and $\nu\in \mathbb{X}^{\vee}\otimes_{\Z}\widehat{\Coef}$.
For $\widetilde{\alpha} = (\alpha,n,0)\in \Phi'$, we put $\widetilde{\alpha}^{\vee} = (\alpha^\vee,0,n((\alpha^{\vee},\alpha^{\vee})/2))$.
We define an action of $W_{\aff}^{\ext\prime}$ on $\widehat{\mathbb{X}}_{\aff}$ (resp.\ $\widehat{\mathbb{X}}_{\aff}^{\vee}$) by $(wt_\lambda)(\nu,r,s) = (w(\nu + s\lambda'),r - (\langle \nu,\lambda\rangle + ((\lambda,\lambda)/2)s),s)$ (resp.\ $(wt_{\lambda})(\mu,r,s) = (w(\mu + r\lambda),r,s - ((\mu,\lambda) + r((\lambda,\lambda)/2))$) where $\lambda'\in \mathbb{X}\otimes_{\Z}\widehat{\Coef}$ is defined by $\langle \lambda',\cdot\rangle = (\lambda,\cdot)$.
The pairing $\widehat{\mathbb{X}}_{\aff}\times \widehat{\mathbb{X}}_{\aff}^{\vee}\to \widehat{\Coef}$ is invariant under these actions.
For each $\widetilde{\alpha} = (\alpha,n,0)\in \Phi_{\aff}$, we put $s_{\widetilde{\alpha}} = s_{(\alpha,n)} = t_{-n\alpha^\vee}s_{\alpha}\in W_{\aff}^{\prime}$.
We have $s_{\widetilde{\alpha}}(\widetilde{\lambda}) = \widetilde{\lambda} - \langle\widetilde{\lambda},\widetilde{\alpha}^{\vee}\rangle\widetilde{\alpha}$ for $\widetilde{\lambda}\in \widehat{\mathbb{X}}_{\aff}$ and $s_{\widetilde{\alpha}}(\widetilde{\nu}) = \widetilde{\nu} - \langle\widetilde{\nu},\widetilde{\alpha}\rangle\widetilde{\alpha}^{\vee}$ for $\widetilde{\nu}\in \widehat{\mathbb{X}}_{\aff}^{\vee}$.
It is easy to see that $\Phi'_{\aff}$ and $(\Phi'_{\aff})^{\vee} = \{\widetilde{\alpha}^{\vee}\mid \widetilde{\alpha}\in \Phi'_{\aff}\}$ are stable under the action of $W_{\aff}^{\ext\prime}$.

As for $\Lambda$, let $\widehat{\Lambda}_{\aff}$ (resp.\ $\widehat{\Lambda}_{\aff}^\vee$) be the space of functions $\lambda\colon \mathcal{A}\to \widehat{\mathbb{X}}_{\aff}$ (resp.\ $\lambda\colon \mathcal{A}\to \widehat{\mathbb{X}}_{\aff}^{\vee}$) such that $\lambda(xA) = x\lambda(A)$ for any $x\in W_{\aff}'$.
The group $W_{\aff}$ acts on $\widehat{\Lambda}_{\aff}$ (resp.\ $\widehat{\Lambda}_{\aff}^{\vee}$) by $(x\lambda)(A) = \lambda(Ax)$.
This action can be extended to $W_{\aff}^{\ext}$ as follows.
We define an action of $\Lambda^{\vee}$ on $\widehat{\Lambda}_{\aff}$ by $(t_{\lambda}\mu)(A) = t_{\lambda(A)}\mu(A)$.
This action is compatible with that of $W_{\aff}$, hence gives an action of $W_{\aff}^{\ext}$ on $\widehat{\Lambda}_{\aff}$.
We also have an action of $W_{\aff}^{\ext}$ on $\widehat{\Lambda}_{\aff}^{\vee}$ by the similar way.
For each $A\in \mathcal{A}$, the map $\lambda\mapsto \lambda_A = \lambda(A)$ gives an isomorphism $\widehat{\Lambda}_{\aff}\simeq \widehat{\mathbb{X}}_{\aff}$ (resp.\ $\widehat{\Lambda}_{\aff}^{\vee}\simeq \widehat{\mathbb{X}}_{\aff}^{\vee}$) and the inverse is denoted by $\nu\mapsto \nu^{A}$.

We also have $\Lambda^{\vee}\simeq \mathbb{X}^{\vee}$ by $\lambda\mapsto \lambda_{A}$ and by combining these isomorphisms we get $W_{\aff}^{\ext}\simeq W_{\aff}^{\ext\prime}$.
Therefore, by fixing $A\in \mathcal{A}$, we get identifications
\begin{equation}\label{eq:identifications of affine Weyl groups}
\widehat{\Lambda}_{\aff}\simeq \widehat{\mathbb{X}}_{\aff}, \quad
\Lambda\simeq \mathbb{X}, \quad
W_{\aff} \simeq W_{\aff}',\quad 
W^{\ext}_{\aff} \simeq W^{\ext\prime}_{\aff}.
\end{equation}
The subset $S_{\aff}\subset W_{\aff}$ corresponds to the reflections with respect to the walls of $A$ in $W_{\aff}^{\prime}$ and $\Omega\subset W_{\aff}^{\ext}$ corresponds to $\{w\in W_{\aff}^{\ext\prime}\mid wA = A\}$.
For each $\nu\in \mathbb{X}^{\vee}$, we have the corresponding element $t_{\nu}$ in $W_{\aff}^{\ext\prime}$.
We write an isomorphism $W_{\aff}^{\ext}\to W_{\aff}^{\exp\prime}$ by $w\mapsto w_{A}$ and the inverse map is written as $x\mapsto x^{A}$.
Sometimes we use these identifications with $A = A_{\lambda}^+$ for $\lambda\in \mathbb{X}^{\vee}$.
In this case we write $t_{\nu}^{\lambda} = t_{\nu}^{A_{\lambda}^{+}}$.

\begin{lem}\label{lem:change of alcove, conjugate}
Let $A\in \mathcal{A}$, $w\in W_{\aff}^{\ext}$ and $x\in W_{\aff}$.
Then $w^{Ax} = x^{-1}w^{A}x$.
\end{lem}
\begin{proof}
If $w\in W_{\aff}$, then this follows from $(Ax)x^{-1}w^{A}x = Aw^{A}x = wAx$.
If $w = t_{\lambda}$ with $\lambda\in \mathbb{X}^{\vee}$, then $w^{Ax} = t_{\lambda^{Ax}} = [1,\lambda^{Ax}]\in W_{\aff}^{\ext}$.
We have $(x^{-1}(\lambda^{A}))(Ax) = \lambda^{A}(Axx^{-1}) =\lambda^{A}(A) = \lambda $.
Hence $\lambda^{Ax} = x^{-1}(\lambda^{A})$.
Therefore $[1,\lambda^{Ax}] = x^{-1}[1,\lambda^{A}]x$.
\end{proof}

\begin{rem}\label{rem:actions on alcoves}
The group $W_{\aff}^{\ext\prime}$ acts on $\mathbb{X}^{\vee}$ and it induces an action on $\mathcal{A}$.
We write this action as $(x,A)\mapsto xA$.
The following can be verified easily.
\begin{itemize}
\item If $x^{A} = [(w,\lambda)]\in W_{\aff}^{\ext}$, then $xA = Aw + \lambda_{Aw}$.
\item The action of $W_{\aff}^{\ext\prime}$ from the left and $W_{\aff}$ from the right do not commute in general.
\item For $\nu\in \mathbb{X}^{\vee}$ we have $t_{\nu}A_{0}^{+} = A_{\nu}^{+}$ and $t_{\nu}A_{0}^{-} = A_{\nu}^{-}$.
\end{itemize}
\end{rem}

Fix $A\in \mathcal{A}$ and we put $\Phi_{\aff} = \{\widetilde{\alpha}^{A}\mid \widetilde{\alpha}\in\Phi'_{\aff}\}$ and $\Phi^{\vee}_{\aff} = \{(\widetilde{\alpha}^{\vee})^{A}\mid \widetilde{\alpha}\in\Phi'_{\aff}\}$.
For $\widetilde{\alpha}^{A}\in \Phi_{\aff}$, we put $(\widetilde{\alpha}^{A})^{\vee} = (\widetilde{\alpha}^{\vee})^{A}\in \Phi^{\vee}_{\aff}$.
We also define $s_{\widetilde{\alpha}^{A}}\in W_{\aff}$ as the element corresponding to $s_{\widetilde{\alpha}}$ via the corresponding $W'_{\aff}\simeq W_{\aff}$ induced by $A$.
We can easily check that these do not depend on $A$.

Let $\lambda\in \mathbb{X}^{\vee}$ and $S_{\aff,\lambda}\subset S_{\aff}$ the set of $s\in S_{\aff}$ such that a facet in $s$ contains $\lambda$ in its closure.
Let $W_{\aff,\lambda}$ be the parabolic subgroup of $W_{\aff}$ generated by $S_{\aff,\lambda}$.
By the isomorphism $W_{\aff}\simeq W_{\aff}'$ induced by $A_{\lambda}^{+}$, it corresponds to $W_{\finite}$.
Let $w_{\lambda}\in W_{\aff,\lambda}$ be the longest element and ${}^{\lambda}W_{\aff}$ (resp.\ $W_{\aff}^{\lambda}$) the set of minimal representatives of $W_{\aff,\lambda}\backslash W_{\aff}$ (resp\ $W_{\aff}/W_{\aff,\lambda}$).

\subsection{Moment graphs}
Let $\widehat{R}_{\aff}^{\vee} = S(\widehat{\Lambda}_{\aff}^\vee)$ and $\widehat{S}^\vee_{\aff} = S(\widehat{\mathbb{X}}_{\aff}^\vee)$ be symmetric algebras.
The isomorphism $\widehat{R}_{\aff}^{\vee}\to \widehat{S}_{\aff}^{\vee}$ induced by $\widehat{\Lambda}_{\aff}^{\vee}\xrightarrow{\sim}\widehat{\mathbb{X}}_{\aff}^{\vee}; \lambda\mapsto \lambda_{A}$ is also denoted by $f\mapsto f_{A}$ and the inverse is written as $g\mapsto g^{A}$.
These are graded $\widehat{\Coef}$-algebras via $\deg(\widehat{\Lambda}^{\vee}_{\aff}) = \deg(\widehat{\mathbb{X}}^{\vee}_{\aff}) = 2$.

We fix a length function $\ell\colon \mathcal{A}\to\Z$ in the sense of \cite[2.11]{MR591724}.
\begin{lem}\label{lem:length of translation}
We have $\ell(A + \lambda) = \ell(A) + 2\langle \rho,\lambda\rangle$.
\end{lem}
\begin{proof}
This follows from \cite[(1.13.1)]{MR772611} and the definition of the function $g$ in \cite{MR772611}.
\end{proof}
The set of alcoves $\mathcal{A}$ has a structure of $\widehat{\Lambda}_{\aff}^{\vee}$-labeled moment graph.
For $A,A'\in \mathcal{A}$, $A$ is connected to $A'$ if and only if there exists $\widetilde{\alpha}\in \Phi_{\aff}$ such that $As_{\widetilde{\alpha}} = A'$ and the label of the edge between $A$ and $A'$ is $\widetilde{\alpha}^{\vee}$.
For a sheaf $\mathcal{F}$ on $\mathcal{A}$ and $A\in \mathcal{A}$, $\mathcal{F}^{A}$ is a graded $\widehat{R}_{\aff}^{\vee}$-module by the definition.
We always regard that this is a right $\widehat{R}_{\aff}^{\vee}$-module.
We define a left $\widehat{S}_{\aff}^{\vee}$-module structure on $\mathcal{F}^{A}$ by $fm = mf^{A}$ for $f\in \widehat{S}_{\aff}^{\vee}$ and $m\in \mathcal{F}^{A}$.
In the same way, we can also define a structure of $\widehat{\Lambda}_{\aff}^{\vee}$-labeled moment graph on $W_{\aff}$.

Let $\lambda\in \mathbb{X}^{\vee}$.
We need the following two lemmas to define a structure of a moment graph on $W_{\aff,\lambda}\backslash W_{\aff}$.
\begin{lem}
Let $x,y\in W_{\aff,\lambda}\backslash W_{\aff}$ and assume that $x\ne y$.
Then there exists at most one reflection $t\in W_{\aff}$ such that $xt = y$.
\end{lem}
\begin{proof}
Take $w\in W_{\aff}$ (resp.\ $w'\in W_{\aff}$) such that its coset is $x$ (resp.\ $y$).
We prove that there is at most one reflection $t$ such that $wt\in W_{\aff,\lambda}w'$.
Since this is equivalent to $(w(w')^{-1})(w't(w')^{-1})\in W_{\aff}$, replacing $w$ with $w(w')^{-1}$ if necessary, we may assume $w' = 1$.
The assumption $x\ne y$ says $w\notin W_{\aff,\lambda}$.
Assume that there exist two distinct reflections $t_{1},t_{2}\in W_{\aff}$ such that $wt_{1},wt_{2}\in W_{\aff,\lambda}$.
Then we have $t_{1}t_{2} = (wt_{1})^{-1}wt_{2}\in W_{\aff,\lambda}$.
Set $z = t_{1}t_{2}$.
The expression $z = t_{1}t_{2}$ is a minimal way to express $z$ as a multiplication of reflections.
Indeed, since $t_{1}\ne t_{2}$, $z$ is not the unit element.
Since $\ell(z)$ is even, it is not a reflection.
Hence the length of such expressions should be greater than one.
By \cite[Theorem~1.4]{MR3294251}, this implies $t_{1},t_{2}\in W_{\aff,\lambda}$.
Hence $w\in W_{\aff,\lambda}t_{1} = W_{\aff,\lambda}$
This contradicts to the assumption $w\notin W_{\aff,\lambda}$.
\end{proof}

\begin{lem}
Let $x,y\in {}^{\lambda}W_{\aff}$ and $t\in W_{\aff}$ a reflection.
Assume that $x\ne y$ and $xt\in W_{\aff,\lambda}y$.
Then $x > y$ or $x < y$.
More precisely, if $xt < x$ (with respect to the Bruhat order on $W_{\aff}$), then $x > y$, otherwise we have $x < y$.
\end{lem}
\begin{proof}
Take $w\in W_{\aff,\lambda}$ such that $xt = wy$.
Since $y\in {}^{\lambda}W_{\aff}$, we have $wy\ge y$.
Hence if $xt < x$, then $x > xt = wy \ge y$.
Assume that $xt > x$.
Let $w = s_{1}\cdots s_{k}$ and $y = t_{1}\ldots t_{l}$ be reduced expressions.
Since $wyt = x < xt = wy$, by the exchange condition, there exists $b$ such that $x = wyt = s_{1}\cdots s_{b - 1}s_{b + 1}\cdots s_{k}y$ or $x = wyt = wt_{1}\cdots t_{b - 1}t_{b + 1}\cdots t_{l}$.
If $x = wyt = s_{1}\cdots s_{b - 1}s_{b + 1}\cdots s_{k}y$, then $x\in W_{\aff,\lambda}y$.
Since $x,y\in {}^{\lambda}W_{\aff}$, this implies $x = y$.
This contradicts to the assumption.
If $x = wt_{1}\cdots t_{b - 1}t_{b + 1}\cdots t_{l}$, then $w^{-1}x < y$.
Hence $x\le w^{-1}x < y$. (The first inequality follows from $x\in {}^{\lambda}W_{\aff}$ and $w^{-1}\in W_{\aff,\lambda}$.)
\end{proof}

The moment graph structure on $W_{\aff,\lambda}\backslash W_{\aff}$ is defined as follows~\cite[Definition~3.7]{MR3324922}.
The set of vertices is $W_{\aff,\lambda}\backslash W_{\aff}$.
We consider this as an ordered set using the bijection ${}^{\lambda}W_{\aff}\simeq W_{\aff,\lambda}\backslash W_{\aff}$, here the order on ${}^{\lambda}W_{\aff}$ is the Bruhat order.
For vertices $x,y\in W_{\aff,\lambda}\backslash W_{\aff}$, $x$ is connected to $y$ if and only if there exists $\widetilde{\alpha}\in\Phi_{\aff}$ such that $x = ys_{\widetilde{\alpha}}$.
The label of this edge is $\widetilde{\alpha}$.

Let $W' = \{1\}$ or $W_{\aff,\lambda}$ and $(\widehat{R}_{\aff}^{\vee})^{W'}$ the $W'$-fixed points in $\widehat{R}_{\aff}^{\vee}$.
For each sheaf $\mathcal{F}$ on $W'\backslash W_{\aff}$ and $x\in W'\backslash W_{\aff}$, we think $\mathcal{F}^{x}$ as an $((\widehat{R}_{\aff}^{\vee})^{W'},\widehat{R}_{\aff}^{\vee})$-bimodule via $fm = mx^{-1}(f)$ for $f\in (\widehat{R}_{\aff}^{\vee})^{W'}$ and $m\in \mathcal{F}^{x}$.

There are relations between these moment graphs.
The map $W_{\aff}\ni w\mapsto A_{0}^{+}w\in \mathcal{A}$ is a bijection and this also induces a bijection with the set of edges preserving labels.
Note that this is not an isomorphism as ordered sets.

We will state a relation between moment graphs $W_{\aff,0}\backslash W_{\aff}$ and $\mathcal{A}$.
We need some preparation.
We say that an affine root $(\alpha,n)$ is positive if $n\ge 0$ and $n > 0$ if $\alpha$ is negative.
Consider the identification $W_{\aff}\simeq W'_{\aff}$ using $A_{0}^{+}$ and regard $W'_{\aff}$ as a Coxeter system.
Then it is known that for $w\in W'_{\aff}$ and an affine root $(\alpha,n)$, $ws_{(\alpha,n)} > w$ if and only if $w(\alpha,n)$ is positive.
(Here we regard $(\alpha,n)\in \Phi'_{\aff}$ by identifying $(\alpha,n)$ with $(\alpha,n,0)$.)

Let $\mathcal{A}^{+}$ be the set of dominant alcoves.

\begin{lem}\label{lem:length of dominant translation of Weyl group elements}
Let $x\in W_{\aff}$.
\begin{enumerate}
\item Assume that $A_{0}^{+}x\in \mathcal{A}^{+}$. Then $\ell(t^{0}_{\lambda}x) = \ell(t^{0}_{\lambda}) + \ell(x)$ for a dominant element $\lambda\in\mathbb{X}^{\vee}$.
\item We have $A_{0}^{+}x\in \mathcal{A}^{+}$ if and only if $x\in {}^{0}W_{\aff}$.
\end{enumerate}
\end{lem}
\begin{proof}
We use the identification $W'_{\aff}\simeq W_{\aff}$ \eqref{eq:identifications of affine Weyl groups} using $A_{0}^{+}$.
Then $W_{\aff,0}$ corresponds to $W_{\finite}$ and $S_{\aff}$ corresponds to the set of simple reflections with respect to the hyperplanes containing faces of $A_{0}^{+}$.
Take $w\in W_{\finite}$ and $\mu\in\mathbb{X}^{\vee}$ such that $x = (t_{\mu}w)_{A_{0}^{+}}$.
We have $t_{\mu}wA_{0}^{+}\in \mathcal{A}^{+}$ if and only if $t_{\mu}w(\varepsilon \rho^{\vee})$ is dominant for small $\varepsilon > 0$.
This is equivalent to: for any $\alpha\in(\Phi')^{+}$, we have $\varepsilon\langle w^{-1}(\alpha),\rho^{\vee}\rangle + \langle \alpha,\mu\rangle > 0$.
Namely $\mu$ is dominant and if $w^{-1}(\alpha) < 0$, then $\langle \alpha,\mu\rangle > 0$.
On the other hand, we have the following length formula
\begin{equation}\label{eq:length formula}
\ell(t_{\mu}w) = \sum_{w^{-1}(\alpha) > 0,\alpha\in(\Phi')^{+}}\lvert\langle \alpha,\mu\rangle\rvert + \sum_{w^{-1}(\alpha) < 0,\alpha\in(\Phi')^{+}}\lvert\langle \alpha,\mu\rangle - 1\rvert.
\end{equation}
From this, we can easily get $\ell(t_{\lambda}t_{\mu}w) = \ell(t_{\lambda}) + \ell(t_{\mu}w)$.
Hence we get (1).

Let $\beta$ be a positive root.
Then $s_{\beta}t_{\mu}w > t_{\mu}w$ if and only if $(w^{-1}(\beta), \langle \mu,\beta\rangle)$ is a positive affine root, namely ($w^{-1}(\beta) > 0$ and $\langle \mu,\beta\rangle\ge 0$) or ($w^{-1}(\beta) < 0$ and $\langle \mu,\beta\rangle > 0$).
Therefore $s_{\beta}t_{\mu}w > t_{\mu}w$ for any positive root $\beta$ if and only if $t_{\mu}wA_{0}^{+}\in \mathcal{A}^{+}$.
\end{proof}

\begin{lem}\label{lem:Bruhat order is generic Bruhat order on dominants}
The bijection $W_{\aff}\simeq \mathcal{A}$ induced by $A_{0}^{+}$ gives an isomorphism ${}^{0}W_{\aff}\simeq \mathcal{A}^{+}$ as ordered sets.
\end{lem}
\begin{proof}
Let $z_{1},z_{2}\in {}^{0}W_{\aff}$ and assume that $A_{0}^{+}z_{1} < A_{0}^{+}z_{2}$.
Set $x = (z_{1})_{A_{0}^{+}},y = (z_{2})_{A_{0}^{+}}$.
Then $xA_{0}^{+} < yA_{0}^{+}$.
We consider $W'_{\aff}$ as a Coxeter system by the identification $W'_{\aff}\simeq W_{\aff}$ induced by $A_{0}^{+}$.
Since $xA_{0}^{+} < yA_{0}^{+}$, there exist $(\alpha_{1},n_{1}),\ldots,(\alpha_{r},n_{r})\in (\Phi')^{+}\times \Z$ and $A_{0},\ldots,A_{r}\in \mathcal{A}$ such that $A_{i} = s_{(\alpha_{i},n_{i})}A_{i - 1}$, $ - 1 < \langle \alpha_{i},a\rangle + n_{i} < 0$ for any $a\in A_{i - 1}$, $A_{0} = xA_{0}^{+}$ and $A_{r} = yA_{0}^{+}$.
Let $\lambda\in\mathbb{X}^{\vee}$ and put $A'_{i} = A_{i} + \lambda$, $m_{i} = n_{i} - \langle \alpha_{i},\lambda\rangle$.
Then for $a'\in A'_{i - 1}$, $ - 1 < \langle \alpha_{i},a'\rangle + m_{i} < 0$ and $A'_{i} = s_{(\alpha_{i},m_{i})}A'_{i - 1}$.
We take sufficiently dominant $\lambda$ such that all $m_{i}$ are negative.
Take $t_{\mu_{i}}w_{i}\in W'_{\aff}$ such that $A_{i} = t_{\mu_{i}}w_{i}A_{0}^{+}$.
Then $t_{\lambda}t_{\mu_{i - 1}}w_{i - 1}(\varepsilon\rho^{\vee})\in A'_{i - 1}$ for sufficiently small $\varepsilon > 0$ and therefore $ - 1 < \langle \alpha_{i},\varepsilon w_{i - 1}(\rho^{\vee})\rangle + \langle \alpha_{i},\mu_{i - 1} + \lambda \rangle + m_{i} < 0$.
Hence $(t_{\lambda}t_{\mu_{i - 1}}w_{i - 1})^{-1}(\alpha_{i},m_{i}) = (w_{i - 1}^{-1}(\alpha_{i}),m_{i} + \langle \alpha_{i},\mu_{i - 1} + \lambda \rangle)$ is a negative affine root.
Since $m_{i} < 0$, $(\alpha_{i},m_{i})$ is a negative affine root.
Hence $t_{\lambda}t_{\mu_{i}}w_{i} = s_{(\alpha_{i},m_{i})}t_{\lambda}t_{\mu_{i - 1}}w_{i - 1} > t_{\lambda}t_{\mu_{i - 1}}w_{i - 1}$.
Since $x = t_{\mu_{0}}w_{0}$ and $y = t_{\mu_{r}}w_{r}$, we have $t_{\lambda}x  < t_{\lambda}y$.
We get $t^{0}_{\lambda}z_{1} < t^{0}_{\lambda}z_{2}$.
By Lemma~\ref{lem:length of dominant translation of Weyl group elements} (1), we have $z_{1} < z_{2}$.
Similar argument shows that if $z_{1} < z_{2}$ then $A_{0}^{+}z_{1} < A_0^{+}z_{2}$.
\end{proof}

\begin{lem}\label{lem:finite weyl group orbit of dominant alcove, order}
If $A\in\mathcal{A}^{+}$ and $x\in W_{\finite}$, then $xA\le A$.
\end{lem}
\begin{proof}
We prove the lemma by induction on the length of $x$.
Take a simple root $\alpha$ such that $sx < x$ with $s = s_{\alpha}$.
Then $x^{-1}(\alpha)$ is negative.
Hence for $\lambda\in A$, $\langle \alpha,x(\lambda)\rangle = \langle x^{-1}(\alpha),\lambda\rangle < 0$.
Therefore $s(xA) = s_{(\alpha,0)}(xA) > xA$ by Lemma~\ref{lem:order of alcoves, reflection}.
\end{proof}

\begin{lem}\label{lem:A is dominant, As > A -> As is dominant}
Let $A\in \mathcal{A}^{+}$ and $s\in S_{\aff}$.
If $As > A$, then $As\in \mathcal{A}^{+}$.
\end{lem}
\begin{proof}
Assume that $As\notin \mathcal{A}^{+}$.
Then $A\cap As$ is contained in a hyperplane which contains a wall of the positive chamber, namely there exists a simple root $\alpha$ such that $A\cap As\subset \{v\in \mathbb{X}^{\vee}_{\R}\mid \langle \alpha,v\rangle = 0\}$.
Hence $As = s_{\alpha}A$.
By Lemma~\ref{lem:finite weyl group orbit of dominant alcove, order}, $As < A$.
\end{proof}

\subsection{Global sections and localization}
Let $X\subset \mathcal{A}$ be a subset and $\mathcal{F}$ a sheaf on $X$.
The space of global sections $\Gamma(\mathcal{F})$ is a module of a structure algebra of $X$.
However, in this paper, we regard $\Gamma(\mathcal{F})$ as an object in the category $\EB(X)$ defined as follows.

Set $Q_{\aff}^{\vee} = \widehat{R}_{\aff}^{\vee}[(\widetilde{\alpha}^{\vee})^{-1}\mid \widetilde{\alpha}\in\Phi_{\aff}]$.
An object of $\EB(X)$ is $(M,(M_{Q}^{A})_{A\in X},\xi_{M})$ where $M$ is a graded $(\widehat{S}_{\aff}^{\vee},\widehat{R}_{\aff}^{\vee})$-bimodule and $M_{Q}^{A}$ is a graded $(\widehat{S}_{\aff}^{\vee},Q_{\aff}^{\vee})$-bimodule such that
\begin{itemize}
\item for $m\in M_{Q}^{A}$ and $f\in \widehat{S}^{\vee}_{\aff}$ we have $fm = mf^{A}$.
\item the set $\{A\in X\mid M_{Q}^{A}\ne 0\}$ is bounded above.
\item $\xi_{M}$ is an injective $(\widehat{S}_{\aff}^{\vee},Q_{\aff}^{\vee})$-bimodule homomorphism $M\otimes_{\widehat{R}_{\aff}^{\vee}}Q_{\aff}^{\vee} \hookrightarrow \prod_{A\in X}M_{Q}^{A}$.
\item for any finite subset $Y\subset X$, $M\otimes_{\widehat{R}_{\aff}^{\vee}}Q_{\aff}^{\vee} \hookrightarrow \prod_{A\in X}M_{Q}^{A}\twoheadrightarrow \bigoplus_{A\in Y}M_{Q}^{A}$ is surjective.
\end{itemize}
We identify $M\otimes_{\widehat{R}_{\aff}^{\vee}}Q_{\aff}^{\vee}$ with the image of $\xi_{M}$ and we regard it as a submodule of $\prod_{A\in X}M_{Q}^{A}$. (Hence we omit $\xi_{M}$.)
We also omit $(M_{Q}^{A})_{A\in X}$ from the notation and simply write $M$ for an object in $\EB(X)$.
A pair $\varphi = (\varphi,(\varphi_{Q}^{A})_{A\in X})$ is a morphism $M\to N$ in $\EB(X)$ if $\varphi\colon M\to N$ is an $(\widehat{S}_{\aff}^{\vee},\widehat{R}_{\aff}^{\vee})$-bimodule homomorphism of degree zero, $\varphi_{Q}^{A}\colon M_{Q}^{A}\to N_{Q}^{A}$ is an $(\widehat{S}_{\aff}^{\vee},Q_{\aff}^{\vee})$-bimodule homomorphism which are compatible with $\xi_{M},\xi_{N}$.
Since $M\otimes_{\widehat{R}_{\aff}^{\vee}}Q_{\aff}^{\vee}\to M_{Q}^{A}$ is surjective for each $A\in \mathcal{A}$, $\varphi_{Q}^{A}$ is determined by $\varphi$.
If $M\in \EB(X)$ and $Y\subset X$, let $M^{Y}$ be the image of $M\to M\otimes_{\widehat{R}_{\aff}^{\vee}}Q_{\aff}^{\vee} \subset \prod_{A\in X}M_{Q}^{A}\twoheadrightarrow \prod_{A\in Y}M_{Q}^{A}$.
We set $M^{x} = M^{\{x\}}$.
Note that if $Y\subset X$ then for $M\in \EB(Y)$, we can regard $M\in \EB(X)$ with $M_{Q}^{A} = 0$ for $A\in X\setminus Y$.
In this way, we regard $\EB(Y)\subset \EB(X)$.

Fiebig defined a localization functor \cite{MR2370278}.
In this paper, we use the following variant of this functor.
Let $M\in \EB(X)$ and we define a sheaf $\mathcal{L}(M)$ on $X$ as follows.
\begin{itemize}
\item $\mathcal{L}(M)^{A} = M^{A}$.
\item Let $E$ be an edge. Then we put $\mathcal{L}(M)^{E} = \Coker(M^{\{A,A'\}}\mathcal{Z}(E)\to M^{A}\oplus M^{A'})$ where $A,A'$ are endpoints of $E$.
\end{itemize}
Here $\mathcal{Z}(E) = \{(f,g)\in (\widehat{R}_{\aff}^{\vee})^{2}\mid f\equiv g\pmod{\alpha_{E}}\}$.

\begin{lem}\label{lem:modules attached to an edge}
Let $X\subset \mathcal{A}$ be a subset, $E$ an edge connecting $A$ with $A'$ in $X$.
Assume that a sheaf $\mathcal{F}$ on $X$ satisfies (BM2) and (BM4).
Then we have $\Gamma(\mathcal{F})^{\{A,A'\}}\mathcal{Z}(E) = \Gamma(\{A,A'\},\mathcal{F}) = \Ker(\mathcal{F}^{A}\oplus\mathcal{F}^{A'}\to \mathcal{F}^{E})$.
\end{lem}
\begin{proof}
We may assume $A < A'$.
The second equality is the definition of the global section.
We prove the first equality.
Obviously, the left-hand side is contained in the right-hand side.
Let $(a,b)\in \Gamma(\{A,A'\},\mathcal{F})$ with $a\in \mathcal{F}^{A}$ and $b\in \mathcal{F}^{A'}$.
Since $\Gamma(\mathcal{F})\to \mathcal{F}^{A}$ is surjective, there exists $m = (m_{A_{1}})_{A_{1}\in X}\in \Gamma(\mathcal{F})$ such that $m_{A} = a$.
Then $(m_{A},m_{A'})$ is in the left hand side.
By replacing $(a,b)$ with $(a,b) - (m_{A},m_{A'})$, we may assume $a = 0$.
Then $b$ is in the kernel of $\mathcal{F}^{A'}\to \mathcal{F}^{E}$, which is $\mathcal{F}^{A'}\alpha_{E}$.
Take $n = (n_{A_{1}})\in \Gamma(\mathcal{F})$ such that $b = n_{A'}\alpha_{E}$.
Then $(0,b) = (n_{A},n_{A'})(0,\alpha_{E})$ and it is in the left hand side.
\end{proof}

Let $W' = \{1\}$ or $W_{\aff,\lambda}$.
Then for $X\subset W'\backslash W_{\aff}$, we can define $\EB(X)$ by the similar way.
An object is a pair $(M,(M^{x}_{Q})_{x\in X})$ where $M$ is an $((\widehat{R}_{\aff}^{\vee})^{W'},\widehat{R}_{\aff}^{\vee})$-bimodule, $M^{x}_{Q}$ is a $((Q^{\vee}_{\aff})^{W'},Q^{\vee}_{\aff})$-bimodule such that $fm = mx^{-1}(f)$ where $m\in M^{x}_{Q}$, $f\in (\widehat{R}_{\aff}^{\vee})^{W'}$ and $M\otimes_{\widehat{R}_{\aff}^{\vee}}Q^{\vee}_{\aff} = \bigoplus_{x\in X}M^{x}_{Q}$.
For a sheaf $\mathcal{F}$ on $X$, $\Gamma(\mathcal{F})\in \EB(X)$ and for $M\in \EB(X)$ we can define a sheaf $\mathcal{L}(M)$ on $X$.

\subsection{Translation}
The category of sheaves on $\mathcal{A}$ is invariant under the translations.
The precise statement is the following.
Let $\lambda\in \mathbb{X}^{\vee}$ and $\sigma_{\lambda}$ the image of $\lambda\in \mathbb{X}^{\vee}$ under the isomorphism $\mathbb{X}^{\vee}/\Z\Phi^{\vee}\simeq \Omega$.
For a sheaf $\mathcal{F}$ on $\mathcal{A}$, we define $T_{\lambda}^{*}\mathcal{F}\in \Sh(\mathcal{A})$ as follows.
We put $(T_{\lambda}^{*}\mathcal{F})^{A} = \mathcal{F}^{A - \lambda}$ for $A\in \mathcal{A}$ and $(T_{\lambda}^{*}\mathcal{F})^{E} = \mathcal{F}^{E - \lambda}$ for an edge $E$ in $\mathcal{A}$ where $E - \lambda$ is an edge connecting $A - \lambda$ with $A' - \lambda$ if $E$ connects $A$ with $A'$.
We twist the $\widehat{R}_{\aff}^{\vee}$-module structure on $\mathcal{F}^{A - \lambda}$ and $\mathcal{F}^{E - \lambda}$ by $\sigma_{\lambda}$.
The module $(T_{\lambda}^{*}\mathcal{F})^{E}$ is killed by $\alpha_{E}$ by the following lemma.

\begin{lem}
Let $E$ be an edge.
Then $\alpha_{E - \lambda} = \sigma_{\lambda}^{-1}(\alpha_{E})$.
\end{lem}
\begin{proof}
Let $A\in \mathcal{A}$.
We have $A - \lambda = At_{-\lambda}^{A}$.
We put $\sigma = \sigma_{\lambda}$.
Then $A - \lambda = At_{-\lambda}^{A} = A(t_{-\lambda}^{A}\sigma\sigma^{-1})$.
The definition of $\sigma$ says $t_{-\lambda}^{A}\sigma\in W_{\aff}$.
Hence, by Lemma~\ref{lem:about root datum etc, length zero element} (3), $A(t_{-\lambda}^{A}\sigma\sigma^{-1}) = (A(t_{-\lambda}^{A}\sigma))\sigma^{-1}$ and, by the definition of $\Omega$, we have $(A(t_{-\lambda}^{A}\sigma))\sigma^{-1} = A(t_{-\lambda}^{A}\sigma)$.
Therefore, for $w\in W_{\aff}$, we have $(A - \lambda)w = (A(t_{-\lambda}^{A}\sigma))w = A(t_{-\lambda}^{A}\sigma w)$.
Moreover, we have $A(t_{-\lambda}^{A}\sigma w) = A(t_{-\lambda}^{A}\sigma w)\sigma^{-1} = A(t_{-\lambda}^{A}\sigma w\sigma^{-1})$ by the definition of $\Omega$ and Lemma~\ref{lem:about root datum etc, length zero element} (3).
By Lemma~\ref{lem:change of alcove, conjugate},  $t_{-\lambda}^{A}\sigma w\sigma^{-1} = \sigma w\sigma^{-1} t_{-\lambda}^{A\sigma w\sigma^{-1}}$.
Hence $A(t_{-\lambda}^{A}\sigma w\sigma^{-1}) = A(\sigma w\sigma^{-1} t_{-\lambda}^{A\sigma w\sigma^{-1}}) = (A(\sigma w\sigma^{-1})) t_{-\lambda}^{A\sigma w\sigma^{-1}}$ by Lemma~\ref{lem:about root datum etc, length zero element} (3).
This is equal to $A(\sigma w\sigma^{-1}) - \lambda$.
As a consequence, we get $(A - \lambda)w = A(\sigma w\sigma^{-1}) - \lambda$.

Let $A,A'$ be endpoints of $E$ and $t$ a reflection such that $A' = At$.
Then by putting $w = \sigma^{-1}t\sigma$, we have $(A - \lambda)(\sigma^{-1} t\sigma) = A' - \lambda$.
Hence $s_{\alpha_{E - \lambda}} = \sigma^{-1}t\sigma$.
Now the lemma follows from $\sigma^{-1}s_{\alpha_{E}}\sigma = s_{\sigma^{-1}\alpha_{E}}$.
\end{proof}

It is obvious that $T_{\lambda}^{*}$ gives an automorphism of the category $\Sh(\mathcal{A})$ that preserves $\BM(\mathcal{A})$.
Moreover, we have $T_{\lambda}^{*}\mathcal{B}^{\mathcal{A}}(A) = \mathcal{B}^{\mathcal{A}}(A + \lambda)$.
In particular, we get the following.
\begin{lem}\label{lem:grk of translation}
Let $A,A'\in \mathcal{A}$ and $\lambda\in \mathbb{X}^{\vee}$.
Then we have $\grk\mathcal{B}^{\mathcal{A}}(A + \lambda)^{A' + \lambda} = \grk\mathcal{B}^{\mathcal{A}}(A)^{A'}$ and $\grk\mathcal{B}^{\mathcal{A}}(A + \lambda)^{[A' + \lambda]} = \grk\mathcal{B}^{\mathcal{A}}(A)^{[A']}$.
\end{lem}

\section{The action of Soergel bimodules}
\subsection{Definition}\label{subsec:Soergel bimodules}
Let $\mathcal{S}$ be the category of Soergel bimodules attached to $(W_{\aff},S_{\aff})$ and $\widehat{\mathbb{X}}_{\aff}^{\vee}$~\cite{MR4321542}.
This category is a full-subcategory of $\EB(W_{\aff})$ and generated by $B_{s}\ (s\in S_{\aff})$ as a monoidal category, where $B_{s} = \{(f,g)\in (\widehat{R}_{\aff}^{\vee})^{2}\mid f\equiv g\pmod{\widetilde{\alpha}_{s}^{\vee}}\}(1)$.
The action of $a,b\in \widehat{R}_{\aff}^{\vee}$ is given by $a(f,g)b = (afb,s(a)gb)$.
We have $B_{s}\otimes Q^{\vee}_{\aff}\simeq (Q^{\vee}_{\aff})^{2}$ and we put $(B_{s})_{Q}^{1} = Q^{\vee}_{\aff}\oplus 0,(B_{s})^{s}_{Q} = 0\oplus Q^{\vee}_{\aff}$ and $(B_{s})^{x}_{Q} = 0$ for $x\in W_{\aff}\setminus\{1,s\}$.
We put $\supp_{W_{\aff}}(B) = \{x\in W_{\aff}\mid B_{Q}^{x}\ne 0\}$ for $B\in \mathcal{S}$.

We define an action of $\mathcal{S}$ on $\BM(\mathcal{A})$.
Namely for $B\in \mathcal{S}$ and $\mathcal{F}\in \BM(\mathcal{A})$, we define $\mathcal{F}\star B\in \BM(\mathcal{A})$.
For this, it is sufficient to define $(\mathcal{F}\star B)|_{O}$ for any $O\in \bbopen$ which are compatible.

Let $O\in\bbopen$ and take $O'\in \bbopen$ such that $O(\supp_{W_{\aff}}B)^{-1}\subset O'$.
Such $O'$ exists by the following lemma.

\begin{lem}\label{lem:existence of O' supset Ox}
Let $O\in \bbopen$ and $x\in W_{\aff}$.
Then there exists $O'\in \bbopen$ such that $O' \supset Ox$.
\end{lem}
\begin{proof}
We may assume $x = s\in S_{\aff}$ and $O = \{A'\in \mathcal{A}\mid A'\ge A\}$ for some $A$.
Then $O' = Os \cup O$ satisfies the lemma.
\end{proof}

Let $\mathcal{F}$ be a sheaf on $O'$.
We have $\Gamma(\mathcal{F})\in \EB(O')$.
We define $(\Gamma(\mathcal{F})\otimes B)|_{O}$ as follows.
We put $(\Gamma(\mathcal{F})\otimes B)_{Q}^{A} = \bigoplus_{x\in \supp_{W_{\aff}}B}\Gamma(\mathcal{F})_{Q}^{Ax^{-1}}\otimes_{Q_{\aff}^{\vee}} B_{Q}^{x}$ for $A\in O$.
We have $(\Gamma(\mathcal{F})\otimes_{\widehat{R}_{\aff}^{\vee}} B)\otimes_{\widehat{R}_{\aff}^{\vee}}Q_{\aff}^{\vee} = \bigoplus_{A\in O'}\bigoplus_{x\in \supp_{W_{\aff}}B}\Gamma(\mathcal{F})_{Q}^{Ax^{-1}}\otimes_{Q_{\aff}^{\vee}}B_{Q}^{x}$ and therefore we have a (surjective) map $(\Gamma(\mathcal{F})\otimes B)\otimes_{\widehat{R}_{\aff}^{\vee}}Q_{\aff}^{\vee}\to \bigoplus_{A\in O}(\Gamma(\mathcal{F})\otimes B)_{Q}^{A}$.
Let $\Gamma(\mathcal{F})\otimes B$ be the image of $\Gamma(\mathcal{F})\otimes_{\widehat{R}_{\aff}^{\vee}}B$ under this map.
Then we have $\Gamma(\mathcal{F})\otimes B\in \EB(O)$.
Now we define $\mathcal{F}\star B = \mathcal{L}(\Gamma(\mathcal{F})\otimes B)$ which is a sheaf on $O$.
It is not difficult to see that this definition ``does not depend on $O'$'' if $\mathcal{F}$ is flabby, namely if $O''\in \bbopen$ contains $O'$ and $\mathcal{F}\in \BM(O'')$, then $\mathcal{F}\star B\simeq (\mathcal{F}|_{O'})\star B$.

We prove the following.
\begin{thm}\label{thm:action of Soergel bimodule preserves BM sheaves}
In the above setting, if $\mathcal{F}\in \BM(O')$ then $\mathcal{F}\star B\in\BM(O)$.
\end{thm}

First we prove this when $B = B_{s}$ for some $s\in S_{\aff}$.
In this case, by the proof of Lemma~\ref{lem:existence of O' supset Ox}, we may assume that $O'$ is $s$-stable.
Hence $\mathcal{F}\star B$ is also a sheaf on $O'$.
Therefore we may assume $O = O'$ is $s$-invariant.
We prove in more general situation.
Let $X\subset \mathcal{A}$ be a bounded below $s$-invariant subset.
We replace $X\cap O$ with $X$, so we prove that if $\mathcal{F}\in \BM(X)$ then $\mathcal{F}\star B_{s}\in \BM(X)$ for any $s$-invariant  bounded below subset $X$.

For the proof,  we express $\mathcal{F}\star B_{s}$ using the translation functor~\cite{MR2395170}.
Fix $s\in S_{\aff}$ and let $X$ be an $s$-invariant locally closed subset of $\mathcal{A}$ and $\mathcal{F}$ a sheaf on $X$.
We define $\theta_{s}\mathcal{F}$ by the following:
Let $A\in X$, we put $(\theta_{s}\mathcal{F})^{A} = \Gamma(\{A,As\},\mathcal{F})(1)$.
If $E$ is an edge connecting $A$ with $A'$, then we have a map $(\theta_{s}\mathcal{F})^{A}\to \mathcal{F}^{A}(1)\to \mathcal{F}^{E}(1)$.
We define $(\theta_{s}\mathcal{F})^{E} = \Ima((\theta_{s}\mathcal{F})^{A}\oplus (\theta_{s}\mathcal{F})^{A'}\to \mathcal{F}^{E}(1)\oplus \mathcal{F}^{Es}(1))$ if $A'\ne As$.
If $A' = As$, then we put $(\theta_{s}\mathcal{F})^{E} = (\theta_{s}\mathcal{F})^{A}/(\theta_{s}\mathcal{F})^{A}\alpha_{E}$.

\begin{prop}\label{prop:action of BS module is the translation functor}
If $\mathcal{F}$ satisfies (BM1)(BM2) and (BM4), we have $\mathcal{F}\star B_{s}\simeq \theta_{s}\mathcal{F}$.
\end{prop}

Let $A,A'\in \mathcal{A}$ and assume that they are connected by an edge.
In this subsection, we write such edge by $E_{A,A'}$.
By the definition we have $B_{s}\simeq \mathcal{Z}(E_{A,As})$.

\begin{lem}
If $\mathcal{F}$ satisfies (BM2) and (BM4), then we have $(\mathcal{F}\star B_{s})^{A}\simeq \Gamma(\{A,As\},\mathcal{F})(1)$.
Here the action of $f\in \widehat{S}_{\aff}^{\vee}$ and $g\in \widehat{R}_{\aff}^{\vee}$ on $(a,b)\in \Gamma(\{A,As\},\mathcal{F})\subset \mathcal{F}^{A}\oplus \mathcal{F}^{As}$ is given by $f(a,b)g = (fag,fbs(g))$.
\end{lem}
\begin{proof}
The image of $\Gamma(\mathcal{F})$ in $(\mathcal{F}\star B_{s})^{A}_{Q} = \mathcal{F}^{A}_{Q}\oplus \mathcal{F}^{As}_{Q}$ is $\Gamma(\mathcal{F})^{\{A,As\}}$ by the definition and we have $B_{s}\simeq \mathcal{Z}(E_{A,As})(1)$.
Since $(\mathcal{F}\star B_{s})^{A}$ is the image of $\Gamma(\mathcal{F})\otimes B_{s}$, $(\mathcal{F}\star B_{s})^{A} = \Gamma(\mathcal{F})^{\{A,As\}}\mathcal{Z}(E_{A,s})(1)$.
We get the conclusion by Lemma~\ref{lem:modules attached to an edge}.
\end{proof}

Next, we compare the modules attached to an edge $E = E_{A,A'}$ connecting $A$ with $A'$.
First, assume that $E\ne Es$.
We have
\begin{align*}
& \Ker((\theta_{s}\mathcal{F})^{A}\oplus (\theta_{s}\mathcal{F})^{A'}\to \mathcal{F}^{E}(1)\oplus \mathcal{F}^{Es}(1))\\
& \simeq \Ker(\Gamma(\{A,As\},\mathcal{F})\oplus \Gamma(\{A',A's\},\mathcal{F})\to (\mathcal{F}^{E}\oplus \mathcal{F}^{Es})(1))
\end{align*}
and it is equal to $\Gamma(\{A,As,A',A's\},\mathcal{F})(1)$ by the definition of global sections.
Recall that $\mathcal{L}(M)^{E} = (M^{A}\oplus M^{A'})/M^{\{A,A'\}}\mathcal{Z}(E)$.
Hence the following lemma proves $(\mathcal{F}\star B_{s})^{E}\simeq (\theta_{s}\mathcal{F})^{E}$.

\begin{lem}
Assume that $\mathcal{F}$ satisfies (BM1), (BM2) and (BM4).
Then we have $(\Gamma(\mathcal{F})\otimes B_{s})^{\{A,A'\}}\mathcal{Z}(E) \simeq \Gamma(\{A,As,A',A's\},\mathcal{F})(1)$ for any $A\ne A's$.
\end{lem}
\begin{proof}
In the proof, we ignore the grading.
Both sides are subspaces of $\mathcal{F}^{A}\oplus \mathcal{F}^{As}\oplus \mathcal{F}^{A'}\oplus \mathcal{F}^{A's}$.
We may assume $A' > A$.
First we prove that the left hand side is contained in the right hand side.
Since the right hand side is stable under the action of $\mathcal{Z}(E)$, we prove $(\Gamma(\mathcal{F})\otimes B_{s})^{\{A,A'\}}$ is contained in the right hand side.
Since $B_{s} = \widehat{R}_{\aff}^{\vee}(1,1) + \widehat{R}_{\aff}^{\vee}(\widetilde{\alpha}^{\vee}_{s},1)$, it is sufficient to prove that $a\otimes (1,1),a\otimes (\widetilde{\alpha}^{\vee}_{s},1)$ is contained in the right hand side for $a = (a_{A_{1}})_{A_{1}\in \mathcal{A}}\in \Gamma(\mathcal{F})$.
In general, we write the $A_{1}$-component of $m\in \bigoplus_{A_{2}}\mathcal{F}^{A_{2}}$ as $m_{A_{1}}$.
Then in $\mathcal{F}^{A}\oplus \mathcal{F}^{As}\oplus \mathcal{F}^{A'}\oplus \mathcal{F}^{A's}$, $a\otimes (1,1) = (a_{A},a_{As},a_{A'},a_{A's})$ which is obviously in the right hand side.
We also have $a\otimes (\widetilde{\alpha}_{s}^{\vee},0) = (a_{A}\widetilde{\alpha}_{s}^{\vee},0,a_{A'}\widetilde{\alpha}_{s}^{\vee},0)$.
The label attached to the edge connecting $A$ with $As$ (resp. $A'$ with $A's$) is $\widetilde{\alpha}_{s}^{\vee}$.
From this, one can easily check that $a\otimes (\widetilde{\alpha}_{s},0)$ is in the right hand side.

We prove the reverse inclusion.
Let $a = (a_{A},a_{As},a_{A'},a_{A's})$ in the right hand side.
First, notice that to prove that this element is in the left-hand side, we may assume one of these is zero.
For example, to assume that $a_{A} = 0$, we argue as follows.
Since $\Gamma(\mathcal{F})\to \mathcal{F}^{A}$ is surjective, there exists $m\in \Gamma(\mathcal{F})$ such that the $A$-component of $m$ is $a_{A}$.
Then $m\otimes (1,1)$ is in the left-hand side and the $A$-component of $a - m\otimes (1,1)$ is zero.

We assume $As < A < A' < A's$.
We may assume $a_{As} = 0$.
Then $a_{A}\in \Ker(\mathcal{F}^{A}\to \mathcal{F}^{E_{A,As}}) = \mathcal{F}^{A}\widetilde{\alpha}_{s}^{\vee}$.
Take $a' = (a_{A_{1}})\in \Gamma(\mathcal{F})$ such that $a_{A} = a'_{A}\widetilde{\alpha}_{s}^{\vee}$.
Then we have $a'\otimes (\widetilde{\alpha}_{s}^{\vee},0)\in \Gamma(\mathcal{F})\otimes B_{s}$.
By replacing $a$ with $a - a'\otimes (\widetilde{\alpha}_{s}^{\vee},0)$, we may assume $a_{A} = 0$.
Then $a_{A'}\in \mathcal{F}^{A'}\alpha_{E}$.
Take $a''\in \Gamma(\mathcal{F})$ such that $a_{A'} = a''_{A'}\alpha_{E}$.
By replacing $a$ with $a - (a''\otimes (1,1))(0,\alpha_{E})$, we may assume $a_{A'} = 0$.
Then $a_{A's}\in \mathcal{F}^{A's}\alpha_{E}\cap \mathcal{F}^{A's}\widetilde{\alpha}_{s}^{\vee}$.
Since $\mathcal{F}^{A's}$ is free and our moment graph $\mathcal{A}$ satisfies the GKM condition, we have $\mathcal{F}^{A's}\alpha_{E}\cap \mathcal{F}^{A's}\widetilde{\alpha}_{s}^{\vee} = \mathcal{F}^{A's}\alpha_{E}\widetilde{\alpha}_{s}^{\vee}$.
We take $a'''\in \Gamma(\mathcal{F})$ such that $a_{A's} = a'''_{A's}\alpha_{E}\widetilde{\alpha}_{s}^{\vee}$.
Then $a = (a'''\otimes (0,\widetilde{\alpha}_{s}^{\vee}))(0,\alpha_{E})$ and this is in the left hand side.

The proof in the other cases are similar.
Namely, we look for an element in the left-hand side and deduce to the case that a certain component of $a$ is zero.
We only explain in which order the components are reduced to zero and leave the details to the reader.

Assume $As < A < A' > A's$.
In this case, we have $As < A's$~\cite[proposition~3.1]{MR591724} and we can reduce components of $a$ to zero in the order $As, A, A's, A'$.
Assume $As > A < A' < A's$.
Then we have $A's > As$, and the order is $A,As,A',A's$.
Assume $As > A < A' > A's$.
Then we have $As < A's$ by the following lemma and the order is $A,As,A's,A'$.
\end{proof}

\begin{lem}\label{lem:lemma on a order in alcoves}
Let $A,A'\in \mathcal{A}$ such that $A$ and $A'$ are connected in the moment graph and $s\in S_{\aff}$.
If $As > A < A' > A's$ and $A'\ne As$, then $As < A's$.
\end{lem}
\begin{proof}
Assume that we are given finite $A_{1},\ldots,A_{r}\in \mathcal{A}$ and take $\lambda\in \mathbb{X}^{\vee}$ such that $A_{i} + \lambda$ is dominant for any $i = 1,\ldots,r$.
Take $w_{i}\in W_{\aff}$ such that $A_{i} + \lambda_{i} = A_{0}^{+}w_{i}$.
Then we have $w_{i} < w_{j}$ if and only if $A_{i} < A_{j}$ for any $i ,j = 1,\ldots,r$ by Lemma~\ref{lem:Bruhat order is generic Bruhat order on dominants}.
Therefore it is sufficient to prove the corresponding lemma in $W_{\aff}$.
Namely we prove the following: let $w\in W_{\aff},s\in S_{\aff}$ and $t\in W_{\aff}$ such that $t\ne s$ is a reflection, $ws > w < wt > wts$.
Then we have $ws < wts$.
Since $wts = ws(s^{-1}ts)$ and $s^{-1}ts$ is a reflection, it is sufficient to prove that $\ell(ws) < \ell(wts)$.
We have $\ell(ws) = \ell(w) + 1$ and $\ell(wts) = \ell(wt) - 1$.
Therefore it is sufficient to prove that $\ell(wt) > \ell(w) + 2$.
Assume that $\ell(wt) \le \ell(w) + 2$.
Since $\ell(wt) > \ell(w)$ and $\ell(wt)\not\equiv \ell(w)\pmod{2}$, we have $\ell(wt) = \ell(w) + 1$.

Take a reduced expression $wt = s_{1}\ldots s_{l}$ of $wt$ such  that $s_{l} = s$.
By the exchange condition, there exists $i = 1,\ldots,l$ such that $w = s_{1} \cdots s_{i - 1}s_{i + 1}\cdots s_{l}$ and by the assumption $\ell(wt) = \ell(w) + 1$, this is a reduced expression.
If $i \ne l$, then $ws < w$.
If $i = l$, then $t = s$.
In both cases, we get a contradiction.
\end{proof}

\begin{lem}
Let $E = E_{A,As}$.
Then for any sheaf $\mathcal{F}$ on $X$, we have $\mathcal{L}(\Gamma(\mathcal{F})\otimes B_{s})^{E} \simeq \mathcal{L}(\Gamma(\mathcal{F})\otimes B_{s})^{A}/\mathcal{L}(\Gamma(\mathcal{F})\otimes B_{s})^{A}\alpha_{E}$.
\end{lem}
\begin{proof}
Set $M = \Gamma(\mathcal{F})\otimes B_{s}$.
By definition we have $(\mathcal{F}\star B_{s})^{E} = \mathcal{L}(M)^{E} = (M^{A}\oplus M^{As})/M^{\{A,As\}}\mathcal{Z}(E)$.
The inclusion into the first factor gives a map $M^{A}\to \mathcal{L}(M)^{E}$.
By the definition $M^{\{A,As\}}\to M^{As}$ is surjective.
Hence the map $M^{A}\to \mathcal{L}(M)^{E}$ is surjective.
We prove that the kernel $K$ of this map is $M^{A}\widetilde{\alpha}_{s}^{\vee}$.

Let $a\in M^{A}$ and take $m\in M$ whose image is $a$.
Then $(a\widetilde{\alpha}_{s}^{\vee},0) = m(\widetilde{\alpha}_{s}^{\vee},0)\in M^{\{A,As\}}\mathcal{Z}(E)$.
Hence $a\widetilde{\alpha}_{s}^{\vee}\in K$.

We prove the reverse inclusion.
As a left $\widehat{R}_{\aff}^{\vee}$-module, $\mathcal{Z}(E)$ has a basis $(1,1),(\widetilde{\alpha}_{s}^{\vee},0)$.
Hence any element of $M^{\{A,As\}}\mathcal{Z}(E)$ can be written as $(m_{A} + m'_{A}\widetilde{\alpha}_{s}^{\vee},m_{As})$ with $m,m'\in M$ and $m_{A}$ is the image of $m$ under $M\to M^{A}$.
If this is in $K$, then $m_{As} = 0$.
We prove $m_{A}\in M^{A}\widetilde{\alpha}_{s}^{\vee}$.
We have $M = \Gamma(\mathcal{F})\otimes B_{s}$ and the element $m$ in $M$ is written as $m = n\otimes (1,1) + n'\otimes (\widetilde{\alpha}_{s}^{\vee},0)$ with $n = (n_{A'})_{A'},n' = (n'_{A'})_{A'}\in \Gamma(\mathcal{F})$.
Note that $M^{A} = M^{As} \subset \mathcal{F}^{A}\oplus \mathcal{F}^{As}$.
The image of $m$ in $M^{A}$ (resp.\ $M^{As}$) is $(n_{A} + n'_{A}\widetilde{\alpha}_{s}^{\vee},n_{As})$ (resp.\ $(n_{A},n_{As} + n'_{As}\widetilde{\alpha}_{s}^{\vee})$).
If $m_{As} = 0$, then $n_{A} = n_{As} + n'_{As}\widetilde{\alpha}_{s}^{\vee} = 0$.
Hence $(n_{A} + n'_{A}\widetilde{\alpha}_{s}^{\vee},n_{As}) = (n'_{A}\widetilde{\alpha}_{s}^{\vee},-n'_{As}\widetilde{\alpha}_{s}^{\vee}) = (n'_{A},n'_{As})\widetilde{\alpha}_{s}^{\vee}$. (Note that $s(\widetilde{\alpha}_{s}) = -\widetilde{\alpha}_{s}$.)
This is in $M^{A}\widetilde{\alpha}_{s}^{\vee}$.
Hence $m_{A}\in M^{A}\widetilde{\alpha}_{s}^{\vee}$.
\end{proof}

Proposition~\ref{prop:action of BS module is the translation functor} follows.

\begin{lem}\label{lem:s-stable global section of translation functor}
Assume that $\mathcal{F}$ satisfies (BM1) and (BM2).
Then we have $\Gamma(X,\theta_{s}\mathcal{F}) \simeq \Gamma(X,\mathcal{F})\otimes B_{s}$.
\end{lem}
\begin{proof}
We have an embedding $B_{s}\subset \widehat{R}_{\aff}^{\vee}\oplus \widehat{R}_{\aff}^{\vee}$ and the cokernel is isomorphic to $\widehat{R}_{\aff}^{\vee}/\widetilde{\alpha}_{s}^{\vee}\widehat{R}_{\aff}^{\vee}$.
Since $\mathcal{F}^{x}$ are graded free for $x\in X$, $\Gamma(X,\mathcal{F})\subset \prod_{x\in X}\mathcal{F}^{x}$ is torsion-free.
Hence the first map of the induced exact sequence
\[
\Gamma(X,\mathcal{F})\otimes B_{s} \to \Gamma(X,\mathcal{F})^{2}\to \Gamma(X,\mathcal{F})\otimes_{\widehat{R}_{\aff}^{\vee}}(\widehat{R}_{\aff}^{\vee}/\widetilde{\alpha}_{s}^{\vee}\widehat{R}_{\aff}^{\vee})\to 0
\]
is injective.
Moreover, by this exact sequence, the image of this embedding is $\{(a,b)\in \Gamma(X,\mathcal{F})^{2}\mid a - b \in \Gamma(X,\mathcal{F})\widetilde{\alpha}_{s}^{\vee}\}$.
We prove that $\Gamma(X,\theta_{s}\mathcal{F})$ has the same description.

Since $\theta_{s}(\mathcal{F})^{A} = \mathcal{F}^{A}\oplus \mathcal{F}^{As}$, we have $\Gamma(X,\theta_{s}\mathcal{F})\subset \bigoplus_{A\in X}(\theta_{s}\mathcal{F})^{A} = \bigoplus_{A\in X}\mathcal{F}^{A}\oplus \bigoplus_{A\in X}\mathcal{F}^{As} = \left(\bigoplus_{A\in X}\mathcal{F}^{A}\right)^{\oplus 2}$.
First we prove that $\Gamma(X,\theta_{s}\mathcal{F})\subset (X,\mathcal{F})^{2}$.
Take $a_{A}\in \mathcal{F}^{A}$ and $b_{As}\in \mathcal{F}^{As}$ such that $(a_{A},b_{As})_{A}\in \Gamma(X,\theta_{s}\mathcal{F})$.
Let $E$ be an edge connecting $A$ with $A'$.
If $A'\ne As$, then $\rho_{E,A}^{\theta_{s}\mathcal{F}}(a_{A},b_{As}) =  \rho_{E,A'}^{\theta_{s}\mathcal{F}}(a_{A'},b_{A's})$ implies $\rho_{E,A}^{\mathcal{F}}(a_{A}) =  \rho_{E,A'}^{\mathcal{F}}(a_{A'})$ and $\rho_{Es,As}^{\mathcal{F}}(b_{As}) =  \rho_{Es,A's}^{\mathcal{F}}(b_{A's})$.
By replacing $E$ with $Es$ in the latter, we get $\rho_{E,A}^{\mathcal{F}}(b_{A}) =  \rho_{E,A'}^{\mathcal{F}}(b_{A'})$.
Therefore, we get
\begin{equation}\label{eq:almost a,b in Gamma}
\text{$\rho_{E,A}^{\mathcal{F}}(a_{A}) =  \rho_{E,A'}^{\mathcal{F}}(a_{A'})$ and $\rho_{E,A}^{\mathcal{F}}(b_{A}) =  \rho_{E,A'}^{\mathcal{F}}(b_{A'})$ for any edge $E\ne Es$.}
\end{equation}

We have $\rho_{E,A}^{\theta_{s}\mathcal{F}}(a_{A},b_{As}) = \rho_{E,As}^{\theta_{s}\mathcal{F}}(a_{As},b_{A})$, hence there exists $(c_{A},c_{As})\in (\theta_{s}\mathcal{F})^{A} = \Gamma(\{A,As\},\mathcal{F})$ such that $(a_{A} - b_{A},b_{As} - a_{As}) = (c_{A},c_{As})\widetilde{\alpha}_{s}^{\vee}$.
In particular, as $\mathcal{F}^{E}\widetilde{\alpha}_{s}^{\vee} = 0$. we have $\rho_{E,A}^{\mathcal{F}}(a_{A}) = \rho_{E,A}^{\mathcal{F}}(b_{A})$ and $\rho_{E,As}^{\mathcal{F}}(a_{As}) = \rho_{E,As}^{\mathcal{F}}(b_{As})$.
Since $(a_{A},b_{As})\in \Gamma(\{A,As\},\mathcal{F})$, we have $\rho^{\mathcal{F}}_{E,A}(a_{A}) = \rho^{\mathcal{F}}_{E,As}(b_{As})$.
Hence $\rho^{\mathcal{F}}_{E,A}(a_{A}) = \rho^{\mathcal{F}}_{E,As}(a_{As})$ and $\rho^{\mathcal{F}}_{E,A}(b_{A}) = \rho^{\mathcal{F}}_{E,As}(b_{As})$.
With \eqref{eq:almost a,b in Gamma}, $a = (a_{A})$ and $b = (b_{A})$ are both in $\Gamma(X,\mathcal{F})$.

We have $(a_{A} - b_{A},b_{As} - a_{As}) = (c_{A},c_{As})\widetilde{\alpha}_{s}^{\vee}$ in $(\theta_{s}\mathcal{F})^{A}$.
Therefore, we have $a_{A} - b_{A} = c_{A}\widetilde{\alpha}_{s}^{\vee}$ for any $A$.
We prove $c = (c_{A})\in \Gamma(X,\mathcal{F})$, namely for any edge $E$ connecting $A$ with $A'$, we have $\rho^{\mathcal{F}}_{E,A}(c_{A}) = \rho_{E,A'}^{\mathcal{F}}(c_{A})$.
If $A' = As$, then this follows from $(c_{A},c_{As})\in \Gamma(\{A,As\},\mathcal{F})$.
If $A'\ne As$, then as $a,b\in \Gamma(X,\mathcal{F})$, we have $\rho_{E,A}^{\mathcal{F}}(a_{A} - b_{A}) = \rho_{E,A'}^{\mathcal{F}}(a_{A'} - b_{A'})$.
By Lemma~\ref{lem:GKM condition}, we have $\rho_{E,A}^{\mathcal{F}}(c_{A}) = \rho_{E,A'}^{\mathcal{F}}(c_{A'})$.

Hence the image of $\Gamma(X,\theta_{s}\mathcal{F})$ is contained in $\{(a,b)\in \Gamma(X,\mathcal{F})^{2}\mid a - b \in \Gamma(X,\mathcal{F})\widetilde{\alpha}_{s}^{\vee}\}$.
Reversing the argument in the above, we get the reverse inclusion.
\end{proof}

\begin{lem}\label{lem:freeness of the stalk of translation functor}
If $\mathcal{F}$ satisfies (BM1), (BM2) and (BM4), then $(\theta_{s}\mathcal{F})^{A}$ is graded free.
\end{lem}
\begin{proof}
We may assume $A > As$.
The map $(\theta_{s}\mathcal{F})^{A} = \Gamma(\{A,As\},\mathcal{F})(1)\to \mathcal{F}^{As}(1)$ is surjective since $\Gamma(\mathcal{F})\to \mathcal{F}^{As}$ is surjective.
The kernel of this map is isomorphic to $\Ker\rho_{E_{A,As},A}^{\mathcal{F}} = \mathcal{F}^{A}\widetilde{\alpha}_{s}^{\vee}$.
Therefore, we have an exact sequence $0\to \mathcal{F}^{A}(-1)\xrightarrow{\widetilde{\alpha}_{s}^{\vee}} \Gamma(\{A,As\},\mathcal{F})(1)\to \mathcal{F}^{As}(1)\to 0$ and hence $\Gamma(\{A,As\},\mathcal{F})$ is graded free.
\end{proof}

\begin{prop}\label{prop:translation preserves BM}
If $\mathcal{F}\in\BM(X)$ then $\theta_{s}\mathcal{F}\in \BM(X)$.
\end{prop}
\begin{proof}
In the proof, we ignore the grading shift.
By \eqref{lem:freeness of the stalk of translation functor}, $\theta_{s}\mathcal{F}$ satisfies (BM1).
Let $E$ be an edge connecting $A$ with $A' < A$.
Assume that $A'\ne As$.
By the definition we have $(\theta_{s}\mathcal{F})^{E} = (\Gamma(\{A,As\},\mathcal{F})\oplus \Gamma(\{A',A's\},\mathcal{F}))/\Gamma(\{A,A',As,A's\},\mathcal{F})$.
Hence $(a_{A},a_{As})\in \Gamma(\{A,As\},\mathcal{F})$ is in the kernel of $(\theta_{s}\mathcal{F})^{A} = \Gamma(\{A,As\},\mathcal{F})\to (\theta_{s}\mathcal{F})^{E}$ if and only if $(a_{A},a_{As},0,0)\in \Gamma(\{A,A',As,A's\},\mathcal{F})$.
Namely $\rho_{E,A}^{\mathcal{F}}(a_{A}) = 0,\rho_{Es,As}^{\mathcal{F}}(a_{As}) = 0$.
Therefore $a_{A}\in \Ker\rho_{E,A}^{\mathcal{F}} = \mathcal{F}^{A}\alpha_{E}$.
Hence there exists $b_{A}\in \mathcal{F}^{A}$ such that $a_{A} = b_{A}\alpha_{E}$.
Similarly we have $a_{As} = b_{As}\alpha_{E}$ for some $b_{As}\in \mathcal{F}^{As}$.
Let $E'$ be the edge connecting $A$ with $As$.
Then $\rho^{\mathcal{F}}_{E',A}(a_{A}) = \rho^{\mathcal{F}}_{E',As}(a_{As})$.
By Lemma~\ref{lem:GKM condition}, $\rho^{\mathcal{F}}_{E',A}(b_{A}) = \rho^{\mathcal{F}}_{E',As}(b_{As})$ and we get $(b_{A},b_{As}) \in \Gamma(\{A,As\},\mathcal{F})$.
Hence $(a_{A},a_{As})\in \Gamma(\{A,As\},\mathcal{F})\alpha_{E}$.
Therefore $(\theta_{s}\mathcal{F})^{E}\simeq (\theta_{s}\mathcal{F})^{A}/(\theta_{s}\mathcal{F})^{A}\alpha_{E}$ if $A'\ne As$.
This is also true when $A' = As$ by the definition of $\theta_{s}$.
We get (BM2).

The sheaf $\theta_{s}\mathcal{F}$ is again flabby by \cite[Theorem~7.17]{MR3324922}.
The map $\Gamma(X,\theta_{s}\mathcal{F})\to (\theta_{s}\mathcal{F})^{A}$ is identified with $\Gamma(\mathcal{F})\otimes B_{s}\to (\Gamma(\mathcal{F})\otimes B_{s})^{A}$ by Proposition~\ref{prop:action of BS module is the translation functor} and Lemma~\ref{lem:s-stable global section of translation functor}.
It is surjective by the definition.
\end{proof}

We now prove Theorem~\ref{thm:action of Soergel bimodule preserves BM sheaves}.
We return the general situation, namely $O,O'\in \bbopen$, such that $O(\supp B)\subset O'$ and $\mathcal{F}\in \BM(O')$.
We may assume $B = B_{s_{1}}\otimes\cdots\otimes B_{s_{l}}$ for $s_{1},\ldots,s_{l}\in S_{\aff}$.
We prove $\mathcal{F}\star B\in \BM(O)$ by induction on $l$.
Set $B' = B_{s_{2}}\otimes\cdots\otimes B_{s_{l}}$.
Take $O_{1}\in \bbopen$ such that $O\subset O_{1},O_{1}s_{1} = O_{1},O(\supp B')\subset O_{1}$.
We can also assume that $O'\subset O_{1}$.
There exists a Braden-MacPherson sheaf on $O_{1}$ whose restriction to $O'$ is isomorphic to $\mathcal{F}$.
Hence we may assume $\mathcal{F}\in \BM(O_{1})$.
We have $\Gamma(O_{1},\mathcal{F}\star B_{s_{1}}) = \Gamma(O_{1},\mathcal{F})\otimes B_{s_{1}}$ by Proposition~\ref{prop:action of BS module is the translation functor} and Lemma~\ref{lem:s-stable global section of translation functor}.
Therefore $\Gamma(O_{1},\mathcal{F}\star B_{s_{1}})\otimes B'\simeq \Gamma(O_{1},\mathcal{F})\otimes B$.
By applying the functor $\mathcal{L}$, we get $(\mathcal{F}\star B_{s_{1}})\star B'\simeq \mathcal{F}\star B$.
Since $\mathcal{F}\star B_{s_{1}}$ is a Braden-MacPherson sheaf, the left hand side is a Braden-MacPherson sheaf by inductive hypothesis.
Hence $\mathcal{F}\star B$ is a Braden-MacPherson sheaf.
During the proof, we got the following.
\begin{prop}\label{prop:associativity of the action}
We have $(\mathcal{F}\star B_{1})\star B_{2} \simeq \mathcal{F}\star (B_{1}\otimes B_{2})$.
\end{prop}

The proof of the following is straghtforward.
We omit it.
\begin{thm}
The operator $\star$ defines an action of $\mathcal{S}$ on $\BM(\mathcal{A})$.
Namely, the isomorphism in Proposition~\ref{prop:associativity of the action} satisfies the coherence conditions.
\end{thm}

\subsection{Sheaves on $W_{\aff}$ and $W_{\aff,\lambda}\backslash W_{\aff}$}
In this subsection we fix $\lambda\in \mathbb{X}^{\vee}$.
Let $\pi_{M}\colon W_{\aff}\to W_{\aff,\lambda}\backslash W_{\aff}$ be the natural projection.
Then, this is a morphism between moment graphs.
Therefore we have the push-forward functor $\pi_{M,*}$ \cite[Definition 5.6]{MR3324922} and the pull-back functor $\pi_{M}^{*}$~\cite[Definition~5.7]{MR3324922}.
The pair $(\pi_{M}^{*},\pi_{M,*})$ is an adjoint pair~\cite[Proposition~5.9]{MR3324922}.

On the other hand, we can define the functors $\pi_{B,*}\colon \EB(W_{\aff})\to \EB(W_{\aff,\lambda}\backslash W_{\aff})$ and $\pi_{B}^{*}\colon \EB(W_{\aff,\lambda}\backslash W_{\aff})\to \EB(W_{\aff})$ as follows.
(See \cite{MR4846616} for the details.)
For $M\in \EB(W_{\aff})$, we put $\pi_{B,*}(M) = M|_{((\widehat{R}_{\aff}^{\vee})^{W_{\aff,\lambda}},\widehat{R}_{\aff}^{\vee})}$ and $(\pi_{B,*}M)_{Q}^{x} = \bigoplus_{a\in x}M_{Q}^{a}$ where we regard $x\in W_{\aff,\lambda}\backslash W_{\aff}$ as a subset of $W_{\aff}$.

Before defining $\pi_{B}^{*}$, we give notation.
Let $x\in W_{\aff,\lambda}\backslash W_{\aff}$ and consider the action of $W_{\aff,\lambda}$ on $\prod_{a\in x}Q_{\aff}^{\vee}$ defined by $w(f_{a}) = (w(f_{w^{-1}a}))$.
Let $(\prod_{a\in x}Q_{\aff}^{\vee})^{W_{\aff,\lambda}}$ be the algebra of $W_{\aff,\lambda}$-fixed points.
Then we have a map $(Q_{\aff}^{\vee})^{W_{\aff,\lambda}}\otimes_{\widehat{\Coef}}Q_{\aff}^{\vee}\to (\prod_{a\in x}Q_{\aff}^{\vee})^{W_{\aff,\lambda}}$ defined by $f\otimes g\mapsto (fa(g))_{a\in x}$.
This is surjective.
Let $N\in \EB(W_{\aff,\lambda}\backslash W_{\aff})$.
Then $(Q_{\aff}^{\vee})^{W_{\aff,\lambda}}\otimes_{\widehat{\Coef}}Q_{\aff}^{\vee}$ acts on $N_{Q}^{x}$ and this defines an action of $(\prod_{a\in x}Q_{\aff}^{\vee})^{W_{\aff,\lambda}}$.

Now we define $\pi_{B}^{*}$.
First we put $\pi_{B}^{*}N = \widehat{R}_{\aff}^{\vee}\otimes_{(\widehat{R}_{\aff}^{\vee})^{W_{\aff,\lambda}}}N$.
We have $(\widehat{R}_{\aff}^{\vee}\otimes_{(\widehat{R}_{\aff}^{\vee})^{W_{\aff,\lambda}}}N)\otimes_{\widehat{R}_{\aff}^{\vee}}Q_{\aff}^{\vee}\simeq \bigoplus_{x\in W_{\aff,\lambda}\backslash W_{\aff}}Q_{\aff}^{\vee}\otimes_{(Q_{\aff}^{\vee})^{W_{\aff,\lambda}}}N_{Q}^{x}$.
The module $Q_{\aff}^{\vee}\otimes_{(Q_{\aff}^{\vee})^{W_{\aff,\lambda}}}N_{Q}^{x}$ is a $Q_{\aff}^{\vee}\otimes_{(Q_{\aff}^{\vee})^{W_{\aff,\lambda}}}(\prod_{a\in x}Q_{\aff}^{\vee})^{W_{\aff,\lambda}}$-module and the latter is isomorphic to $\prod_{a\in x}Q_{\aff}^{\vee}$~\cite[Lemma~2.21]{MR4846616}.
Therefore it has a decomposition $Q_{\aff}^{\vee}\otimes_{(Q_{\aff}^{\vee})^{W_{\aff,\lambda}}}N_{Q}^{x} = \prod_{a\in x}(\pi_{B}^{*}N)_{Q}^{x}$.
It is not difficult to see that $(\pi_{B}^{*},\pi_{B,*})$ is an adjoint pair, see \cite[Proposition~2.27]{MR4846616}.
Note that $(\pi_{B}^{*}N)_{Q}^{x}\simeq N_{Q}^{x}$ as a right $Q_{\aff}^{\vee}$-module.

\begin{lem}\label{lem:push, pull and global section}
We have $\pi_{B}^{*}(\Gamma(\mathcal{F}))\simeq \Gamma(\pi_{M}^{*}\mathcal{F})$ and $\pi_{B,*}(\Gamma(\mathcal{F}))\simeq \Gamma(\pi_{M,*}\mathcal{F})$ if $\mathcal{F}$ satisfies (BM1) and (BM2).
\end{lem}
\begin{proof}
Let $\{*\}$ be the trivial moment graph consisting of one point and $q\colon W_{\aff}\to \{*\}$ (resp.\ $q_{0}\colon W_{\aff,\lambda}\backslash W_{\aff}\to \{*\}$) be the unique map.
Then we have $\Gamma\circ\pi_{M,*} = q_{0,*}\circ \pi_{M,*}\simeq q_{*} \simeq \Gamma$ as right $\widehat{R}_{\aff}^{\vee}$-modules.
By the construction of this isomorphism, this implies $\pi_{B,*}\circ \Gamma\simeq \Gamma\circ\pi_{M,*}$.

We have $\pi_{B,*}\circ \Gamma\circ\pi_{M}^{*}\simeq \Gamma\circ \pi_{M,*}\circ \pi_{M}^{*}$ and we have a map $\id\to \pi_{M,*}\circ\pi_{M}^{*}$ by the adjointness.
Hence we have a map $\Gamma\to \pi_{B,*}\circ\Gamma\circ\pi_{M}^{*}$ and, by the adjointness, it induces $\pi_{B}^{*}\circ\Gamma\to \Gamma\circ\pi_{M}^{*}$.
We prove that this is an isomorphism.
By the definition of the global sections, we have the following diagram with horizontal exact sequences (note that $\Gamma\circ\pi_{M,*} = \Gamma$):
\[
\begin{tikzcd}
0\arrow[r] & \widehat{R}_{\aff}^{\vee}\otimes_{(\widehat{R}_{\aff}^{\vee})^{W_{\aff,\lambda}}}\Gamma(\mathcal{F})\arrow[r]\arrow[d] &\bigoplus_{x}\widehat{R}_{\aff}^{\vee}\otimes_{(\widehat{R}_{\aff}^{\vee})^{W_{\aff,\lambda}}}\mathcal{F}^{x} \arrow[r]\arrow[d] & \bigoplus_{E}\widehat{R}_{\aff}^{\vee}\otimes_{(\widehat{R}_{\aff}^{\vee})^{W_{\aff,\lambda}}}\mathcal{F}^{E}\arrow[d]  \\
0\arrow[r] & \Gamma(\pi_{M,*}\pi_{M}^{*}\mathcal{F}) \arrow[r] &\bigoplus_{x}(\pi_{M,*}\pi_{M}^{*}\mathcal{F})^{x} \arrow[r] & \bigoplus_{E}(\pi_{M,*}\pi_{M}^{*}\mathcal{F})^{E}. 
\end{tikzcd}
\]
Here, $x$ (resp.\ $E$) runs through vertices (resp.\ edges) in $W_{\aff,\lambda}\backslash W_{\aff}$.
The first horizontal sequence is exact since $\widehat{R}_{\aff}^{\vee}$ is free as $(\widehat{R}_{\aff}^{\vee})^{W_{\aff,\lambda}}$-module (since $p$ is not a torsion prime).

Therefore, it is sufficient to prove that $\widehat{R}_{\aff}^{\vee}\otimes_{(\widehat{R}_{\aff}^{\vee})^{W_{\aff,\lambda}}}\mathcal{F}^{x}\simeq (\pi_{M,*}\pi_{M}^{*}\mathcal{F})^{x}$ for any $x\in W_{\aff,\lambda}\backslash W_{\aff}$ and $\widehat{R}_{\aff}^{\vee}\otimes_{(\widehat{R}_{\aff}^{\vee})^{W_{\aff,\lambda}}}\mathcal{F}^{E}\to (\pi_{M,*}\pi_{M}^{*}\mathcal{F})^{E}$ is injective for any edge $E$ in $W_{\aff,\lambda}\backslash W_{\aff}$.

Let $x\in W_{\aff,\lambda}\backslash W_{\aff}$ and we regard $x\subset W_{\aff}$.
We have $(\pi_{M,*}\pi_{M}^{*}\mathcal{F})^{x} = \Gamma(x,\pi_{M}^{*}\mathcal{F})$.
Let $\mathcal{Z}(x)$ be the structure sheaf on $x\subset W_{\aff}$.
For each $a\in x$, $(\pi_{M}^{*}\mathcal{F})^{a} = \mathcal{F}^{\pi_{M}(a)} = \mathcal{F}^{x}$ is constant on $x$.
Since $\mathcal{F}^{x}$ is a flat $\widehat{R}_{\aff}^{\vee}$-module, we have $\Gamma(x,\pi_{M}^{*}\mathcal{F})\simeq \Gamma(\mathcal{Z}(x))\otimes_{\widehat{R}_{\aff}^{\vee}}\mathcal{F}^{x}$.
By \cite[Lemma~2.20]{MR4620135}, $\Gamma(\mathcal{Z}(x))\simeq \widehat{R}_{\aff}^{\vee}\otimes_{(\widehat{R}_{\aff}^{\vee})^W_{\aff,\lambda}}\widehat{R}_{\aff}^{\vee}$.
We get the claim for $x$.

Let $E$ be an edge connecting $x$ with $y < x$.
We prove $\widehat{R}_{\aff}^{\vee}\otimes_{(\widehat{R}_{\aff}^{\vee})^{W_{\aff,\lambda}}}\mathcal{F}^{E}\to \bigoplus_{\pi_{M}(E') = E}\mathcal{F}^{E}$ is injective.
As $\mathcal{F}^{E}\simeq \mathcal{F}^{x}/\mathcal{F}^{x}\alpha_{E}$ and $\mathcal{F}^{x}$ is graded free, it is sufficient to prove $\widehat{R}_{\aff}^{\vee}\otimes_{(\widehat{R}_{\aff}^{\vee})^{W_{\aff,\lambda}}}\widehat{R}_{\aff}^{\vee}/(\alpha_{E})\to \bigoplus_{\pi_{M}(E') = E}\widehat{R}_{\aff}^{\vee}/(\alpha_{E})$ is injective.
Note that we have $\widehat{R}_{\aff}^{\vee}\otimes_{(\widehat{R}_{\aff}^{\vee})^{W_{\aff,\lambda}}}\widehat{R}_{\aff}^{\vee}\simeq \Gamma(\mathcal{Z}(x))$ and hence the map is $\Gamma(\mathcal{Z}(x))/\Gamma(\mathcal{Z}(x))\alpha_{E}\to \bigoplus_{\pi_{M}(E') = E}\widehat{R}_{\aff}^{\vee}/(\alpha_{E})$.
Explicitly, this is given as follows: Let $(z_{a})_{a\in x}\in \Gamma(\mathcal{Z}(x))$.
We take $t\in W_{\aff}$ such that $y = xt$.
For each $a\in x$, $at\in y$.
Let $E_{a}$ be the edge connecting $a$ with $at$.
Then the image of $(z_{a})$ is $(z_{a}\pmod{\alpha_{E}})_{E_{a}}$.
Therefore $(z_{a})_{a\in x}$ is in the kernel of $\Gamma(\mathcal{Z}(x))/\Gamma(\mathcal{Z}(x))\alpha_{E}\to \bigoplus_{\pi_{M}(E') = E}\widehat{R}_{\aff}^{\vee}/(\alpha_{E})$ if and only if $z_{a} = z'_{a}\alpha_{E}$ for some $z'_{a}\in \widehat{R}_{\aff}^{\vee}$.
For an edge $F$ in $x$ connecting $a$ with $a'$, since $(z_{a})\in \Gamma(\mathcal{Z}(x))$, we have $z_{a}\equiv z_{a'}\pmod{\alpha_{F}}$.
By the GKM condition, $\alpha_{E_{a}} = \alpha_{E}$ and $\alpha_{F}$ are linearly independent.
Hence $z'_{a}\equiv z'_{a'}\pmod{\alpha_{F}}$.
Namely $(z'_{a})\in \Gamma(\mathcal{Z}(x))$.
Hence $(z_{a})\in \Gamma(\mathcal{Z}(x))\alpha_{E}$.
\end{proof}

\begin{lem}\label{lem:pull and localization}
We have $\mathcal{L}\pi_{B}^{*}M\simeq \pi_{M}^{*}\mathcal{L}M$ for $M\in \mathcal{EB}(W_{\aff,\lambda}\backslash W_{\aff})$.
\end{lem}
\begin{proof}
Let $x\in W_{\aff}$ and we prove $(\mathcal{L}\pi_{B}^{*}M)^{x}\simeq (\pi_{M}^{*}\mathcal{L}M)^{x}$.
By the definitions, this means $(\pi_{B}^{*}M)^{x} \simeq M^{\pi_{M}(x)}$.
We calculate the left hand side.
This is the image of $\pi_{B}^{*}M\to (\pi_{B}^{*}M)_{Q}^{x}$.
We have $\pi_{B}^{*}M = \widehat{R}_{\aff}^{\vee}\otimes_{(\widehat{R}_{\aff}^{\vee})^{W_{\aff,\lambda}}}M$ and $(\pi_{B}^{*}M)_{Q}^{x} = M_{Q}^{\pi_{M}(x)}$.
The map $\pi_{B}^{*}M\to (\pi_{B}^{*}M)_{Q}^{x}$ is given by $f\otimes m\mapsto mx^{-1}(f)$.
Hence the image is the same as $M\to M_{Q}^{\pi_{M}(x)}$, which is $M^{\pi_{M}(x)}$ by the definition.
Hence we get the lemma.
A similar argument gives the identification of the modules attached to edges.
\end{proof}

The functor $\Gamma$ induces an equivalence of categories $\mathcal{S}\simeq \BM(W_{\aff})$~\cite{MR4321542}.
Let ${}_{\lambda}\mathcal{S}$ be the category of singular Soergel bimodules attached to $W_{\aff,\lambda}\backslash W_{\aff}$ introduced in \cite{MR4620135}.
By an analogue of the proof of $\mathcal{S}\simeq \BM(W_{\aff})$, we also have ${}_{\lambda}\mathcal{S}\simeq \BM(W_{\aff,\lambda}\backslash W_{\aff})$.

\begin{lem}\label{lem:*B commutes with Gamma}
For $\mathcal{F}\in \BM(W_{\aff})$ and $B\in \mathcal{S}$, we have $\Gamma(\mathcal{F}\star B)\simeq \Gamma(\mathcal{F})\otimes B$ in $\mathcal{S}$.
\end{lem}
\begin{proof}
By the definition of $\mathcal{F}\star B$, we have $\mathcal{F}\star B = \mathcal{L}(\Gamma(\mathcal{F})\otimes B)$.
In general, for $M\in \mathcal{EB}(W_{\aff})$, we have a natural morphism $M\to \Gamma(\mathcal{L}(M))$ defined by the natural projection $M\to M^{x}$ for $x\in W_{\aff}$.
Hence we have a natural map $\Gamma(\mathcal{F})\otimes B\to \Gamma(\mathcal{F}\star B)$.
This is an isomorphism when $B = B_{s}\ (s\in S_{\aff})$ by Proposition~\ref{prop:action of BS module is the translation functor} and Lemma~\ref{lem:s-stable global section of translation functor}.
Hence it is an isomorphism in general.
\end{proof}

Indecomposable objects in $\mathcal{S}$ (resp.\ ${}_{\lambda}\mathcal{S}$) are parametrized by $W_{\aff}\times \Z$ (resp.\ $(W_{\aff,\lambda}\backslash W_{\aff})\times \Z)$).
For $w\in W_{\aff}$ (resp.\ $w\in W_{\aff,\lambda}\backslash W_{\aff}$), we have the indecomposable object $B(w)\in \mathcal{S}$ (resp.\ ${}_{\lambda}B(w)\in {}_{\lambda}\mathcal{S}$) which is characterized by $\supp_{W_{\aff}}(B(w))\subset \{y\in W_{\aff}\mid y\le w\}$, $B(w)^{w}\simeq \widehat{R}_{\aff}^{\vee}(\ell(w))$ (resp.\ $\supp_{W_{\aff,\lambda}\backslash W_{\aff}}({}_{\lambda}B(w))\subset \{W_{\aff,\lambda}y\mid y\in {}^{\lambda}W_{\aff}, y\le w\}$, ${}_{\lambda}B(w)^{w}\simeq \widehat{R}_{\aff}^{\vee}(\ell(w))$, here the length of $w$ is the length of the minimal representative of $w$).
Then we have 
\begin{equation}\label{eq:comparison of indecomposables}
\text{$B(w)\simeq \Gamma(\mathcal{B}^{W_{\aff}}(w))(\ell(w))$ and ${}_{\lambda}B(w)\simeq \Gamma(\mathcal{B}^{W_{\aff,\lambda}\backslash W_{\aff}}(w))(\ell(w))$}
\end{equation}
As in \cite{MR4620135}, the functor $\pi_{B,*}$ (resp.\ $\pi_{B}^{*}$) sends ${}_{\lambda}\mathcal{S}$ to $\mathcal{S}$ (resp.\ $\mathcal{S}$ to ${}_{\lambda}\mathcal{S}$).
It is compatible with $\pi_{M,*}$ and $\pi_{M}^{*}$ by Lemma~\ref{lem:push, pull and global section}.
In particular, from \cite[Theorem~2.26]{MR4620135}, we have the following.
\begin{lem}\label{lem:pull-back for parabolic Soergel bimodules}
Let $x\in W_{\aff,\lambda}\backslash W_{\aff}$ and $w\in x$ the maximal element.
Then $\pi_{M}^{*}\mathcal{B}^{W_{\aff,\lambda}\backslash W_{\aff}}(x)\simeq \mathcal{B}^{W_{\aff}}(w)$.
\end{lem}

We often use the identification ${}^{\lambda}W_{\aff}\simeq W_{\aff,\lambda}\backslash W_{\aff}$.
In particular, for $w\in {}^{\lambda}W_{\aff}$ we write $\mathcal{B}^{W_{\aff,\lambda}\backslash W_{\aff}}(w)$.
With this notation, Lemma~\ref{lem:pull-back for parabolic Soergel bimodules} is $\pi_{B}^{*}\mathcal{B}^{W_{\aff,\lambda}\backslash W_{\aff}}(w)\simeq \mathcal{B}^{W_{\aff}}(w_{\lambda}w)$.

\subsection{A result of Lanini}\label{subsec:A result of Lanini}
We have $\mathcal{A}^{+} = A_{0}^{+}({}^{0}W_{\aff})$ (Lemma~\ref{lem:length of dominant translation of Weyl group elements}).
Let $\varphi\colon \mathcal{A}\to W_{\aff}$ be the inverse of $w\mapsto A_{0}^{+}w$ and $i\colon \mathcal{A}^{+}\hookrightarrow \mathcal{A}$ the inclusion map.
Define a map $\varphi_{+}$ by
\begin{equation}\label{eq:maps between alcoves}
\begin{tikzcd}
\mathcal{A} \arrow[r,"\varphi"] & W_{\aff} \arrow[r,"\pi_{M}"] & W_{\aff,0}\backslash W_{\aff} \\
\mathcal{A}^{+} \arrow[u,"i"]\arrow[rr]\arrow[rru,"\varphi_{+}"] && {}^{0}W_{\aff}.\arrow[u,"\sim" sloped]
\end{tikzcd}
\end{equation}
Note that $\varphi_{+}$ gives an isomorphism between partially ordered sets, but $W_{\aff,0}\backslash W_{\aff}$ has more edges than $\mathcal{A}^{+}$.
\begin{lem}\label{lem:translation and varphi^*}
For $s\in S_{\aff}$, $\varphi^{*}\theta_{s}\simeq \theta_{s}\varphi^{*}$.
\end{lem}
\begin{proof}
We do not use the order for the definition of $\theta_{s}$.
Since $\varphi$ is an isomorphism between labeled graph, we get the lemma.
\end{proof}

In general, for a sheaf on a moment graph $\mathcal{V}$, we define $\supp^{+}\mathcal{F}\subset \mathcal{V}$ as the set of $x\in \mathcal{V}$ such that $\mathcal{F}^{x}\ne 0$ or there exists an edge $E$ and $y\in \mathcal{V}$ such that $E$ connects $x$ with $y$, $\mathcal{F}^{E}\ne 0$ and $\mathcal{F}^{y}\ne 0$.
Note that if $\mathcal{F}$ satisfies (BM2) and $\supp\mathcal{F}\subset \{y\in \mathcal{V}\mid y\le x\}$, $\supp^{+}\mathcal{F}\subset \{y\in \mathcal{V}\mid y\le x\}$.
In particular, $\supp^{+}\mathcal{B}^{\mathcal{V}}(x)\subset\{y\in \mathcal{V}\mid y\le x\}$.
\begin{lem}\label{lem:restriction of flabby, by supp^+}
Let $\mathcal{V}$ be a moment graph, $\mathcal{F}$ a flabby sheaf on $\mathcal{V}$ and $X\subset \mathcal{V}$ a subset.
Assume that $\supp^{+}\varphi\subset X$.
Then $\mathcal{F}|_{X}$ is also flabby.
\end{lem}
\begin{proof}
Let $U\subset \mathcal{V}$ be an open subset and $a = (a_{x})\in \Gamma(X\cap U,\mathcal{F})$.
Define $b = (b_{x})\in \prod_{x\in U}\mathcal{F}^{x}$ as follows: if $x \in X$, then $b_{x} = a_{x}$, otherwise $b_{x} = 0$.
We prove $b \in \Gamma(U,\mathcal{F})$.
Let $E$ be an edge connecting $x\in U$ with $y\in U$ and we check $\rho_{E,x}^{\mathcal{F}}(b_{x}) = \rho_{E,y}^{\mathcal{F}}(b_{y})$.
If $x,y\in X$, then this follows from $a \in \Gamma(X\cap U,\mathcal{F})$.
If $x,y\notin X$, then this follows from $b_{x} = b_{y} = 0$.
Assume that $x\in X$ and $y\notin X$.
Then $y\notin \supp^{+}\mathcal{F}$.
Hence $\mathcal{F}^{E} = 0$ or $\mathcal{F}^{x} = 0$.
This implies $\rho_{E,x}^{\mathcal{F}}(b_{x}) = 0 = \rho_{E,y}^{\mathcal{F}}(b_{y})$ (since $b_{y} = 0$).
Hence $b\in \Gamma(U,\mathcal{F})$.
Therefore we have an extension $\widetilde{b}\in \Gamma(\mathcal{V},\mathcal{F})$ since $\mathcal{F}$ is flabby.
Then $\widetilde{b}|_{X}$ gives an extension of $a$.
\end{proof}

\begin{lem}
Let $\mathcal{F}$ be a sheaf on $W_{\aff,0}\backslash W_{\aff}$.
Then $\supp^{+}\pi_{M}^{*}\mathcal{F} = \pi_{M}^{-1}(\supp^{+}\mathcal{F})$.
\end{lem}
\begin{proof}
Let $x\in \supp^{+}\pi_{M}^{*}\mathcal{F}$.
If $(\pi_{M}^{*}\mathcal{F})^{x} \ne 0$, then $\mathcal{F}^{\pi_{M}(x)}\ne 0$, hence $\pi_{M}(x)\in \supp^{+}\mathcal{F}$.
Otherwise, we have $\mathcal{F}^{\pi_{M}(x)} = 0$.
We also have that there exist an edge $E$ and $y\in W_{\aff}$ such that $E$ connects $x$ with $y$, $(\pi_{M}^{*}\mathcal{F})^{E}\ne 0$ and $(\pi_{M}^{*}\mathcal{F})^{y}\ne 0$.
If $\pi_{M}(x) \ne \pi_{M}(y)$, then $\pi_{M}(E)$ connects $\pi_{M}(x)$ with $\pi_{M}(y)$, $\mathcal{F}^{\pi_{M}(E)} = (\pi_{M}^{*}\mathcal{F})^{E}\ne 0$ and $\mathcal{F}^{\pi_{M}(y)} = (\pi_{M}^{*}\mathcal{F})^{y}\ne 0$.
Hence $\pi_{M}(x)\in \supp^{+}\mathcal{F}$.
If $\pi_{M}(x) = \pi_{M}(y)$, then $(\pi_{M}^{*}\mathcal{F})^{E} = \mathcal{F}^{\pi_{M}(x)}/\alpha_{E}\mathcal{F}^{\pi_{M}(x)} = 0$ as $\mathcal{F}^{\pi_{M}(x)} = 0$.
This is a contradiction.

On the other hand, assume that $\pi_{M}(x)\in \supp^{+}\mathcal{F}$.
If $\mathcal{F}^{\pi_{M}(x)}\ne 0$, then $(\pi_{M}^{*}\mathcal{F})^{x}\ne 0$.
Otherwise, there exists an edge $E'$ in $W_{\aff,0}\backslash W_{\aff}$ and $y'\in W_{\aff,0}\backslash W_{\aff}$ such that $E'$ connects $\pi_{M}(x)$ with $y'$, $\mathcal{F}^{E'}\ne 0$ and $\mathcal{F}^{y'}\ne 0$.
Take $y\in \pi_{M}^{-1}(y')$ and an edge $E$ in $W_{\aff}$ such that $E$ connects $x$ with $y$.
Then $\pi_{M}(E) = E'$.
We have $(\pi_{M}^{*}\mathcal{F})^{E} = \mathcal{F}^{\pi_{M}(E)}\ne 0$ and $(\pi_{M}^{*}\mathcal{F})^{y} = \mathcal{F}^{\pi_{M}(y)}\ne 0$.
Hence $x\in \supp^{+}\pi_{M}^{*}\mathcal{F}$.
\end{proof}

We have a bijection ${}^{0}W_{\aff}\simeq W_{\aff,0}\backslash W_{\aff}$ and therefore we can regard $\supp^{+}\mathcal{F}\subset {}^{0}W_{\aff}$ for a sheaf $\mathcal{F}$ on $W_{\aff,0}\backslash W_{\aff}$.
Under this identification, we have $\pi_{M}^{-1}(\supp^{+}\mathcal{F}) = W_{\aff,0}\supp^{+}\mathcal{F}$.
We also have $\supp^{+}\varphi^{*}\mathcal{G} = A_{0}^{+}\supp^{+}\mathcal{G}$ for a sheaf $\mathcal{G}$ on $W_{\aff}$ and $A_{0}^{+}W_{\aff,0} = W_{\finite}A_{0}^{+}$.
Hence 
\begin{equation}\label{eq:supp^+ of pull-backs}
\supp^{+}\varphi^{*}\pi_{M}^{*}\mathcal{F} = A_{0}^{+}(\supp^{+}\pi_{M}^{*}\mathcal{F}) = W_{\finite}A_{0}^{+}\supp^{+}\mathcal{F}.
\end{equation}

The following result is essentially proved by Lanini~\cite{MR3324922}.
\begin{prop}\label{prop:result of Lanini}
Let $\mathcal{F}\in \BM(W_{\aff,0}\backslash W_{\aff})$ and $O\in \bbopen$ such that $O\cap \supp^{+}\varphi^{*}\pi_{M}^{*}\mathcal{F}\subset \mathcal{A}^{+}$.
Then $(\varphi_{+}^{*}\mathcal{F})|_{O\cap \mathcal{A}^{+}}$ is a Braden-MacPherson sheaf.
\end{prop}
\begin{proof}
Since $\varphi_{+}$ is a bijection between partially ordered sets (Lemma~\ref{lem:Bruhat order is generic Bruhat order on dominants}), $\varphi_{+}^{*}\mathcal{F}$ satisfies (BM1), (BM2).
The moment graph $\mathcal{A}^{+}$ has less edges than $W_{\aff,0}\backslash W_{\aff}$.
Hence $\Gamma(\varphi_{+}^{*}\mathcal{F}|_{O\cap \mathcal{A}^{+}})\supset \Gamma(\mathcal{F}|_{\varphi_{+}(O\cap \mathcal{A}^{+})})$.
This implies (BM4).

We prove that $\varphi^{*}\pi_{M}^{*}\mathcal{F}$ is flabby.
More generally, we prove $\varphi^{*}\mathcal{G}$ is flabby for any $\mathcal{G}\in \BM(W_{\aff})$.
Let $\mathcal{G}_{0}$ be the sheaf on $W_{\aff}$ defined as follows: $\mathcal{G}_{0}^{1} = \widehat{R}_{\aff}^{\vee}$, $\mathcal{G}_{0}^{x} = 0$ for $x\in W_{\aff}$, $x\ne 1$ and $\mathcal{G}_{0}^{E} = 0$ for any edge $E$.
Then any $\mathcal{G}$ is a direct summand of $\theta_{s_{1}}\cdots \theta_{s_{l}}\mathcal{G}_{0}$ for some $s_{1},\ldots,s_{l}\in S_{\aff}$.
Hence $\varphi^{*}\mathcal{G}$ is a direct summand of $\theta_{s_{1}}\cdots \theta_{s_{l}}\varphi^{*}\mathcal{G}_{0}$ by Lemma~\ref{lem:translation and varphi^*}.
Since the sheaf $\varphi^{*}\mathcal{G}_{0}$ is flabby and since $\theta_{s}$ preserves flabby sheaves~\cite[Proposition~7.16]{MR3324922}, $\varphi^{*}\mathcal{G}$ is also flabby.
Hence $\varphi^{*}\pi_{M}^{*}\mathcal{F}$ is flabby and $\varphi^{*}\pi_{M}^{*}\mathcal{F}|_{O}$ is also flabby.
By the assumption, $\supp^{+}\varphi^{*}\pi_{M}^{*}\mathcal{F}|_{O}\subset \mathcal{A}^{+}$.
Therefore $(\varphi^{*}\pi_{M}^{*}\mathcal{F}|_{O})|_{\mathcal{A}^{+}} = i^{*}\varphi^{*}\pi_{M}^{*}\mathcal{F}|_{O\cap \mathcal{A}^{+}}$ is also flabby by Lemma~\ref{lem:restriction of flabby, by supp^+}.
We have $i^{*}\varphi^{*}\pi_{M}^{*} = \varphi_{+}^{*}$.
\end{proof}

In this paper, we say that $\lambda\in \mathbb{X}^{\vee}$ is sufficiently dominant if $\langle \alpha,\lambda\rangle$ is sufficiently large for any $\alpha\in(\Phi')^{+}$.

\begin{lem}\label{lem:dominant lambda for Lanini}
Let $A_{1},A_{2}\in \mathcal{A}$.
\begin{enumerate}
\item For sufficiently dominant $\lambda\in \mathbb{X}^{\vee}$, we have $W_{\finite}(A_{1} + \lambda)\cap\{A\in\mathcal{A}\mid A\ge A_{2} + \lambda\} \subset \{A_{1} + \lambda\}$.
\item For sufficiently dominant $\lambda\in \mathbb{X}^{\vee}$, we have $\{A\in\mathcal{A}\mid A\ge A_{2} + \lambda\}\cap W_{\finite}\{A\mid A\in\mathcal{A}^{+},A\le A_{1} + \lambda\}\subset\mathcal{A}^{+}$.
\end{enumerate}
\end{lem}
\begin{proof}
Let $x\in W_{\finite}$ be a non-unit element.
Take $w\in W'_{\aff}$ such that $A_{2} = wA_{1}$, fix $a_{1}\in A_{1}$ and put $a_{2} = w(a_{1})$.
Then for sufficiently dominant $\lambda$, we have $(x(\lambda) - \lambda) - (a_{2} - x(a_{1}))\notin \R_{\ge 0}(\Phi')^{+}$.
Hence $x(a_{1} + \lambda)\notin a_{2} + \lambda + \R_{\ge 0}(\Phi')^{+}$.
Therefore $x(A_{1} + \lambda)\not\ge A_{2} + \lambda$ by Lemma~\ref{lem:order of alcove and weights}.
This implies (1).

We prove (2).
Use (1) to take sufficiently dominant $\lambda$ such that for any $A_{0}\in \mathcal{A}$ with $A_{1}\ge A_{0}\ge A_{2}$, we have $W_{\finite}(A_{0} + \lambda) \cap \{A'\mid A'\ge A_{2} + \lambda\} \subset \{A_{0} + \lambda\}$.
Let $A\in \mathcal{A}^{+}$ and $x\in W_{\finite}$ such that $A\le A_{1} + \lambda$ and $xA\ge A_{2} + \lambda$.
(Namely $xA$ is in the left hand side of (2).)
We prove $xA = A$.
Since $A\in \mathcal{A}^{+}$, $A\ge xA$ by Lemma~\ref{lem:finite weyl group orbit of dominant alcove, order}.
Hence $A\ge A_{2} + \lambda$.
Applying the above condition to $A_{0} = A - \lambda$, we have $W_{\finite}A\cap \{A'\mid A'\ge A_{2} + \lambda\}\subset \{A\}$.
Since $xA$ is in the left hand side, we have $xA = A$.
\end{proof}

\begin{lem}\label{lem:restriction to subset of supp^+}
Let $\mathcal{V}$ be a moment graph, $X\subset \mathcal{V}$ a subset and $\mathcal{F}$ a Braden-MacPherson sheaf on $\mathcal{V}$ such that $\supp^{+}\mathcal{F}\subset X$.
Then $\mathcal{F}|_{X}$ is a Braden-MacPherson sheaf.
Moreover, if $\mathcal{F}$ is indecomposable, then $\mathcal{F}|_{X}$ is also indecomposable.
\end{lem}
\begin{proof}
It is obvious that $\mathcal{F}|_{X}$ satisfies (BM1) and (BM2).
For any $x\in X$, we have $\Gamma(\mathcal{F})\to \Gamma(\mathcal{F}|_{X})\to \mathcal{F}^{x}$ and the composition is surjective.
Hence $\Gamma(\mathcal{F}|_{X})\to \mathcal{F}^{x}$ is also surjective.
Therefore $\mathcal{F}$ satisfies (BM4).
The sheaf $\mathcal{F}|_{X}$ is flabby by Lemma~\ref{lem:restriction of flabby, by supp^+}.

Let $e = ((e_{x}),(e_{E}))\in \End(\mathcal{F}|_{X})$ be an idempotent.
Set $e_{x} = 0$ for $x\in \mathcal{V}\setminus X$ and $e_{E} = 0$ for an edge $E$ outside $X$.
Then by a similar argument in the proof of Lemma~\ref{lem:restriction of flabby, by supp^+}, we can prove $((e_{x})_{x\in \mathcal{V}},(e_{E})_{E\subset \mathcal{V}})\in \End(\mathcal{F})$ and obviously it is an idempotent.
We get the last claim of the lemma.
\end{proof}

\begin{cor}\label{cor:Corollary of Lanini}
Let $A\in \mathcal{A}$ and $O\in\bbopen$.
For $\lambda\in \mathbb{X}^{\vee}$, set $O_{\lambda} = O + \lambda$.
Take $x_{\lambda}\in W_{\aff}$ such that $A + \lambda = A_{0}^{+}x_{\lambda}$.
If $\lambda$ is sufficiently dominant, then $\supp^{+} \varphi^{*}\pi_{M}^{*}\mathcal{B}^{W_{\aff,0}\backslash W_{\aff}}(x_{\lambda})|_{O_{\lambda}}\subset \mathcal{A}^{+}$ and $\mathcal{B}^{\mathcal{A}}(A + \lambda)|_{O_{\lambda}\cap \mathcal{A}^{+}}$ is a direct summand of $\varphi_{+}^{*}\mathcal{B}^{W_{\aff,0}\backslash W_{\aff}}(x_{\lambda})|_{O_{\lambda}\cap \mathcal{A}^{+}}$.
\end{cor}
\begin{proof}
Take $A'\in \mathcal{A}$ such that $O\subset \{A''\in\mathcal{A}\mid A''\ge A'\}$.
We may assume $O = \{A''\in\mathcal{A}\mid A''\ge A'\}$.
Set  $\mathcal{F} = \mathcal{B}^{W_{\aff,0}\backslash W_{\aff}}(x_{\lambda})$ for simplifying the notation.
We regard $\supp^{+}\mathcal{F}\subset {}^{0}W_{\aff}\subset W_{\aff}$ by using a bijection ${}^{0}W_{\aff}\simeq W_{\aff,0}\backslash W_{\aff}$.
Since $\supp^{+}\mathcal{F}\subset \{y\in {}^{0}W_{\aff}\mid y\le x_{\lambda}\}$, by Lemma~\ref{lem:Bruhat order is generic Bruhat order on dominants}, we have $A_{0}^{+}\supp^{+}\mathcal{F}\subset \{A'\in \mathcal{A}^{+}\mid A'\le A_{0}^{+}x_{\lambda} = A + \lambda\}$.
By \eqref{eq:supp^+ of pull-backs}, $\supp^{+}\varphi^{*}\pi_{M}^{*}\mathcal{F}\subset W_{\finite}\{A''\in \mathcal{A}^{+}\mid A''\le A + \lambda\}$.
By Lemma~\ref{lem:dominant lambda for Lanini}, for sufficiently dominant $\lambda$, we have $W_{\finite}\{A''\in \mathcal{A}^{+}\mid A''\le A + \lambda\}\cap O_{\lambda}\subset \mathcal{A}^{+}$.
For such $\lambda$, we have $\supp^{+}\varphi^{*}\pi_{M}^{*}\mathcal{F}|_{O_{\lambda}} \subset \mathcal{A}^{+}$, namely the first condition of the corollary hold.
We also have $O_{\lambda}\cap \supp^{+}\varphi^{*}\pi_{M}^{*}\mathcal{F}\subset \mathcal{A}^{+}$.
By Proposition~\ref{prop:result of Lanini}, $\varphi_{+}^{*}\mathcal{F}|_{O_{\lambda}\cap \mathcal{A}^{+}}$ is a Braden-MacPherson sheaf.
Moreover, by Lemma~\ref{lem:Bruhat order is generic Bruhat order on dominants}, $A + \lambda = A_{0}^{+}x_{\lambda}$ is maximal in $\supp\varphi_{+}^{*}\mathcal{F}|_{O_{\lambda}\cap \mathcal{A}^{+}}$.
Hence $\mathcal{B}^{O_{\lambda}\cap \mathcal{A}^{+}}(A + \lambda)$ is a direct summand of $\varphi_{+}^{*}\mathcal{F}|_{O_{\lambda}}\cap \mathcal{A}^{+}$.
Therefore it is sufficient to prove that $\mathcal{B}^{O_{\lambda}\cap \mathcal{A}^{+}}(A + \lambda) = \mathcal{B}^{\mathcal{A}}(A + \lambda)|_{O_{\lambda}\cap \mathcal{A}^{+}}$.
The set $O\cap \{A''\in \mathcal{A}\mid A''\le A\}$ is finite.
Hence for sufficiently dominant $\lambda$, $O_{\lambda}\cap \{A''\in \mathcal{A}\mid A''\le A + \lambda\} = (O\cap \{A''\in \mathcal{A}\mid A''\le A\}) + \lambda$ is contained in $\mathcal{A}^{+}$.
Hence, replacing $\lambda$ with more dominant element, we have $\mathcal{B}^{O_{\lambda}\cap \mathcal{A}^{+}}(A + \lambda) = \mathcal{B}^{\mathcal{A}}(A + \lambda)|_{O_{\lambda}\cap \mathcal{A}^{+}}$ by Lemma~\ref{lem:restriction to subset of supp^+}.
\end{proof}

\subsection{Some calculation of graded ranks}\label{subsec:Some calculation of graded ranks}
By the proof of Lemma~\ref{lem:freeness of the stalk of translation functor}, we have the following.

\begin{lem}\label{lem:rank of stalk of translation}
For $s\in S_{\aff}$, $O\in \bbopen$ such that $Os = O$ and $A\in O$ such that $A > As$, if $\mathcal{F}\in \Sh(O)$ satisfies (BM1), (BM2) and (BM4), then we have an exact sequence
\[
0\to \mathcal{F}^{A}(-1)\xrightarrow{\widetilde{\alpha}_{s}^{\vee}}(\theta_{s}\mathcal{F})^{A} = (\theta_{s}\mathcal{F})^{As} \to \mathcal{F}^{As}(1)\to 0.
\]
\end{lem}

We also have the following for $\mathcal{F}^{[A]}$.

\begin{lem}\label{lem:rank of succ filtration of translation}
Let $O\in \bbopen$, $s\in S_{\aff}$ such that $Os = O$.
Assume that $\mathcal{F}\in \Sh(O)$ has a Verma flag, satisfies (BM1), (BM3) and $\mathcal{F}^{E}$ is a free $\widehat{R}^{\vee}_{\aff}/(\alpha_{E})$-module for any edge $E$.
Let $A\in O$.
Then $\theta_{s}\mathcal{F}$ also has a Verma flag and
\[
\grk((\theta_{s}\mathcal{F})^{[A]})
=
\begin{cases}
v(\grk(\mathcal{F}^{[A]}) + \grk(\mathcal{F}^{[As]})) & (A > As),\\
v^{-1}(\grk(\mathcal{F}^{[A]}) + \grk(\mathcal{F}^{[As]})) & (A < As).
\end{cases}
\]
\end{lem}
\begin{proof}
First, assume that $A > As$.
Let $a = (a_{A},a_{As})\in (\theta_{s}\mathcal{F})^{[A]}$.
For each edge $E$ connecting $A$ with $A' > A$, we have $\rho_{E,A}^{\theta_{s}\mathcal{F}}(a) = 0$.
Since $A'\ne As$, this means $\rho_{E,A}^{\mathcal{F}}(a_{A}) = 0$ and $\rho_{Es,As}^{\mathcal{F}}(a_{As}) = 0$.
In particular we have $a_{A} \in \mathcal{F}^{[A]}$.

We prove $a_{As}\in \mathcal{F}^{[As]}$ when $a_{A} = 0$.
Let $E$ be an edge connecting $As$ with $A' > As$.
If $A' = A$, then since $a\in \Gamma(\{A,As\},\mathcal{F})$, $\rho_{E,As}^{\mathcal{F}}(a_{As}) = \rho_{E,A}^{\mathcal{F}}(a_{A}) = 0$.
Otherwise, we have $A',A's > As$ by \cite[Proposition~3.2]{MR591724}.
In this case, we have $A's > A$.
Indeed, if $A' < A's$, then this follows \cite[Proposition~3.2]{MR591724}.
Otherwise, this follows by applying Lemma~\ref{lem:lemma on a order in alcoves} as $(A,A') = (As,A')$.
The edge $Es$ connects $A$ with $A's$.
By $\rho_{Es,A}^{\theta_{s}(\mathcal{F})}(a) = 0$, we get $\rho_{E,As}^{\mathcal{F}}(a_{As}) = 0$.
Hence we conclude $a_{As}\in \mathcal{F}^{[As]}$.

Therefore we have an exact sequence $0\to \mathcal{F}^{[As]}(1)\to (\theta_{s}\mathcal{F})^{[A]}\to \mathcal{F}^{[A]}(1)$.
We prove that the last map is surjective.
Let $b_{A}\in \mathcal{F}^{[A]}$ and for each $A' > A$ put $b_{A'} = 0$.
Then $b = (b_{A})$ defines an element in $\Gamma(\{A'\mid A'\ge A\},\mathcal{F})$.
Since $\mathcal{F}$ satisfies (BM3), this extends to an element $b$ in $\Gamma(O,\mathcal{F})$.
We prove that the restriction of $b$ to $\{A,As\}$ gives an element in $(\theta_{s}\mathcal{F})^{[A]}$.
Let $E$ be an edge connecting $A$ with $A' > A$.
Then $\rho_{E,A}^{\theta_{s}\mathcal{F}}(b) = (\rho_{E,A}^{\mathcal{F}}(b_{A}),\rho_{Es,As}^{\mathcal{F}}(b_{As}))$.
We have $\rho_{E,A}^{\mathcal{F}}(b_{A}) = 0$ by the assumption.
We prove $A's > A$, which gives $\rho_{Es,As}^{\mathcal{F}}(b_{As}) = 0$ since $b_{A's} = 0$ for such $A'$.
If $A's > A'$, then this follows from \cite[Proposition~3.2]{MR591724}.
If $A's < A'$, then by \cite[Proposition~3.2]{MR591724}, we have $As < A's$.
Hence we have $A's > A$ by applying Lemma~\ref{lem:lemma on a order in alcoves} as $(A,A') = (As,A')$.
Therefore $(\theta_{s}\mathcal{F})^{[A]}\to \mathcal{F}^{[A]}(1)$ is surjective.

We assume $A < As$ and let $a = (a_{A},a_{As})\in (\theta_{s}\mathcal{F})^{[A]}$.
Then for the edge $E$ connecting $A$ with $As$, we have $\rho_{E,A}^{\theta_{s}\mathcal{F}}(a) = 0$.
Hence there exists $b = (b_{A},b_{As})\in (\theta_{s}\mathcal{F})^{A}$ such that $a = b\widetilde{\alpha}_{s}^{\vee}$.
By using $(\theta_{s}\mathcal{F})^{A} = (\theta_{s}\mathcal{F})^{As}$, we regard $b$ as an element in $(\theta_{s}\mathcal{F})^{As}$.
We prove $\rho_{E,As}^{\theta_{s}\mathcal{F}}(b) = 0$ for any edge $E$ connecting $As$ with $A' > As$.
The edge $Es$ connects $A$ with $A's$, and we have $A's > A$ by \cite[Proposition~3.2]{MR591724}.
Hence $0 = \rho_{Es,A}^{\theta_{s}\mathcal{F}}(a) = \rho_{E,As}^{\theta_{s}\mathcal{F}}(b)\widetilde{\alpha}_{s}^{\vee}$.
We have $(\theta_{s}\mathcal{F})^{Es}\subset \mathcal{F}^{E}\oplus \mathcal{F}^{Es}$.
The right hand side is a free $\widehat{R}^{\vee}_{\aff}/(\alpha_{E})$-module, hence by GKM condition, the multiplication by $\widetilde{\alpha}_{s}^{\vee}$ on $\mathcal{F}^{E}\oplus \mathcal{F}^{Es}$ has no non-trivial kernel.
Hence $\rho_{E,As}^{\theta_{s}\mathcal{F}}(b) = 0$.
Therefore $b\in (\theta_{s}\mathcal{F})^{[As]}$.
Hence we have $(\theta_{s}\mathcal{F})^{[A]} = (\theta_{s}\mathcal{F})^{[As]}\widetilde{\alpha}_{s}^{\vee}$ and this implies the lemma.
\end{proof}

Let $\varphi,\varphi_{+},i,\pi_{M}$ be as in \eqref{eq:maps between alcoves}.

\begin{lem}\label{lem:easy property of image of phi^*}
For any $\mathcal{F}\in \BM(W_{\aff})$, $\varphi^{*}\mathcal{F}$ has a Verma flag.
\end{lem}
\begin{proof}
Let $\mathcal{G}\in \BM(W_{\aff})$.
Note that, since $\mathcal{G}$ satisfies (BM1), $\varphi^{*}\mathcal{G}$ also satisfies (BM1).
We also know that $\varphi^{*}\mathcal{G}^{E} = \mathcal{G}^{\varphi(E)}$ is a free $\widehat{R}_{\aff}^{\vee}/\alpha_{E}\widehat{R}_{\aff}^{\vee}$-module since $\mathcal{G}$ satisfies (BM1) and (BM2).
We may assume $\mathcal{G} = \theta_{s_{1}}\cdots \theta_{s_{l}}\mathcal{B}^{W_{\aff}}(1)$ for some $s_{1},\ldots,s_{l}\in S_{\aff}$.
Then $\mathcal{G}$ is flabby by \cite[Theorem~7.17]{MR3324922}.
Now the lemma follows from Lemma~\ref{lem:rank of succ filtration of translation} by induction on $l$.
\end{proof}

\begin{lem}\label{lem:comparison of F^[A]}
Let $\mathcal{V}$ be a moment graph, $\mathcal{F}$ a sheaf on $\mathcal{V}$ and $X\subset \mathcal{V}$ a subset of $\mathcal{V}$.
Assume that $X$ is open or $\supp^{+}\mathcal{F}\subset X$.
Then for $x\in X$, we have $(\mathcal{F}|_{X})^{[x]} = \mathcal{F}^{[x]}$.
\end{lem}
\begin{proof}
We may assume $\mathcal{F}^{x}\ne 0$.
Obviously we have $\mathcal{F}^{[x]}\subset (\mathcal{F}|_{X})^{[x]}$.
Let $m\in (\mathcal{F}|_{X})^{[x]}$ and $E$ an edge connecting $x$ with $y > x$.
We prove $\rho_{E,x}^{\mathcal{F}}(m) = 0$.
If $y\in X$, then $\rho_{E,x}^{\mathcal{F}}(m) = 0$ since $m\in (\mathcal{F}|_{X})^{[x]}$.
Assume that $y\notin X$.
If $X$ is open, this implies $x\notin X$.
This is a contradiction.
Assume that $\supp^{+}\mathcal{F}\subset X$.
If $y\notin X$, then $y\notin \supp^{+}\mathcal{F}$.
Since we have assumed that $\mathcal{F}^{x}\ne 0$, we have $\mathcal{F}^{y} = 0$ or $\mathcal{F}^{E} = 0$.
Hence $\rho_{E,x}^{\mathcal{F}}(m) = 0$.
\end{proof}

\begin{prop}\label{prop:BM on alcoves has a Verma flag}
Any $\mathcal{F}\in \mathcal{BM}(\mathcal{A})$ has a Verma flag.
\end{prop}
\begin{proof}
We may assume $\mathcal{F} = \mathcal{B}^{\mathcal{A}}(A)$ and we prove $\mathcal{B}^{\mathcal{A}}(A)^{[A']}$ is free for $A'\in \mathcal{A}$.
For $\lambda\in \mathbb{X}^{*}$, put $O_{\lambda} = \{A''\in \mathcal{A}\mid A''\ge A' + \lambda\}$.
By Corollary~\ref{cor:Corollary of Lanini}, for sufficiently dominant $\lambda$, $\mathcal{B}^{\mathcal{A}}(A + \lambda)|_{O_{\lambda}\cap \mathcal{A}^{+}}$ is a direct summand of $\varphi_{+}^{*}\mathcal{B}^{W_{\aff,0}\backslash W_{\aff}}(x_{\lambda})|_{O_{\lambda}\cap \mathcal{A}^{+}}$ where $x_{\lambda}\in {}^{0}W_{\aff}$ satisfies $A + \lambda = A_{0}^{+}x_{\lambda}$.
Hence it is sufficient to prove that $(\varphi_{+}^{*}\mathcal{B}^{W_{\aff,0}\backslash W_{\aff}}(x_{\lambda})|_{O_{\lambda}\cap \mathcal{A}^{+}})^{[A' + \lambda]}$ is free.
Since $\varphi_{+}^{*} = i^{*}\varphi^{*}\pi_{M}^{*}$, this is $(\varphi^{*}\pi_{M}^{*}\mathcal{B}^{W_{\aff,0}\backslash W_{\aff}}(x_{\lambda})|_{O_{\lambda}\cap \mathcal{A}^{+}})^{[A' + \lambda]}$.
Set $\mathcal{G} = \varphi^{*}\pi_{M}^{*}\mathcal{B}^{W_{\aff,0}\backslash W_{\aff}}(x_{\lambda})$.
As in Corollary~\ref{cor:Corollary of Lanini}, we can take $\lambda$ such that $O_{\lambda}\cap \supp^{+}\mathcal{G}\subset \mathcal{A}^{+}$.
By Lemma~\ref{lem:comparison of F^[A]}, $\mathcal{G}|_{O_{\lambda}\cap \mathcal{A}^{+}}^{[A' + \lambda]} = \mathcal{G}^{[A' + \lambda]}$.
The proposition follows from Lemma~\ref{lem:easy property of image of phi^*}.
\end{proof}

Let $\mathcal{H} = \bigoplus_{w\in W_{\aff}}\Z[v,v^{-1}]H_{w}$ be the Hecke algebra attached to $W_{\aff}$.
Here, we use the notation of Soergel~\cite[Theorem~2.1]{MR1444322}.
In particular, we have $(H_{s} - v^{-1})(H_{s} + v) = 0$ for $s\in S_{\aff}$.
Let $\widehat{\mathcal{P}}$ be the space of formal sums $\sum_{A\in \mathcal{A}}c_{A}A$ with $c_{A}\in \Z[v,v^{-1}]$ such that $\{A\in\mathcal{A}\mid c_{A}\ne 0\}$ is bounded above.
We define an action of $\mathcal{H}$ on $\widehat{\mathcal{P}}$ by
\[
AH_{s} = 
\begin{cases}
As & (As < A),\\
As + (v^{-1} - v)A & (As > A).
\end{cases}
\]
(Note that the order is reversed from the periodic module in \cite{MR1444322}.)

Let $[\mathcal{S}]$ (resp.\ $[\BM(\mathcal{A})]$) be the split Grothendieck group of $\mathcal{S}$ (resp.\ $\BM(\mathcal{A})$).
For $M\in \mathcal{S}$ or $M\in \BM(\mathcal{A})$, $[M]$ denotes the corresponding element in the split Grothendieck group.
These split Grothendieck groups are $\Z[v,v^{-1}]$-module via $v[M] = [M(1)]$.
Moreover, $[\mathcal{S}]$ is a $\Z[v,v^{-1}]$-algebra by $[B_{1}][B_{2}] = [B_{1}\otimes B_{2}]$ and $[\BM(\mathcal{A})]$ is a right $[\mathcal{S}]$-module by $[\mathcal{F}][B] = [\mathcal{F}\star B]$.
By \cite[Theorem~4.3]{MR4321542}, we have an isomorphism $[\mathcal{S}]\simeq \mathcal{H}$.
The isomorphism is characterized by $[B_{s}]\mapsto H_s + v$.
Hence $[\BM(\mathcal{A})]$ is an $\mathcal{H}$-module.
\begin{thm}\label{thm:BM sheaves and periodic module}
The map $[\BM(\mathcal{A})]\to \widehat{\mathcal{P}}$ defined by
\[
\mathcal{F}\mapsto \sum_{A} v^{-\ell(A)}\grk\mathcal{F}^{[A]}A
\]
is an $\mathcal{H}$-module homomorphism.
\end{thm}
\begin{proof}
By Proposition~\ref{prop:BM on alcoves has a Verma flag}, the map is well-defined and by Lemma~\ref{lem:rank of succ filtration of translation}, it is an $\mathcal{H}$-module homomorphism.
\end{proof}

\section{Stability of homomorphisms}
In this section, we prove the following.
\begin{thm}\label{thm:finiteness}
Let $\mathcal{F},\mathcal{G}\in \BM(\mathcal{A})$ be indecomposable sheaves.
Then $\Hom(\mathcal{F}|_{O},\mathcal{G}|_{O})$ is stable for sufficiently large $O\in \bbopen$, namely there exists $O'\in\bbopen$ such that if $O\supset O'$ then $\Hom(\mathcal{F}|_{O},\mathcal{G}|_{O})\to \Hom(\mathcal{F}|_{O'},\mathcal{G}|_{O'})$ is an isomorphism.
In particular, $\Hom(\mathcal{F},\mathcal{G})$ is finite-dimensional.
\end{thm}
\begin{rem}
The theorem does not claim that $\Hom^{\bullet}(\mathcal{F},\mathcal{G})$ is finitely generated.
\end{rem}

The theorem follows from the following.
\begin{lem}\label{lem:for finiteness}
Let $\mathcal{F}\in \BM(\mathcal{A})$ be an indecomposable sheaf.
\begin{enumerate}
\item There exist $C_{1}\in \R$ and $N_{1} \in \R_{>0}$ such that if the coefficient of $v^{k}$ in $\grk(\mathcal{F}^{A})$ is not zero then $k + \ell(A) \ge C_{1} - \ell(A)/N_{1}$ for any $A\in \mathcal{A}$.
\item There exists $C_{2}\in \R$ such that if a coefficient of $v^{k}$ in $\grk(\mathcal{F}^{[A]})$ is not zero then $k + \ell(A)\le C_{2}$ for any $A\in \mathcal{A}$.
\end{enumerate}
\end{lem}

To prove that Lemma~\ref{lem:for finiteness} implies Theorem~\ref{thm:finiteness}, we need the following lemma.
\begin{lem}\label{lem:bounded above by A, bounded below by the length is finite}
For $k\in\Z$ and $A\in \mathcal{A}$, $\{A'\in \mathcal{A}\mid A'\le A,\ell(A')\ge k\}$ is a finite set.
\end{lem}
\begin{proof}
We may assume $A = \max_{w\in W_{\finite}}wA$.
Assume that $A'\in \mathcal{A}$ satisfies $A'\le A$ and $\ell(A')\ge k$.
Take $w\in W_{\finite}$ and $\lambda\in \Z(\Phi')^{\vee}$ such that $A' = t_{\lambda}wA$.
We can take $n_{\alpha}\in\Z$ for each simple root $\alpha$ such that $\lambda = \sum_{\alpha}n_{\alpha}\alpha^{\vee}$.
We prove that we have only finitely many possibilities of $(n_{\alpha})$.

Let $w_{\finite}\in W_{\finite}$ be the longest element.
Fix $a\in A$.
Since $w_{\finite}A\le A$, by Lemma~\ref{lem:order of alcove and weights}, there exist $r_{\alpha}\in\R_{\ge 0}$ for each simple root $\alpha$ such that $a - w_{\finite}(a) = \sum_{\alpha}r_{\alpha}\alpha^{\vee}$.
The alcove $w_{\finite}A$ is minimal in $W_{\finite}A$.
Hence $A' = t_{\lambda}wA \ge t_{\lambda}w_{\finite}A$.
Therefore $t_{\lambda}w_{\finite}A\le A$.
Hence $a - w_{\finite}(a) - \lambda\in \R_{\ge 0}(\Phi')^{+}$ by Lemma~\ref{lem:order of alcove and weights}.
Therefore $r_{\alpha} - n_{\alpha}\ge 0$.
On the other hand, by $\ell(A')\ge k$, we have $2\langle \rho,\lambda\rangle + \ell(wA)\ge k$.
Set $k' = \min_{w\in W_{\finite}}(k - \ell(wA))$.
Then $2\langle \rho,\lambda\rangle\ge k'$.
Hence $(n_{\alpha})$ is contained in
\[
\left\{(n_{\alpha})\mid n_{\alpha}\in\Z,n_{\alpha}\le r_{\alpha},\langle\rho,\sum_{\alpha}n_{\alpha}\alpha^{\vee}\rangle\ge k'\right\}.
\]
This is a finite set.
\end{proof}

\begin{proof}[Proof of Theorem~\ref{thm:finiteness} assuming Lemma~\ref{lem:for finiteness}]
By Lemma~\ref{lem:graded rank of hom of BM sheaves}, it is sufficient to prove that $\Hom(\mathcal{F}^{A},\mathcal{G}^{[A]})$ is zero except finitely many $A$.
Take $A'\in \mathcal{A}$ such that $\supp\mathcal{F}\cup \supp \mathcal{G}\subset \{A\in\mathcal{A}\mid A\le A'\}$.
Then $\Hom(\mathcal{F}^{A},\mathcal{G}^{[A]})\ne 0$ only when $A\le A'$.
Therefore, by Lemma~\ref{lem:bounded above by A, bounded below by the length is finite}, it is sufficient to prove that $\Hom(\mathcal{F}^{A},\mathcal{G}^{[A]})$ is zero if $\ell(A)$ is sufficiently small.
We take $C_{1},N_{1}$ in Lemma~\ref{lem:for finiteness} for $\mathcal{F}$ and $C_{2}$ for $\mathcal{G}$.
Then $\Hom^{\bullet}(\mathcal{F}^{A},\mathcal{G}^{[A]})$ is a direct sum of $\widehat{R}_{\aff}^{\vee}(k_{2} - k_{1})$ where $k_{1} + \ell(A)\ge C_{1} -\ell(A)/N_{1}$ and $k_{2} + \ell(A)\le C_{2}$.
Hence $k_{2} - k_{1} \le C_{2} - C_{1} + \ell(A)/N_{1}$.
Therefore if $\ell(A)$ is sufficiently small then $k_{2} - k_{1} $ is negative and then the $0$-th degree part of $\widehat{R}_{\aff}^{\vee}(k_{2} - k_{1})$ is zero.
Hence $\Hom(\mathcal{F}^{A},\mathcal{G}^{[A]}) = 0$ for such $A$.
\end{proof}

We have some reductions.
\begin{lem}\label{lem:easy reduction of finiteness theorme}
Assume $\mathcal{F}\in \BM(\mathcal{A})$ satisfies the consequence of Lemma~\ref{lem:for finiteness} (1) (resp.\ (2)).
\begin{enumerate}
\item Any direct summand of $\mathcal{F}$ satisfies the consequence of Lemma~\ref{lem:for finiteness} (1) (resp.\ (2)).
\item For any $B\in \mathcal{S}$, $\mathcal{F}\star B$ satisfies the consequences of Lemma~\ref{lem:for finiteness} (1) (resp.\ (2)).
\end{enumerate}
\end{lem}
\begin{proof}
(1) is obvious.
For (2) we may assume $B = B_{s}$ for some $s\in S_{\aff}$.
Then it follows from Lemma~\ref{lem:rank of stalk of translation} and \ref{lem:rank of succ filtration of translation}.
\end{proof}

\begin{lem}\label{lem:reduction of finiteness theorem}
To prove Lemma~\ref{lem:for finiteness}, we may assume $\mathcal{F} = \mathcal{B}^{\mathcal{A}}(A_{\lambda}^{+})$ for some $\lambda\in \mathbb{X}^{\vee}$.
\end{lem}
\begin{proof}
We take $A\in \mathcal{A}$ such that $\mathcal{F} = \mathcal{B}^{\mathcal{A}}(A)$.
There exists $\lambda\in \mathbb{X}^{\vee}$ and $w\in W_{\aff}$ such that $A = A_{\lambda}^{+}w$ and $A_{\lambda}^{+} < A_{\lambda}^{+}s_{1} < \cdots < A_{\lambda}^{+}s_{1}\cdots s_{l} = A$ where $w = s_{1}\ldots s_{l}$ is a reduced expression~\cite[Lemma~3.6]{MR591724}.
The sheaf $\mathcal{B}^{\mathcal{A}}(A)$ is a direct summand of $\mathcal{B}^{\mathcal{A}}(A_{\lambda}^{+})\star (B_{s_{1}}\otimes\cdots \otimes B_{s_{l}})$ up to grading shift.
Hence if Lemma~\ref{lem:for finiteness} holds for $\mathcal{B}^{\mathcal{A}}(A_{\lambda}^{+})$ then it also holds for $\mathcal{B}^{\mathcal{A}}(A)$ by Lemma~\ref{lem:easy reduction of finiteness theorme}.
\end{proof}

\subsection{Some Kazhdan-Lusztig combinatorics}
We introduced the notation $\mathcal{H}$ for the Hecke algebra attached to $W_{\aff}$ (subsection~\ref{subsec:Some calculation of graded ranks}).
Let $\mathcal{H}^{\ext} = \bigoplus_{w\in W^{\ext}_{\aff}}\Z[v,v^{-1}]H_{w}$ be the Hecke algebra attached to $W_{\aff}^{\ext}$.
This contains $\mathcal{H}$ as a subalgebra.
For each $x\in W_{\aff}^{\ext}$, we have the Kazhdan-Lusztig element $\underline{H}_{x}\in \mathcal{H}^{\ext}$.
Again, we use the notation of Soergel~\cite[Theorem~2.1]{MR1444322}.
Let $\mathcal{H}_{0}$ be the Hecke algebra of $W_{\aff,0}$ and we regard this as a subalgebra of $\mathcal{H}^{\ext}$.
Let $\triv\colon \mathcal{H}_{0}\to \Z[v,v^{-1}]$ (resp.\ $\sgn\colon \mathcal{H}_{0}\to \Z[v,v^{-1}]$) be the character defined by $\triv(H_{x}) = v^{-\ell(x)}$ (resp.\ $\sgn(H_{x}) = (-v)^{\ell(x)}$).
Put ${}^{0}W_{\aff}^{\ext} = {}^{0}W_{\aff}\Omega$.
Then for $x\in {}^{0}W_{\aff}^{\ext}$, we also have elements $\underline{M}_{x}\in \triv\otimes_{\mathcal{H}_{0}}\mathcal{H}^{\ext}$ and $\underline{N}_{x}\in \sgn\otimes_{\mathcal{H}_{0}}\mathcal{H}^{\ext}$~\cite[Theorem~3.1]{MR1444322}.
We define $h_{y,x}\in\Z[v,v^{-1}]$ by $\underline{H}_{x} = \sum_{y\in W_{\aff}^{\ext}}h_{y,x}H_{y}$ for $x\in W_{\aff}^{\ext}$.
For $x,y\in {}^{0}W_{\aff}^{\ext}$, we define $m_{y,x},n_{y,x}\in\Z[v,v^{-1}]$ by $\underline{M}_{x} = \sum_{y}m_{y,x}\otimes H_{y}$ and $\underline{N}_{x} = \sum_{y}n_{y,x}\otimes H_{y}$.
We have the $p$-version of these elements ${}^{p}\underline{H}_{x} = \sum {}^{p}h_{y,x}H_{y}\in \mathcal{H}^{\ext}$, ${}^{p}\underline{M}_{x} = \sum {}^{p}m_{y,x}\otimes H_{y}\in \triv\otimes_{\mathcal{H}_{0}}\mathcal{H}^{\ext}$ and ${}^{p}\underline{N}_{x} = \sum {}^{p}n_{y,x}\otimes H_{y}\in \sgn\otimes_{\mathcal{H}_{0}}\mathcal{H}^{\ext}$, see \cite[1.5]{MR4275245}.
If $\omega\in\Omega$, then $\underline{H}_{x\omega}  = \underline{H}_{x}H_{\omega}$ and $h_{y\omega,x\omega} = h_{y,x}$.
The same formula holds for other elements.

By \eqref{eq:comparison of indecomposables}, we have $\grk \mathcal{B}^{W_{\aff}^{\ext}}(x)^{y} = v^{-\ell(x)}\grk B(x)^{y} = v^{\ell(y) - \ell(x)}({}^{p}h_{y,x})$ from the definition of ${}^{p}\underline{H}_{x}$.
We also have $\grk \mathcal{B}^{W_{\aff,0}\backslash W_{\aff}^{\ext}}(x)^{y} = v^{\ell(y) - \ell(x)}({}^{p}m_{y,x})$ for $x,y\in {}^{0}W_{\aff}^{\ext}$.

\begin{lem}\label{lem:translation by dominant element}
Let $\lambda\in \mathbb{X}^{\vee}$ be a regular dominant element.
\begin{enumerate}
\item $t_{\lambda}^{0}$ is maximal in $t_{\lambda}^{0}W_{\aff,0}$.
\item For any $x\in W_{\aff,0}$, we have $t_{\lambda}^{0}x\in {}^{0}W_{\aff}^{\ext}$.
\end{enumerate}
\end{lem}
\begin{proof}
By using the isomorphism $W_{\aff}^{\ext\prime}\simeq W_{\aff}^{\ext}$ induced by $A_{0}^{+}$, we prove the corresponding statement in $W_{\aff}^{\ext\prime}$.
Under this isomorphism, $W_{\aff,0}$ corresponds to $W_{\finite}$ and $t_{\lambda}^{0}$ corresponds to $t_{\lambda}$.
Since $\lambda$ is dominant regular, the length formula~\eqref{eq:length formula} implies $\ell(t_{\lambda}w) = \ell(t_{\lambda}) - \ell(w)$ for $w\in W_{\finite}$.
Hence, we get (1).

We prove (2).
Let $s = s_{\beta}$ be a simple reflection with respect to a simple root $\beta$.
We have
\begin{align*}
\ell(st_{\lambda}x) & = \ell(t_{s(\lambda)}sx)  = \sum_{(sx)^{-1}(\alpha) > 0,\alpha\in(\Phi')^{+}}\lvert\langle s(\lambda),\alpha\rangle\rvert + \sum_{(sx)^{-1}(\alpha) < 0,\alpha\in(\Phi')^{+}}\lvert\langle s(\lambda),\alpha\rangle - 1\rvert\\
& = \sum_{x^{-1}(\alpha) > 0,\alpha\in(\Phi')^{+},\alpha\ne\beta}\lvert\langle s(\lambda),s(\alpha)\rangle\rvert + \sum_{x^{-1}(\alpha) < 0,\alpha\in(\Phi')^{+},\alpha\ne\beta}\lvert\langle s(\lambda),s(\alpha)\rangle - 1\rvert\\
& \qquad + 
\begin{cases}
\lvert \langle s(\lambda),\beta\rangle\rvert & (x^{-1}(\beta) < 0),\\
\lvert \langle s(\lambda),\beta\rangle - 1\rvert & (x^{-1}(\beta) > 0).
\end{cases}
\end{align*}
In the second line, we replace $\alpha$ with $s(\alpha)$ and use the fact that $s_{\beta}\colon (\Phi')^{+}\setminus \{\beta\} \to  (\Phi')^{+}\setminus \{\beta\}$ is bijective.
Note that $\langle s(\lambda),\beta\rangle = -\langle \lambda,\beta\rangle < 0$.
Hence, from the above formula, it is easy to see that $\ell(st_{\lambda}x) = \ell(t_{\lambda}x) + 1$.
\end{proof}

\begin{lem}\label{lem:N does not depend on p}
For $n\in\Z_{\ge 1}$, we have ${}^{p}\underline{N}_{t^{0}_{p^{n}\rho^{\vee}}} = \sum_{z\in W_{\aff,0}}v^{\ell(z)}\otimes H_{t^{0}_{p^{n}\rho^{\vee}}z} = \underline{N}_{t^{0}_{p^{n}\rho^{\vee}}}$.
\end{lem}
\begin{proof}
The second equality is well-known.
We prove the first one.
We have ${}^pn_{y,x}\in n_{y,x} + \sum_{x' < x}\Z_{\ge 0}n_{y,x'}$~\cite[Proposition~4.2 (3)]{MR3611719} and by \cite[Corollary~6.3]{MR4437613} the coefficients of $n_{y,x}$ are non-negative.
Hence, the coefficients of ${}^pn_{y,x}$ are also non-negative.

We calculate ${}^{p}n_{y,t^{0}_{p^{n}\rho^{\vee}}}(1)$ next.
By \cite[Theorem~8.9]{MR4517647}, the polynomials ${}^{p}n_{x,y}$ are related to the multiplicity of standard modules in indecomposable tilting modules.
More precisely, we have the following.
Let $G$ be the connected reductive group over an algebraically closed field of characteristic $p$ with the root datum $(\mathbb{X}^{\vee},(\Phi')^{\vee},\mathbb{X},\Phi')$.
For dominant $\lambda\in \mathbb{X}^{\vee}$ we have the indecomposable tilting module $T(\lambda)$ and the standard module $\Delta(\lambda)$ having the highest weight $\lambda$.
Fix $\lambda\in \mathbb{X}^{\vee}$ such that $(\lambda + \rho^{\vee})/p\in \overline{A_{0}^{+}}$.
For each $w\in W_{\finite}$ and $\mu\in \mathbb{X}^{\vee}$, we put $(wt_{\mu})\cdot \lambda = w(\lambda + p\mu + \rho^{\vee}) - \rho^{\vee}$ and if $w\in W_{\aff}^{\ext}$ corresponds to $w'\in W_{\aff}^{\ext\prime}$ by the isomorphism $W_{\aff}^{\ext}\simeq W_{\aff}^{\ext\prime}$ induced by $A_{0}^{+}$, we put $w\cdot \lambda = w'\cdot \lambda$.
Let ${}^{0}W_{\aff}^{(\lambda)}$ be the set of $w\in {}^{0}W_{\aff}^{\ext}$ which is maximal in $w\Stab_{(W_{\aff}^{\ext},\cdot)}(\lambda)$.
Then for $w\in {}^{0}W_{\aff}^{(\lambda)}$
\[
[T(w\cdot \lambda)] = \sum_{y\in {}^{0}W_{\aff}^{(\lambda)}}{}^{p}n_{y,w}(1)[\Delta(y\cdot \lambda)]
\]
in the Grothendieck group.
Now we put $\lambda = -\rho^{\vee}$ and $w = t^{0}_{p^{n}\rho^{\vee}}$.
Then $\Stab_{(W_{\aff}^{\ext},\cdot)}(-\rho^{\vee}) = W_{\aff,0}$.
The module $T(w\cdot \lambda) = T((p^{n + 1} - 1)\rho^{\vee})$ is the Steinberg representation and isomorphic to $\Delta((p^{n + 1} - 1)\rho^{\vee})$.
Therefore we get
\[
{}^{p}n_{y,t^{0}_{p^{n}\rho^{\vee}}}(1)
=
\begin{cases}
1 & (y = t^{0}_{p^{n}\rho^{\vee}}),\\
0 & (y \ne t^{0}_{p^{n}\rho^{\vee}})
\end{cases}
\]
for $y\in {}^{0}W_{\aff}^{(-\rho^{\vee})}$ and, as the coefficients of ${}^{p}n_{y,x}$ is non-negative, we have
\begin{equation}\label{eq:consequence of tilging character formula}
{}^{p}n_{y,t^{0}_{p^{n}\rho^{\vee}}}
=
\begin{cases}
1 & (y = t^{0}_{p^{n}\rho^{\vee}}),\\
0 & (y \ne t^{0}_{p^{n}\rho^{\vee}})
\end{cases}
\end{equation}
for any $y\in {}^{0}W_{\aff}^{(-\rho^{\vee})}$.
(Note that we always have ${}^{p}n_{x,x} = 1$.)

We prove 
\begin{equation}\label{eq:target of calculation of N basis}
{}^{p}n_{y,t^{0}_{p^{n}\rho^{\vee}}}
=
\begin{cases}
v^{\ell(z)} & (y = t^{0}_{p^{n}\rho^{\vee}}z,\ z\in W_{\aff,0}),\\
0 & (y \notin t^{0}_{p^{n}\rho^{\vee}}W_{\aff,0})
\end{cases}
\end{equation}
for $y\in {}^{0}W_{\aff}$.
This gives the lemma.
To prove this, we first prove the following.
Let $s\in S_{\aff}\cap W_{\aff,0}$ and assume that $ys > y$.
\begin{enumerate}
\item If $ys \notin {}^{0}W_{\aff}^{\ext}$, then ${}^{p}n_{y,t^{0}_{p^{n}\rho^{\vee}}} = 0$.
\item If $ys\in {}^{0}W_{\aff}^{\ext}$, then ${}^{p}n_{y,t^{0}_{p^{n}\rho^{\vee}}} = v({}^{p}n_{ys,t^{0}_{p^{n}\rho^{\vee}}})$.
\end{enumerate}
We note
\begin{equation}\label{eq:n in terms of h}
{}^{p}n_{y,x} = \sum_{z\in W_{\aff,0}}(-v)^{\ell(z)}({}^{p}h_{zy,x}).
\end{equation}
Assume that $ys\notin {}^{0}W_{\aff}^{\ext}$.
Then there exists $t\in W_{\aff,0}\cap S_{\aff}$ such that $tys < ys$.
Since $ty > y$, by the exchange condition of Coxeter systems, we have $y = tys$.
Therefore, for $z\in W_{\aff,0}$ such that $zt > z$, we have $\ell(zys) = \ell(zty) = \ell(zt) + \ell(y) = \ell(z) + \ell(t) + \ell(y) = \ell(zy) + 1$.
By Lemma~\ref{lem:translation by dominant element}, $t_{p^{n}\rho^{\vee}}^{0}s < t_{p^{n}\rho^{\vee}}^{0}$.
Hence we have ${}^{p}h_{zys,t^{0}_{p^{n}\rho^{\vee}}} = v^{-1}({}^{p}h_{zy,t^{0}_{p^{n}\rho^{\vee}}})$~\cite[Lemma~4.3]{MR3611719}.
Therefore ${}^{p}h_{zty,t^{0}_{p^{n}\rho^{\vee}}} = v^{-1}({}^{p}h_{zy,t^{0}_{p^{n}\rho^{\vee}}})$.
By \eqref{eq:n in terms of h}, this implies 
\[
{}^{p}n_{y,t^{0}_{p^{n}\rho^{\vee}}} = \sum_{z\in W_{\aff,0},zt > z}(-v)^{\ell(z)}({}^{p}h_{zy,t^{0}_{p^{n}\rho^{\vee}}} - v\,{}^{p}h_{zty,t^{0}_{p^{n}\rho^{\vee}}}) = 0.
\]

If $ys\in {}^{0}W_{\aff}^{\ext}$, then for each $z\in W_{\aff,0}$, we have $\ell(zys) = \ell(z) + \ell(ys) = \ell(z) + \ell(y) + 1 = \ell(zy) + 1$.
Therefore $zys > zy$ and ${}^{p}h_{zy,t^{0}_{p^{n}\rho^{\vee}}} = v({}^{p}h_{zys,t^{0}_{p^{n}\rho^{\vee}}})$.
Hence ${}^{p}n_{y,t^{0}_{p^{n}\rho^{\vee}}} = v({}^{p}n_{ys,t^{0}_{p^{n}\rho^{\vee}}})$ by \eqref{eq:n in terms of h}.
We get two claims.

We prove \eqref{eq:target of calculation of N basis} by backward induction on $y$.
If $y\in {}^{0}W_{\aff}^{(-\rho^{\vee})}$, then this is \eqref{eq:consequence of tilging character formula}.
Assume that $y\notin {}^{0}W_{\aff}^{(-\rho^{\vee})}$.
Then there exists $s\in S_{\aff}$ such that $ys > y$.
Assume that $y = t^{0}_{p^{n}\rho^{\vee}}z$ where $z\in W_{\aff,0}$.
By Lemma~\ref{lem:translation by dominant element}, $ys = t^{0}_{p^{n}\rho^{\vee}}zs\in {}^{0}W_{\aff}$.
Hence by the claim (2), ${}^{p}n_{y,t^{0}_{p^{n}\rho^{\vee}}} = v({}^{p}n_{ys,t^{0}_{p^{n}\rho^{\vee}}})$.
By inductive hypothesis, we have ${}^{p}n_{ys,t^{0}_{p^{n}\rho^{\vee}}} = v^{\ell(zs)}$.
Hence ${}^{p}n_{y,t^{0}_{p^{n}\rho^{\vee}}} = v^{\ell(zs) + 1}$.
By Lemma~\ref{lem:translation by dominant element}, $\ell(y) = \ell(t^{0}_{p^{n}\rho^{\vee}}) - \ell(z)$ and $\ell(ys) = \ell(t^{0}_{p^{n}\rho^{\vee}}) - \ell(zs)$.
Therefore $\ell(ys) = \ell(y) + 1$ implies $\ell(zs) + 1 = \ell(z)$.

Assume that $y\notin t^{0}_{p^{n}\rho^{\vee}}W_{\aff,0}$.
Then $ys\notin t^{0}_{p^{n}\rho^{\vee}}W_{\aff,0}$.
Hence ${}^{p}n_{ys,t^{0}p^{n}\rho^{\vee}} = 0$.
By the claim, ${}^{p}n_{y,t^{0}p^{n}\rho^{\vee}} = 0$.
\end{proof}

Recall that $w_{0}\in W_{\aff,0}$ is the longest element.

\begin{lem}\label{lem:KL basis, steinberg}
Let $n\in\Z_{>0}$.
Then ${}^{p}\underline{M}_{t^{0}_{(p^{n} - 1)\rho^{\vee}}} = \underline{M}_{t^{0}_{(p^{n} - 1)\rho^{\vee}}}$ and ${}^{p}\underline{H}_{w_{0}t^{0}_{(p^{n} - 1)\rho^{\vee}}} = \underline{H}_{w_{0}t^{0}_{(p^{n} - 1)\rho^{\vee}}}$.
\end{lem}
\begin{proof}
By \cite[Theorem~3.5]{MR4275245}, the first one follows from the previous lemma.
By \cite[(3.2)]{MR4275245}, the second one follows from the first one.
\end{proof}

\subsection{Proof of Lemma~\ref{lem:for finiteness}(1)}\label{subsec:Proof of Lemma for finiteness (1)}
For $A,A'\in \mathcal{A}$, let $q_{A,A'}\in \Z[v,v^{-1}]$ be the generic Kazhdan-Lusztig polynomial \cite[Theorem~6.1]{MR1444322}.
Recall that, for $a = \sum_{i\in\Z}a_{i}v^{i},b = \sum_{i\in\Z}b_{i}v^{i}\in\Z[v,v^{-1}]$, we write $a\le b$ if $a_{i}\le b_{i}$ for any $i\in\Z$.

\begin{lem}
We have $\grk\mathcal{B}^{\mathcal{A}}(A_{\lambda}^{+})^{A}\le v^{\ell(A) - \ell(A_{\lambda}^{+})}q_{A,A_{\lambda}^{+}}$ for $A\in \mathcal{A}$.
\end{lem}
\begin{proof}
Let $n\in\Z_{>0}$ and put $A_{n} = A + (p^{n} - 1)\rho^{\vee} - \lambda$.
Set $O = \{A'\in \mathcal{A}\mid A'\ge A_{n}\}$.
Take $x_{n}\in W_{\aff}$ such that $A_{n} = A_{0}^{+}x_{n}$.
Define $\varphi_{+}$ by \eqref{eq:maps between alcoves}.
Then by Corollary~\ref{cor:Corollary of Lanini}, for sufficiently large $n$, $\mathcal{B}^{\mathcal{A}}(A^{+}_{(p^{n} - 1)\rho^{\vee}})|_{O\cap \mathcal{A}^{+}}$ is a direct summand of $\varphi_{+}^{*}\mathcal{B}^{W_{\aff,0}\backslash W_{\aff}}(t^{0}_{(p^{n} - 1)\rho^{\vee}})|_{O\cap \mathcal{A}^{+}}$.
Therefore, with Lemma~\ref{lem:grk of translation}, we have
\begin{align*}
\grk\mathcal{B}^{\mathcal{A}}(A_{\lambda}^{+})^{A} &= 
\grk(\mathcal{B}^{\mathcal{A}}(A_{(p^{n} - 1)\rho^{\vee}})^{A_{n}})\\
&\le \grk(\mathcal{B}^{W_{\aff,0}\backslash W_{\aff}}(t^{0}_{(p^{n} - 1)\rho^{\vee}})^{x_{n}})\\
& = v^{\ell(x_{n}) - \ell(t^{0}_{(p^{n} - 1)\rho^{\vee}})}\,{}^{p}m_{x_{n},t^{0}_{(p^{n} - 1)\rho^{\vee}}}.
\end{align*}
The right hand side is $v^{\ell(x_{n}) - \ell(t^{0}_{(p^{n} - 1)\rho^{\vee}})}\,m_{x_{n},t^{0}_{(p^{n} - 1)\rho^{\vee}}}$ by Lemma~\ref{lem:N does not depend on p}.
If $n$ is sufficiently large, then $v^{\ell(x_{n}) - \ell(t^{0}_{(p^{n} - 1)\rho^{\vee}})}\,m_{x_{n},t^{0}_{(p^{n} - 1)\rho^{\vee}}}$ converges to $v^{\ell(A) - \ell(A_{\lambda}^{+})}q_{A,A_{\lambda}^{+}}$ by the definition of $q_{A,A^{+}_{\lambda}}$.
\end{proof}

We are ready to prove Lemma~\ref{lem:for finiteness} (1).
Recall an $\mathcal{H}$-module $\widehat{\mathcal{P}}$ from subsection~\ref{subsec:Some calculation of graded ranks}.
Let $\underline{P}_{A}$ be the element in \cite[Theorem~4.3]{MR1444322}.

\begin{proof}[Proof of Lemma~\ref{lem:for finiteness} (1)]
By Lemma~\ref{lem:reduction of finiteness theorem}, we may assume $\mathcal{F} = \mathcal{B}^{\mathcal{A}}(A_{\lambda}^{+})$.
For each $\lambda\in \Z(\Phi')^{\vee}$, we define $\langle \lambda\rangle \colon \widehat{\mathcal{P}}\to \widehat{\mathcal{P}}$ by $\langle \lambda\rangle \sum_{A\in \mathcal{A}}c_{A}A = \sum_{A\in \mathcal{A}}c_{A}(A + \lambda)$.
For each $\alpha\in\Phi'$, we put $\vartheta_{\alpha} = \sum_{i = 0}^{\infty}v^{2i}\langle -i\alpha^{\vee}\rangle$.
Set $\eta = \prod_{\alpha\in (\Phi')^{+}}\vartheta_{\alpha}$.
By \cite[Theorem~6.3]{MR772611} and $\underline{P}_{A_{\lambda}^{+}} = \sum_{z\in W_{\aff,0}} v^{\ell(z)}(A_{0}^{+}z + \lambda)$, we have $\sum_{A\in \mathcal{A}}q_{A,A_{\lambda}^{+}}A = \eta\sum_{z\in W_{\aff,0}}v^{\ell(z)}(A_{0}^{+}z + \lambda)$.
Hence if the coefficient of $v^{k'}$ in $q_{A,A_{\lambda}^{+}}$ is not zero, then there exists $n_{\alpha}\in\Z_{\ge 0}$ for each $\alpha\in(\Phi')^{+}$ and $z\in W_{\aff,0}$ such that $A = A_{0}^{+}z + \lambda - \sum_{\alpha}n_{\alpha}\alpha^{\vee}$ and $k' = \ell(z) + \sum_{\alpha}2n_{\alpha}$.
We have $\ell(A_{0}^{+}z) = \ell(A_{0}^{+}) - \ell(z)$.
Hence, by Lemma~\ref{lem:length of translation}, we have $\ell(A_{\lambda}^{+}) - \ell(A) = \ell(A_{0}^{+}) - \ell(A) + 2\langle \rho,\lambda\rangle = \ell(z) + 2\sum_{\alpha}n_{\alpha}\langle \rho,\alpha^{\vee}\rangle$.
Set $N_{1} = \max(\max_{\alpha\in(\Phi')^{+}}\langle \rho,\alpha^{\vee}\rangle,1)$.
Then $k' = \ell(z) + 2\sum_{\alpha}n_{\alpha}\ge (\ell(z) + 2\sum_{\alpha}n_{\alpha}\langle \rho,\alpha^{\vee}\rangle)/N_{1} = (\ell(A_{\lambda}^{+}) - \ell(A))/N_{1}$.
Therefore, by the above lemma, if the coefficient of $v^{k}$ in $\grk\mathcal{B}^{\mathcal{A}}(A_{\lambda}^{+})^{A}$ is not zero, then $k - (\ell(A) - \ell(A_{\lambda}^{+})) \ge (\ell(A_{\lambda}^{+}) - \ell(A))/N_{1}$.
Hence $k - \ell(A)\ge (1/N_{1} - 1)\ell(A_{\lambda}^{+}) - \ell(A)/N_{1}$.
\end{proof}

\subsection{Proof of Lemma~\ref{lem:for finiteness} (2)}
Define $\varphi,\varphi_{+}$ as in \eqref{eq:maps between alcoves}.
Consider the full subcategory $\varphi^{*}\BM(W_{\aff})\subset \Sh(\mathcal{A})$.
By Lemma~\ref{lem:easy property of image of phi^*}, any objects in $\varphi^{*}\BM(W_{\aff})$ has a Verma flag.

Let $[\mathcal{S}]$ be the split Grothendieck group of $\mathcal{S}$.
Recall that we have an isomorphism $\ch_{\mathcal{S}}\colon [\mathcal{S}]\xrightarrow{\sim}\mathcal{H}$ characterized by $\ch_{\mathcal{S}}(B_{s}) = H_{s} + v$~\cite[Theorem~4.3]{MR4321542}.
We also let $[\varphi^{*}\BM(W_{\aff})]$ be the split Grothendieck group of $\varphi^{*}\BM(W_{\aff})$ and define $\ch_{\mathcal{A}}\colon [\varphi^{*}\BM(W_{\aff})]\to \widehat{\mathcal{P}}$ by $\ch_{\mathcal{A}}([\mathcal{F}]) = \sum_{A\in \mathcal{A}}v^{-\ell(A)}\grk(\mathcal{F}^{[A]})A$.
By Lemma~\ref{lem:rank of succ filtration of translation}, we have the following.
\begin{prop}
For $\mathcal{F}\in \varphi^{*}\BM(\mathcal{A})$ and $B\in \mathcal{S}$ we have $\ch_{\mathcal{A}}([\mathcal{F}\star B]) = \ch_{\mathcal{A}}([\mathcal{F}])\ch_{\mathcal{S}}([B])$.
\end{prop}

Let $\mathcal{F}\in\BM(W_{\aff})$ and $B\in \mathcal{S}$ the corresponding object.
Then $\mathcal{F} = \mathcal{B}^{W_{\aff}}(1)\star B$.
Hence $\ch_{\mathcal{A}}([\varphi^{*}\mathcal{F}]) = \ch_{\mathcal{A}}([\varphi^{*}\mathcal{B}^{W_{\aff}}(1)])\ch_{\mathcal{S}}([B])$.
We have $\ch_{\mathcal{A}}([\varphi^{*}\mathcal{B}^{W_{\aff}}(1)]) = v^{-\ell(A_{0}^{+})}A_{0}^{+}$.
Hence $\ch_{\mathcal{A}}([\varphi^{*}\mathcal{F}]) = v^{-\ell(A_{0}^{+})}A_{0}^{+}\ch_{\mathcal{S}}([B])$.

We use the Bernstein element to calculate the right-hand side.
Let $\lambda\in \mathbb{X}^{\vee}$ and take anti-dominant elements $\lambda_{1},\lambda_{2}\in \mathbb{X}^{\vee}$ such that $\lambda = \lambda_{1} - \lambda_{2}$.
Set $\theta_{\lambda} = H_{t_{\lambda_{1}}^{0}}H_{t_{\lambda_{2}}^{0}}^{-1}$.
It is known that this does not depend on $\lambda_{1},\lambda_{2}$ and $\{\theta_{\lambda}H_{w}\mid \lambda\in\mathbb{X}^{\vee},w\in W_{\aff,0}\}$ is a $\Z[v,v^{-1}]$-basis of $\mathcal{H}^{\ext}$.

\begin{lem}[{\cite[Lemma~1.8]{MR772611}}]\label{lem:lemma by Kato}
For $\lambda\in \Z\Phi'$ and $w\in W_{\aff,0}$, we have $A_{0}^{+}\theta_{\lambda}H_{w} = A_{0}^{+}t^{0}_{\lambda}w$ in $\widehat{\mathcal{P}}$.
\end{lem}

\begin{proof}[Proof of Lemma~\ref{lem:for finiteness} (2)]
For $n\in\Z_{>0}$ we put $A_{n} = A + (p^{n} - 1)\rho^{\vee} - \lambda$.
Let $\pi_{M}\colon W_{\aff}\to W_{\aff,0}\backslash W_{\aff}$ be the natural projection.
Set $O = \{A'\in \mathcal{A}\mid A'\ge A_{n}\}$.
Then, by Corollary~\ref{cor:Corollary of Lanini}, if $n$ is sufficiently large, $\mathcal{B}^{\mathcal{A}}(A_{(p^{n} - 1)\rho^{\vee}}^{+})|_{O\cap \mathcal{A}^{+}}$ is a direct summand of $\varphi_{+}^{*}\mathcal{B}^{W_{\aff,0}\backslash W_{\aff}}(t^{0}_{(p^{n} - 1)\rho^{\vee}})|_{O\cap \mathcal{A}^{+}}$.
Let $i\colon \mathcal{A}^{+}\hookrightarrow \mathcal{A}$ be the inclusion map.
Then we have $\varphi_{+} = \pi_{M}\circ\varphi\circ i$.
Hence $\varphi_{+}^{*}\mathcal{B}^{W_{\aff,0}\backslash W_{\aff}}(t^{0}_{(p^{n} - 1)\rho^{\vee}})|_{O\cap \mathcal{A}^{+}}\simeq \varphi^{*}\pi_{M}^{*}\mathcal{B}^{W_{\aff,0}\backslash W_{\aff}}(t^{0}_{(p^{n} - 1)\rho^{\vee}})|_{O\cap \mathcal{A}^{+}}$.
Set $\mathcal{F} = \varphi^{*}\pi_{M}^{*}\mathcal{B}^{W_{\aff,0}\backslash W_{\aff}}(t^{0}_{(p^{n} - 1)\rho^{\vee}})$.
By Lemma~\ref{lem:comparison of F^[A]}, we have $\mathcal{F}^{[A_{n}]} = (\mathcal{F}|_{O\cap \mathcal{A}^{+}})^{[A_{n}]}$.

Therefore the graded rank of $\mathcal{B}^{\mathcal{A}}(A_{\lambda}^{+})^{[A]}$, which is equal to $\grk\mathcal{B}^{\mathcal{A}}(A^{+}_{(p^{n} - 1)\rho^{\vee}})^{[A_{n}]}$ by Lemma~\ref{lem:grk of translation}, is less than or equal to the graded rank of $\mathcal{F}^{[A_{n}]}$.
We have $\mathcal{F}\simeq \varphi^{*}\mathcal{B}^{W_{\aff}}(w_{0}t_{(p^{n} - 1)\rho^{\vee}}^{0})$ (Lemma~\ref{lem:pull-back for parabolic Soergel bimodules}).
Hence $\varphi^{*}\mathcal{F}\simeq (\varphi^{*}\mathcal{B}^{W_{\aff}}(1))\star \Gamma(\mathcal{B}^{W_{\aff}}(w_{0}t_{(p^{n} - 1)\rho^{\vee}}^{0}))$.
Hence $\ch_{\mathcal{A}}(\varphi^{*}\mathcal{F}) = A_{0}^{+}\ch_{\mathcal{S}}(\Gamma(\mathcal{B}^{W_{\aff}}(w_{0}t_{(p^{n} - 1)\rho^{\vee}}^{0})))$.
By \eqref{eq:comparison of indecomposables} and the definition of ${}^{p}\underline{H}_{w_{0}t_{(p^{n} - 1)\rho^{\vee}}^{0}}$, we have
\begin{align*}
\ch_{\mathcal{S}}(\Gamma(\mathcal{B}^{W_{\aff}}(w_{0}t_{(p^{n} - 1)\rho^{\vee}}^{0}))) & = v^{-\ell(w_{0}t_{(p^{n} - 1)\rho^{\vee}})}\ch_{\mathcal{S}}(B(w_{0}t_{(p^{n} - 1)\rho^{\vee}}))\\
& = v^{-\ell(w_{0}t_{(p^{n} - 1)\rho^{\vee}})}\underline{H}_{w_{0}t_{(p^{n} - 1)\rho^{\vee}})}
\end{align*}
Let $c\in\Z[v,v^{-1}]$ be the coefficient of $A_{n}$ in $A_{0}^{+}\underline{H}_{w_{0}t^{0}_{(p^{n} - 1)\rho^{\vee}}}$.
Then, by the above calculation and the definition of $\ch_{\mathcal{A}}$, we have $\grk \mathcal{B}^{\mathcal{A}}(A_{(p^{n} - 1)\rho^{\vee}})^{[A_{n}]}\le v^{\ell(A_{n}) - \ell(w_{0}t^{0}_{(p^{n} - 1)\rho^{\vee}})}c$.
To prove the lemma it is sufficient to prove that there exists $C > 0$ depending only on our root system such that if the coefficient of $v^{k}$ in $v^{-\ell(A) + \ell(A_{n}) - \ell(w_{0}t^{0}_{(p^{n} - 1)\rho^{\vee}})}c$ is not zero then $\lvert k\rvert \le C$.

We calculate $-\ell(A) + \ell(A_{n}) - \ell(w_{0}t^{0}_{(p^{n} - 1)\rho^{\vee}})$.
We have $\ell(A) - \ell(A_{\lambda}^{+}) = \ell(A_{n}) - \ell(A^{+}_{(p^{n} - 1)\rho^{\vee}})$ and $\ell(A^{+}_{(p^{n} - 1)\rho^{\vee}}) = \ell(A_{0}^{+}) + 2\langle\rho,(p^{n} - 1)\rho^{\vee}\rangle$.
We also have $\ell(w_{0}t^{0}_{(p^{n} - 1)\rho^{\vee}}) = \ell(w_{0}) + \ell(t^{0}_{(p^{n} - 1)\rho^{\vee}})$ as $t^{0}_{(p^{n} - 1)\rho^{\vee}}\in {}^{0}W_{\aff}^{\ext}$ (Lemma~\ref{lem:translation by dominant element}).
By the length formula \eqref{eq:length formula}, we have $\ell(t^{0}_{(p^{n} - 1)\rho^{\vee}}) = 2\langle\rho,(p^{n} - 1)\rho^{\vee}\rangle$. 
Hence $-\ell(A) + \ell(A_{n}) - \ell(w_{0}t^{0}_{(p^{n} - 1)\rho^{\vee}}) = \ell(A_{0}^{+}) - \ell(A_{\lambda}^{+}) - \ell(w_{0})$.
Since this does not depend on $A$, it is sufficient to prove the existence of $C$ such that if the coefficient of $v^{k}$ in $c$ is not zero, then $\lvert k\rvert\le C$.

By Lemma~\ref{lem:lemma by Kato}, $c$ is a coefficient in $\underline{H}_{w_{0}t^{0}_{(p^{n} - 1)\rho^{\vee}}}$ with the basis by Bernstein.
Such an expansion is calculated in \cite[Proposition 8.6]{MR737932}.
Namely, let $G$ be a connected reductive group over $\mathbb{C}$ with the root datum $(\mathbb{X}^{\vee},(\Phi')^{\vee},\mathbb{X},\Phi')$, $V$ the irreducible representation of $G$ with the highest weight $(p^{n} - 1)\rho^{\vee}$.
For each $\mu\in\mathbb{X}^{\vee}$, let $V_{\mu}$ be the $\mu$-weight space of $V$.
Then, by \cite[Proposition 8.6]{MR737932}, we have 
\[
\underline{H}_{w_{0}t^{0}_{(p^{n} - 1)\rho^{\vee}}} = \sum_{\mu\in\mathbb{X}^{\vee},w\in W_{\aff,0}}(\dim V_{\mu})v^{\ell(w_{0}) - \ell(w)}\theta(\mu)H_{w}.
\]
Hence we can take $C = \ell(w_{0})$.
\end{proof}

\end{document}